\documentclass[a4paper,notitlepage]{article}
\usepackage{fullpage}
\usepackage[colorlinks=true,linkcolor=blue,unicode, psdextra]{hyperref}
\usepackage{color}
\usepackage{amsthm}
\usepackage{amsmath}
\usepackage{amssymb}
\usepackage{enumitem}
\usepackage{mathtools}
\usepackage{pgfplots}
\pgfplotsset{compat=1.17}

\input{xypic}
\usepackage[all]{xy}

\usepackage{mathrsfs}
\newcommand\similarrightarrow{\stackrel{\sim}{\smash{\longrightarrow}\rule{0pt}{0.4ex}}}

\DeclareMathOperator{\Der}{Der}

\DeclareMathOperator{\End}{End}

\DeclareMathOperator{\gr}{gr}
\DeclareMathOperator{\Hoch}{HH}
\DeclareMathOperator{\Hom}{Hom}
\DeclareMathOperator{\id}{Id}

\DeclareMathOperator{\res}{res}

\DeclareMathOperator{\SL}{SL}

\DeclareMathOperator{\Spec}{Spec}

\DeclareMathOperator{\Sym}{Sym}

\DeclareMathOperator{\tr}{tr}
\DeclareMathOperator{\zhu}{Zhu}
\DeclareMathOperator{\hh}{HH}

\newcommand{\Om}{\Omega}

\newcommand{\D}{\Delta}

\newcommand{\Z}{\mathbb{Z}}

\newcommand{\CD}{\mathcal{D}}

\newcommand{\HH}{\mathbb{H}}

\newcommand{\OO}{\mathcal{O}}

\newcommand\cA{\mathscr{A}}

\newcommand\cE{\mathscr{E}}
\newcommand\cF{\mathscr{F}}
\def\cH{\mathscr{H}}

\newcommand\cM{\mathscr{M}}
\newcommand\cO{\mathscr{O}}

\newcommand\bM{\mathbb{M}}
\newcommand\bJ{\mathbb{J}}
\newcommand{\HP}{\operatorname{HP}}
\newcommand{\HK}{\operatorname{HK}}
\newcommand\im{\Im}

\newcommand{\zuta}{\underline{\zeta}}

\newcommand{\vac}{\mathbf{1}}
\newcommand{\vir}{\text{Vir}}
\newcommand\bQM{\mathbb{QM}}
\newcommand{\mring}{\mathbb{M}_*}
\newcommand{\qmring}{\mathbb{QM}_*}
\newcommand{\g}{\mathfrak{g}}
\newcommand{\h}{\mathfrak{h}}
\newcommand{\ZZZ}{\mathfrak{Z}}
\def\ch{\text{ch}}
\def\owp{\overline{\wp}}

\newcommand{\HchiralV}{H_1^{\text{ch}}(V)}

\def\cL{\mathscr{L}}

\swapnumbers
\newtheoremstyle{exps}{\topsep}{\topsep}{}{0pt}{\bfseries}{.}{0pt}{}

\newtheorem*{thm*}{Theorem}
\newtheorem*{prop*}{Proposition}
\newtheorem*{lem*}{Lemma}
\newtheorem*{cor*}{Corollary}
\newtheorem*{rem*}{Remark}
\newtheorem{thm}{Theorem}[section]
\newtheorem{prop}[thm]{Proposition}
\newtheorem{lem}[thm]{Lemma}
\newtheorem{ex}[thm]{Example}
\newtheorem{cor}[thm]{Corollary}

\theoremstyle{definition}
\newtheorem*{defn*}{Definition}
\newtheorem*{exer*}{Exercise}

\newtheorem{defn}[thm]{Definition}
\newtheorem*{problem*}{Problem}
\newtheorem{rem}[thm]{Remark}
\newtheorem{nolabel}[thm]{ }

\theoremstyle{exps}

\allowdisplaybreaks

\numberwithin{equation}{section}
\setenumerate[1]{label={\alph*)}} 

\makeatletter
\def\thmhead@plain#1#2#3{%
\thmname{#1}\thmnumber{\@ifnotempty{#1}{ }\@upn{#2}}%
\thmnote{ {\the\thm@notefont#3}}}
\let\thmhead\thmhead@plain
\def\swappedhead#1#2#3{%
\thmnumber{#2}%
\thmname{\@ifnotempty{#2}{~}#1}%
\thmnote{ {\the\thm@notefont#3}}}
\let\swappedhead@plain=\swappedhead
\makeatother

\makeatletter
\def\th@definition{
\thm@notefont{}%
\normalfont
}
\makeatother

\date{}
\title{The First Chiral Homology Group}
\author{Jethro van Ekeren \thanks{Instituto de Matem\'{a}tica e Estat\'{i}stica, Universidade Federal Fluminense, Niter\'{o}i RJ, Brazil. \newline \indent \indent \footnotesize{\texttt{\href{mailto:jethrovanekeren@gmail.com}{jethrovanekeren@gmail.com}}}}  \and  Reimundo Heluani \thanks{Instituto de Matem\'{a}tica Pura e Aplicada, Rio de Janeiro, RJ, Brazil \newline \indent \indent \footnotesize{\texttt{\href{mailto:heluani@potuz.net}{heluani@potuz.net}}}}}

\begin{document}
\maketitle
\begin{abstract}We study the first chiral homology group of elliptic curves with coefficients in
vacuum insertions of a conformal vertex algebra $V$. We find finiteness conditions
on $V$ guaranteeing that these homologies are finite dimensional, generalizing
the $C_2$-cofinite, or quasi-lisse condition in the degree $0$ case. We determine explicitly
the flat connections that these homologies acquire under smooth variation of the
elliptic curve, as insertions of the conformal vector and the Weierstrass
$\zeta$ function. We
construct linear functionals associated to self-extensions of $V$-modules and prove
their convergence under said finiteness conditions. These linear functionals turn out to be degree
$1$ analogs of the $n$-point functions in the degree $0$ case. As a corollary we
prove the vanishing of the first chiral homology group of an elliptic curve with
values in several rational vertex algebras, including affine $\mathfrak{sl}_2$
at non-negative integral level, the $(2,2k+1)$-minimal models and arbitrary simple affine vertex algebras
at level $1$. Of independent interest, we prove a
Fourier space version of the Borcherds formula. 
\end{abstract}
\section{Introduction}
\begin{nolabel}Let $k$ be a field, $A$ an associative $k$-algebra and $\rho_M:A \rightarrow \End_k(M)$ a finite dimensional $A$-module. The linear functional 
\begin{equation}
\varphi^0_M: A \rightarrow k, \quad a \mapsto \tr_M \rho_M(a),
\label{eq:hh0-1}
\end{equation}
defines a class in the dual of the zeroth Hochschild homology of $A$: 
\[  \hh_0(A)^* = \left( A/[A,A] \right)^*.\]
If $A$ is finite dimensional and semisimple over $k$, then every class in
$\hh_0(A)^*$ arises as a linear combination of such linear functionals. In fact, as $M$ runs over a set of representatives of the isomorphism classes of irreducible $A$-modules, the $\varphi^0_M$ provide a basis of $\hh_0(A)^*$. 
\label{no:zerohh}
\end{nolabel}
\begin{nolabel} Let 
\begin{equation} \label{eq:self-ext} 0 \rightarrow M \rightarrow E \rightarrow M \rightarrow 0, \end{equation}
be a self-extension of finite dimensional $A$-modules. We choose a $k$-linear splitting of \eqref{eq:self-ext} so that the representation $\rho_E: A \rightarrow \End_k(E)$ may be written in block matrix form as 
\[ \rho_E(a) = 
\begin{pmatrix}
\rho_M(a) & \sigma(a) \\ 
0 & \rho_M(a) 
\end{pmatrix}
\]
where $\sigma: A \rightarrow \End_k(M)$ is an $A$-derivation. 
The linear functional 
\[ \varphi^1_{E}: A \otimes A \rightarrow k, \quad a \otimes b \mapsto  \tr_M \sigma(a) \rho_M(b), \]
defines a class in the dual $\hh_1(A)^*$ of the first Hochschild homology group. This class is independent of the choice of splitting of \eqref{eq:self-ext}. 
\label{no:onehh}
\end{nolabel}
\begin{nolabel}In this article we study a \emph{chiral} or \emph{vertex} analog
of the construction in \ref{no:zerohh}--\ref{no:onehh} with the
\emph{genus $1$ chiral homology} of Beilinson and Drinfeld
\cite{beilinsondrinfeld} in place of {Hochschild homology} above. The degree zero case is due to Zhu \cite{zhu}. Let $V$ be a conformal vertex algebra of central
charge $c$ and $M$ an admissible $V$-module (that is, a positive energy $V$-module which decomposes into finite
dimensional Jordan blocks for the action of $L_0$).
If $V$ is $C_2$-cofinite, or more generally quasi-lisse (see
\cite{arakawa-kawasetsu}), then the linear functional
\begin{equation} \varphi^0_M: V \rightarrow k, \quad a \mapsto \tr_M Y^M \left( e^{2 \pi i z L_0}a,e^{2 \pi i z} \right)
e^{2 \pi i \tau (L_0 - c/24)}
\label{eq:1.exp1} \end{equation}
converges to a holomorphic function of $\tau \in \mathbb{H}$ the upper half
complex plane. It defines a class of the dual of the zeroth \emph{chiral
homology group of $V$ on the universal elliptic curve}, also known as the \emph{genus $1$ conformal block of $V$}. 
If in addition $V$ is rational then the classes $\varphi^0_M$, where $M$ runs over a set of representatives of the isomorphism classes of irreducible $V$-modules, constitute a basis for the space of conformal blocks in genus $1$. In this case, the space of conformal blocks is
finite dimensional. This is the chiral analog of \ref{no:zerohh}. 
\label{no:zhufirst}
\end{nolabel}
\begin{nolabel}Let $V$ be a conformal vertex algebra, $\tau \in \mathbb{H}$
the upper-half plane, $q = e^{2 \pi i \tau}$, and $E_\tau$ the complex elliptic curve defined as the quotient of $\mathbb C$ by the lattice $\Z + \Z \tau$ or equivalently as the quotient of $\mathbb{C}^*$ by the action
of $\mathbb{Z}$ given by $n \mapsto q^n$. In \cite{beilinsondrinfeld} Beilinson and Drinfeld define a complex of vector spaces
$C_\bullet(E_\tau, V)$ whose homology $H_\bullet^{\text{ch}}(E_\tau,V)$ is known as the \emph{chiral homology} of
$E_\tau$ with coefficients in $V$. As $E_\tau$ varies in the moduli space of
elliptic curves, the spaces $H^{\text{ch}}_k(E_\tau, V)^*$ define bundles with flat
connections. For $k=0$ the space of flat sections of this bundle coincides with
the space of conformal blocks in genus $1$ as defined by Zhu. In this article we
detail the construction of the $k=1$ case, leading to the notion of \emph{degree
$1$ conformal block} of $V$ in genus $1$. 
\label{no:somenolabel}
\end{nolabel}
\begin{nolabel}
Let $M$ be an admissible $V$-module and
\[
0 \rightarrow M \rightarrow E \rightarrow M \rightarrow 0
\]
a self-extension of $M$. Upon choosing a splitting as vector spaces, we may write the action $Y^E(\cdot, z)$ of $V$ on $E$ in block matrix form as 
\[ Y^E(a, z) = 
\begin{pmatrix}
Y^M(a,z) & \psi(a,z) \\ 
0 & Y^M(a,z)
\end{pmatrix}
.\]
The vertex analog of the construction described in \ref{no:onehh} is as follows.
\begin{thm*}[\ref{thm:convergence1}]
Let $\cF$ be the space of meromorphic functions $f(z)$ which are bi-periodic
with possible poles at $\mathbb{Z} + \tau \mathbb{Z}$.  
Define the following linear
functional on $V \otimes V \otimes \cF$
\begin{equation}
a \otimes b \otimes f \mapsto
\frac{1}{2 \pi i }\int_C dz f\left(e^{2
\pi i (z-w)}\right)  \tr_M
\psi\left( e^{2 \pi i z L_0} a, e^{2 \pi i z} \right) Y^M\left( e^{2 \pi i w L_0}b,e^{2 \pi i w} \right) e^{2 \pi i \tau (L_0 - c/24)}, 
\label{eq:1integral}
\end{equation}
where $C$ is a small contour of $z$ around $0$ not containing the point $w$.
Then if $\res_t f(t) Y[a,t]b = 0$ and under certain finiteness conditions on $V$
(described below), this linear functional converges to a holomorphic function of $\tau$ in the upper half plane (it is
independent of $w$), and it determines a flat section of the dual of the first chiral homology
group of $V$ in genus $1$. This class in independent of the chosen splitting of
$E$.  
\end{thm*}
Notice that in contrast to the situation of \ref{no:onehh}, our linear
functional takes as input an elliptic function $f$ in addition to two elements
in the algebra $V$. 
\label{no:1.deg1-gen}
\end{nolabel}
\begin{nolabel} We now describe the \emph{finiteness conditions} alluded to in Theorem \ref{thm:convergence1}. The role they play in connection with the first chiral cohomology group of $V$ in genus $1$ is analogous to that played in degree zero by the $C_2$--cofiniteness condition or more generally by the quasi-lisse
condition. In order to describe them, we recall from \cite{li05} that any
vertex algebra carries a
 decreasing filtration
$\{F_p V\}$ known as the \emph{Li filtration}. The associated graded
$A = \gr_F V$ with respect to this
filtration is a $\mathbb{Z}_+$--graded commutative algebra with a derivation
of degree $1$ which is generated (as a differential algebra) by its degree $0$
component $R_V = A^0$. The spectrum of the commutative algebra $A$ is known as the
\emph{singular support of $V$} \cite{arakawa}. The subalgebra $R_V$ in fact carries the natural structure of a
Poisson algebra, and its spectrum is known as the \emph{associated scheme of $V$}. 

The vertex algebra $V$ is said to be $C_2$-cofinite (or lisse) if $\dim R_V < \infty$, and is said to be quasi-lisse if its associated variety consists of a finite number of symplectic leaves. If $V$ is quasi-lisse then in particular the zeroth Poisson homology of its associated scheme is finite dimensional, that is,
\begin{equation} \label{eq:1hp0} 
\mathrm{HP}_0(R_V) = 
\frac{R_V}{\left\{ R_V, R_V \right\}},
\end{equation}
is finite dimensional. (We briefly recall the construction of Poisson homology of Poisson algebras in Section \ref{sec:homology}.)

It follows from the work of Zhu \cite{zhu} that for a $C_2$-cofinite conformal vertex algebra $V$ the zeroth chiral homology of $E_\tau$ with coefficients in $V$ is finite dimensional. In fact, as was observed in \cite{arakawa-kawasetsu}, it is sufficient that $V$ be quasi-lisse or more generally that $\dim \mathrm{HP}_0(R_V) < \infty$. And indeed we have
\[
\dim H^{\text{ch}}_0 (E_\tau, V) \leq \dim \mathrm{HP}_0(R_V).
\]
See Proposition \ref{prop:H0.fin.gen} below. 
In this work we prove finite dimensionality of the first chiral homology group of $E_\tau$ under two finiteness hypotheses on $V$, the first being a natural degree $1$ analog of \eqref{eq:1hp0}, namely
\begin{equation} \label{eq:1hp1}
\dim \mathrm{HP}_1(R_V) < \infty.
\end{equation}
To describe the second finiteness hypothesis we recall $J R_V$ the $\mathbb Z_+$-graded commutative 
algebra with a derivation of degree $1$, freely generated as a differential
algebra by $R_V$. Its spectrum is known as the \emph{arc space} of the
associated scheme $\Spec R_V$. By the universal property of the arc space construction, there is a surjection of
differential algebras 
\begin{equation} \label{eq:1surjection}
J R_V \twoheadrightarrow A
\end{equation}
where, we recall, $A = \gr_F V$. The kernel of the surjection \eqref{eq:1surjection} is a differential ideal of $J R_V$, and we say that the vertex algebra $V$ is \emph{classically free} if \eqref{eq:1surjection} is an isomorphism. We may now state our result on finiteness of the first chiral homology group.
\begin{thm*}[\ref{thm:convergence2}] Let $E_\tau$ be an elliptic curve and $V$ a vertex algebra that satisfies (1) the condition \eqref{eq:1hp1} and (2) that
the kernel of the surjection \eqref{eq:1surjection} is finitely generated as a differential
ideal. Then $\dim H^{\text{ch}}_1(E_\tau, V) < \infty$. 
\end{thm*}
\label{no:1.45}
\end{nolabel}
\begin{nolabel}Here is another description of the condition on the morphism
\eqref{eq:1surjection}. The algebra $A$ comes equipped with a derivation of
degree $1$. This derivation induces a morphism of $A$-modules $d: \Omega^1_A
\rightarrow A$, where $\Omega^1_A$ is the module of
K\"ahler differentials. This morphism can be extended by the Leibniz rule to a
differential $d$ of degree $-1$
of the free (super) commutative algebra 
$K_\bullet(A) = \wedge^\bullet \Omega^1_A$. This is the \emph{Koszul complex} of the differential algebra
$A$.  We let $\HK_i(A)$ be the $i^{\text{th}}$ homology of the complex $K_\bullet(A)$. It was shown in \cite{eh2018} that $\HK_1(A) = 0$ if and only if $V$ is classically free, and in Proposition \ref{prop:Koszul.finiteness} below that $\dim \HK_{1}(A) < \infty$ if and only if the kernel of the morphism \eqref{eq:1surjection} is finitely generated as a differential ideal.
\label{no:1kosz} 
\end{nolabel}
\begin{nolabel} Stronger results can be obtained for $V$ classically free. In \cite{eh2018} a relationship was established between the \emph{nodal curve} limit $\tau \rightarrow i\infty$ of the chiral homology of $E_\tau$, and the Hochschild homology of the Zhu algebra $\zhu(V)$ of $V$. In the case that $V$ is classically free this takes the particularly simple form
\[ \lim_{\tau \rightarrow i \infty} H_1^{\text{ch}}(E_\tau, V) \cong \Hoch_1(\zhu(V)). \]
See Proposition \ref{prop:eh2018.summary} below. In Section \ref{sec:series.exp} we exploit this result to obtain 
\begin{thm*}[\ref{thm:chiral.H1.vanishing}] Let $V$ be a classically free vertex algebra and $E_\tau$ a smooth elliptic curve. If $\zhu(V)$ is semisimple, in particular if $V$ is rational, then $H^{\text{ch}}_1(E_\tau, V) = 0$.
\end{thm*}
\label{no:classically-free-rational}
\end{nolabel}
\begin{cor*}[\ref{cor:examples}]
The chiral homology $H_1^{\text{ch}}(E_\tau, V)$ vanishes in the following cases:
\begin{itemize}
\item[(a)] $V = \vir_{2, 2s+1}$ is a boundary Virasoro minimal model,

\item[(b)] $V = V_k(\mathfrak{sl}_2)$ is the simple affine vertex algebra at level $k \in \Z_+$,

\item[(c)] $V = V_1(\mathfrak{g})$ is the simple affine vertex algebra at level $1$, for $\mathfrak{g}$ a simple Lie algebra.
\end{itemize}
\end{cor*}
\begin{nolabel}
In order to prove convergence of the linear functional \eqref{eq:1.exp1} Zhu showed that the right hand side, regarded as a formal power series, satisfies an ordinary differential equation in $q = e^{2 \pi i \tau}$ with a regular singularity at $q=0$. He then used results from the classical theory of ordinary differential equations, in particular the Frobenius-Fuchs theory, to guarantee the convergence of such a formal solution. 

We follow the
same technique in this article. We write down an explicit ordinary differential
equation that our functionals \eqref{eq:1integral} satisfy (see
\eqref{eq:flatness-0}). This differential equation arises as the connection
that the bundles of chiral homology (or their duals) acquire as $E_\tau$ moves in
the moduli space of elliptic curves. We then show that this differential equation has a regular singularity at $q=0$.
\label{no:connection1}
\end{nolabel}
\begin{nolabel} In \cite{zhu}, the space of conformal blocks is defined as a system of compatible \emph{$n$-point functions}, the trace functional \eqref{eq:1.exp1} being a special case of such for $n=1$. It is proven that one can express an $n+1$ point
function as a linear combination of $n$-point functions. Convergence of the
$n=1$ linear functionals and this ``insertion formula''  guarantees convergence
of the $n$-point functions. Coversely, a suitable
integral of an $n+1$ point function recovers an $n$-point function. 
Concretely, the $n$-point functions are given by
\[ a_1 \otimes \dots \otimes a_n \mapsto \tr_M Y^M\left(e^{2 \pi i z_1 L_0} a_1,e^{2 \pi i z_1}\right)\dots Y^M \left( e^{2 \pi i z_n L_0} a_n, e^{2 \pi i z_n} \right)
q^{L_0 - c/24}. \]

The following algebro-geometric counterpart was given in \cite{beilinsondrinfeld}. For each collection of points $t_1,\dots,t_n$ on the elliptic curve
$E_\tau$ the authors define a complex whose homology, denoted here
$H_\bullet^{ch}(V^{\otimes n},t_1,\dots,t_n)$ represents \emph{chiral
homology of $E_\tau$ with supports on the modules $V$ inserted at the points
$t_i$}. The authors moreover prove \cite[(4.4.2.1)]{beilinsondrinfeld} that these
homologies are isomorphic for all $n$. These statements remain true for arbitrary (not necessarily elliptic) curves and collections of points. In the
degree $0$ case this was carried out explicitly in \cite{frenkelzvi}, recovering
the result of Zhu in genus $1$. 

We here define explicit complexes computing the homologies
$H_\bullet^{ch}(V^{\otimes n}, t_1,\dots,t_n)$ and construct natural classes, the analog of the
$n$-point functions, associated to self extensions of modules as in
\ref{no:1.deg1-gen}. The $n=1$ version of this construction is just \eqref{eq:1integral}. These classes are linear functionals 
\[ V^{\otimes n +1} \otimes \cF_{n+1} \rightarrow \cF_{n}, \]
where $\cF_{n}$ denotes the space of meromorphic functions $f(t_1,\dots,t_n)$,
bi-periodic with possible poles at the diagonals $t_i = t_j$ modulo $\mathbb{Z}
+ \tau \mathbb{Z}$. We show that $n+1$ point functions can be described in terms
of $n$ point functions (see Proposition \ref{prop:8.15}), thus guaranteeing their
convergence. Similarly we show that
an integral of an $n+1$-point function is given by $n$-point functions (see
\eqref{eq:integral}). Concretely these linear functionals are given by
\[ 
\begin{multlined}
a_0 \otimes \dots \otimes a_n \otimes f(z_0,\dots,z_n) \mapsto 
\frac{1}{2 \pi i} \int_C dz_0 f\left( e^{2 \pi i z_0}, \dots, e^{2 \pi i z_n}
\right) \tr_M \psi \left( e^{2 \pi i z_0 L_0} a_0,e^{2 \pi i z_0}\right) \times \\ \times  Y^M\left(e^{2 \pi i z_1 L_0 a_1},e^{2 \pi i z_1} \right)\dots Y^M\left(e^{2 \pi i z_n L_0}a_n,e^{2 \pi i z_n} \right) q^{L_0 - c/24}, 
\end{multlined}
\]
where $C$ is a small contour around $0$ not containing any of the points
$z_1,\dots,z_n$. The main result is that this formal power series converges uniformly and its limit is an elliptic function in $\cF_n$ (see Theorem
\ref{thm:convergence3}).
\label{no:general-n} 
\end{nolabel}
\begin{nolabel} 
Instrumental in Zhu's construction is the identification of $E_\tau$ with the
quotient $\mathbb{C}/ (\mathbb{Z} + \tau \mathbb{Z})$ via the exponential map. Indeed the relationship between conformal blocks as defined by correation functions, and as chiral homology, requires the introduction of an auxiliary vertex algebra structure associated with the exponential change of coordinates. For a conformal vertex algebra with state--field correspondence $Y(\cdot, z)$, Zhu constructs an isomorphic vertex algebra structure on $V$ with state-field correspondence $Y[\cdot, z]$, defined by
\begin{equation}
 Y[a,z] = Y \left( e^{2 \pi i L_0} a, e^{2 \pi i z} - 1
\right). 
\label{eq:1.exp2} 
\end{equation}
This vertex algebra structure has already appeared above in the
statement of Theorem \ref{thm:convergence1}. 

A new feature that arises in our study of chiral homology in degree greater than $0$, is the appearance of elliptic functions in the description of chains and not only of the differentials. This forces us to study these functions and their Fourier expansions. Recall that for a vertex algebra $V$ and a rational function $f(s,t)$ with possible poles at $s=0$, $t=0$ and $s=t$ the following Borcherds identity holds
\begin{multline*}
\res_s \res_t Y(a,s) Y(b,t)c\, i_{s,t} f(s,t) - \res_t \res_s Y(b,t) Y(a,s) c\, i_{t,s} f(t,s) = \\ 
\res_{t} Y\left( \res_{s-t} i_{t,s-t} f(s,t) Y(a,s-t)b, t \right)c
\end{multline*}
where $i_{s,t}: \mathbb{C}(s,t) \rightarrow \mathbb{C}( (s))( (t))$ is the
canonical embedding and $\res_s: V ( (s)) \rightarrow V$ is the \emph{residue map} (coefficient of $s^{-1}$). When dealing with the exponentiated coordinates $z = e^{2\pi i s}, w = e^{2 \pi i t}$ it is useful to use the following modified vertex operators 
\[ X(a,z)  = z^{-1} Y( z^{L_0} a, z). \]
Let $f(t)$ be a meromorphic function of a complex variable $t$ which is periodic, i.e., $f(t+1) = f(t)$. Let $\tau \in \mathbb{H}$ and suppose $f$ is analytic in the domain $-\im \tau < \im t < \im \tau$ except possibly at $t=0$ where it may have a pole. Put $q = e^{2 \pi i \tau}$ and let $F_+(z)$ (resp. $F_-(z)$ ) be the Fourier expansion of $f$, in the domain $|q|<|z|<1$ (resp. $1 < |z| < |q|^{-1}$). If $f$ is \emph{quasi-periodic} with period $\tau$, that is,
 $f$ satisfies $f(t+\tau) = f(t) - \alpha$ for all $t \in \mathbb{C}$ for some constant $\alpha \in \mathbb{C}$, then we prove
\begin{thm*}[\ref{prop:borcherds-modified-fourier}]
The following identity holds
\begin{equation}\label{eq:borcherds-fourier}
2\pi i X(a_{[f]} b, w)  = \res_z F_-\left( \frac{z}{w} \right) X(a,z) X(b,w) - \res_z F_+\left( \frac{z}{w} \right) X(b,w) X(a,z)
\end{equation}
as an equality in $\Hom(V,V ((w)))$, where 
\[ a_{[f]}b = \res_t f(t) Y[a,t]b.\]
\end{thm*}
Notice that both vertex algebra structures $Y(a,z)$ and $Y[a,z]$ of $V$ appear in \eqref{eq:borcherds-fourier}, as well as both $f$ and its Fourier expansions on different domains appear. 
\label{no:borcherds-intro}
\end{nolabel}
\begin{nolabel}
This article is a self-contained introduction to chiral homology in degree $1$ for vertex algebras, aimed at the representation theorist. No background on algebraic geometry nor chiral algebras is required. It is our aim to provide a bridge between the abstract world of chiral homology on arbitrary curves and the linear algebra world in the presence of coordinates. The translation between these two was carried out in \cite{eh2018}, in this article we focus on techniques to compute these homologies. In doing so, we have had to combine tools from different areas of mathematics, the analyical proof of the Borcherds identity in Fourier coordinates \eqref{eq:borcherds-fourier}, Zhu algebras and their Hochschild homology, the Li filtration and the higher Poisson homology of the associated scheme, arc spaces and classical-freeness of vertex algebras. It would be interesting to see if these techniques have geometrical counterparts. Some of these tools make sense immediately in the geometric side, for example, Li's \emph{increasing standard filtration} makes sense as a filtration of chiral algebras. Others, like Zhu's theory, do not yet seem to be incorporated into the theory of chiral algebras. 

This article continues with the program detailed in \cite[\S 1.6]{eh2018}. In particular
we cover item a) described there confirming that certain extensions of modules
produce degree $1$ chiral homology classes by extending the result from the
limit as $q \rightarrow 0$ using the flat connection on the moduli space of
elliptic curves. The other points of the cited program include chiral homology
in higher degrees and in higher genera. 

The complexes described in Section \ref{sec:first.chiral} admit obvious
generalizations to higher degrees, and so do the finiteness conditions on
$\HK_\bullet(A)$ and $\HP_\bullet(R_V)$. By increasing the number of insertion
points, namely computing $H^{\text{ch}}_k(V^{\otimes n})$ with $n \geq k$, the
higher derived functors of global sections do not play any role, as the
varieties involved become affine. We expect the same techniques used this article
to apply to compute higher degree chiral homology of elliptic curves. 

Regarding chiral homology in higher genus. 
Recently, it was proved in a series of papers by Damiolini, Gibney and Tarasca
\cite{damiolini2020conformal,damiolini2019factorization} that the degree $0$
chiral homology of a $C_2$ cofinite vertex algebra $V$ is finite dimensional in
arbitrary curves, not only genus $1$. Their approach is to describe the chiral
homology on genus $g$ by understanding the structure in the boundary of the
moduli space, thus reducing to lower genera. It would be interesting to see if
this appoach can be generalized to higher chiral homology, and in particular if
the conditions on the Poisson homology of $R_V$ and the Koszul complex
$K_\bullet(A)$ described above, continue to play the same role in arbitrary
genus. 
\label{no:speculations}
\end{nolabel}
\begin{nolabel}[Algebro-geometric considerations] 
We collect here some observations that may help the reader acquainted with
\cite{beilinsondrinfeld}. This subsection is not needed anywhere in the text.

Let $X$ be a smooth algebraic curve of genus $g$, $V$ a conformal vertex algebra and $\cA = \cA_V$
the corresponding chiral algebra on $X$. There are different flavours of chiral
homology that one may compute. First, there is simply $H^{\ch}_\bullet(X, \cA)$,
the \emph{chiral homology of $X$ with coefficients in $\cA$}. When $X$ is an
elliptic curve $E_\tau = \mathbb{C} / \mathbb{Z} + \tau \mathbb{Z}$, these
homologies are denoted by $H^{\ch}_i(E_\tau, V)$ below, for $i=0,1$. 

For a collection of points $\{ x_i \}_{i=1}^n \subset X$  we have the chiral homology \emph{with supports} $H^{ch}_\bullet(X,
\{\cA_{x_i} \})$, that is, the chiral homology with coefficients in the collection of $\cA$-modules $\cA_{x_i}$, supported at the points $x_i$. For every $n$ and every
collection $\{ x_i\}$ these homologies are isomorphic. When $X = E_\tau$ and $x_i$
is the image of $t_i \in \mathbb{C}$. These homologies are denoted
$H^{\ch}_i(V^{\otimes n}, t_1,\dots,t_n, \tau)$ below. 

The construction works in families, that is, when $X/S$ is a curve over a base
$S$, and $\left\{ x_i \right\}$ is a collection of $S$-valued points. One
obtains a twisted $\mathcal{D}_S$-module, whose fiber at $s \in S$ equals
$H^{\ch}_\bullet(X_s, \{ \cA_{x_i(s)} \})$. Applying this to 
$S=\mathcal{M}_{g,n}$, the moduli space of $n$-marked, genus $g$ curves and $X$
is the universal curve,  the
chiral homologies  give rise to twisted $\mathcal{D}$-modules, or quasi-coherent sheaves with
projectively flat connections on $\cM_{g,n}$, denote these sheaves by
$\cH^{\ch}_\bullet(V^{\otimes n})$. In the case when $g = 1$, the
global sections of this $\mathcal{D}$-module on $\cM_{1,n}$ are denoted
$H^{\ch}_i(V^{\otimes n})$ for $i = 0,1$ below. Its zeroth de Rham homology
(coinvariants by the action of vector fields) is denoted by
$H^{\ch}_i(V^{\otimes})_{\nabla}$ below. Similarly global sections of the dual
quasi-coherent sheaf $\cH^{\ch}(V^{\otimes n})^*$ are denoted
$H_i^{\ch}(V^{\otimes n})^*$ below, and their
flat sections, or equivalently its zeroth de Rham cohomology, or the
$\mathcal{D}$-module push-forward to the point, are denoted by $H_i^{\ch}(V^{\otimes n})^{*\nabla}$. 

The moduli space of marked curves fibers over the moduli space of curves
$\pi:\mathcal{M}_{g,n} \rightarrow \mathcal{M}_{g}$ by forgetting the points (in the
case $g=1$ we do remember one point, the polarization of the elliptic curve).
Instead of pushing forward $\cH^{\ch}_\bullet(V^{\otimes n})^*$ down to the
point, we may consider $\pi_*$ of it. This is the relative $0$ de Rham
cohomology and corresponds to taking flat sections with respect to the action of
the vertical vector fields, or the relative connection corresponding to moving
the points $\{x_i\}$ in the curve $X \in \mathcal{M}_g$. The global sections of
$\pi_* \cH^{\ch}_\bullet(V^{\otimes n})^*$, when $g=1$,  are denoted
$H^{\ch}_i(V^{\otimes})^{* \text{tr}}$ below.

We can fix $X$ and move the points
$\{ x_i \}$, obtaining a $\mathcal{D}$-module on $X^n$, which is symmetric,
namely it is a pullback of a $\mathcal{D}$-module on  $\Sym^n X$, the universal degree
$n$-divisor on $X$. Its fiber at the point $\{ x_i\} \in X^n \setminus \Delta$
(where $\Delta \subset X^n$ is the diagonal divisor) being
$H^{ch}_\bullet(X, \{ A_{x_i}\})$. Denote this $\mathcal{D}_{X^n}$--module
by $\mathcal{H}^{ch}_\bullet(X,V^{\otimes n})$. 

Here is another construction $\cH^{\text{ch}}_\bullet(X, V^{\otimes n})$. When viewed
as a Lie algebra on the Ran space $\mathrm{Ran}(X)$ of $X$, the chiral algebra
$\cA$ can have modules supported not only at $X$ but at higher powers
$\Delta^{(n)}: X^n \rightarrow \mathrm{Ran}(X)$. We can consider for example the
image of $\Delta^{(n)}_* \cA^{\boxtimes n}$ as a module over $\cA$. The chiral
homology with coefficients in $\Delta^{(n)}_* \cA^{\boxtimes n}$ gives rise to a
$\mathcal{D}$-module on $\Sym^n X$, its pullback to $X^n$ coincides with
$\mathcal{H}^{ch}_\bullet(X, V^{\otimes n})$. The construction of the complex
$C^n_\bullet$ below, follows this pattern in the case when $X$ is the universal
elliptic curve and we restrict to homologies in degrees $0$ and $1$. 
\label{no:goemetric-intro}
\end{nolabel}
\begin{nolabel}The structure of this paper is as follows. In section
\ref{sec:prelim} we recall the definition of vertex algebras and their modules.
We recall the exponential change of coordinate formula. In section
\ref{sec:homology} we recall the definition of the Poisson homology of a Poisson
algebra and prove the relation between the finite generation of the kernel of
\eqref{eq:1surjection} and the first homology of the Koszul complex. In section
\ref{sec:elliptic} we define several spaces of meromorphic elliptic functions
that we will use throughout.  We study their Fourier series expansions and the
failure of ellipticity when differentiating with respect to the modular parameter.
In section \ref{sec:first.chiral} we define the complexes computing the chiral
homologies and their duals in genus $1$. We give explicit proofs that these
complexes are in fact complexes. We define the homology with
supports and study their $\mathcal{D}$-module structure when moving these
points. 
We describe explicitly the flat connection with respect to the modular
parameter, and the corresponding modular differential equation that flat
sections satisfy. In section \ref{sec:degree-1-conf} we define the notion of
\emph{degree $1$ conformal block on the torus} and prove its modular invariance. This is a generalization of the
notion of $n$-point functions in the degree $0$ case, and consists of a coherent system of
flat sections of the dual of chiral homology with coefficients in $V^{\otimes n}$ for all
$n \geq 1$. In section \ref{sec:modified.v.o} we prove the main technical tool
of the article, a version of Borcherds identity on Fourier variables (Theorem
\ref{prop:borcherds-modified-fourier}). In section \ref{sec:derivations} we
study self extensions of modules and their associated derivations. We prove a
version of the above Borcherds formula for these derivations. Section
\ref{sec:higher.traces} is the main technical section of the article. We define
the linear functionals associated to self extensions as formal power series. We
prove that they satisfy the formal differential equation corresponding to the
flat connection and we prove an \emph{insertion formula}. In section
\ref{sec:series.exp} we prove the finite generation of chiral homology under the
finiteness conditions on $V$ explained above. We use this to prove convergence
of the trace functional of section \ref{sec:higher.traces} for $n=1$. We prove here the
vanishing of chiral homology of classically free rational vertex algebras. In
section \ref{sec:formal-to-conformal} we prove the convergence of the $n$-point functions of section \ref{sec:higher.traces} for arbitrary $n$. We show
that their limit is a degree 
$1$ conformal block on the
torus.  In section \ref{sec:examples} we give some examples of vanishing of
chiral homology and some non-trivial classes obtained by the above construction.
In section \ref{sec:conclusion} we summarize some of the results obtained.
\label{no:description}
\end{nolabel}
\begin{nolabel}[Acknowledgements] This work has been carried out over an
extended period of time. We have benefited from discussion with many
mathematicians. We  would like to thank Drazen Adamovic, Tomoyuki Arakawa, David Ben-Zvi, John Francis, Dennis
Gaitsgory, Deniz Kus, Antun Milas, Sam Raskin and  Pavel Safronov for enlightening discussions. RH was
supported by CNPq Grant Number 305688/2019-7. JvE was supported by CNPq Grant Number 303806/2017-6 and by a grant from the Serrapilheira Institute (grant number Serra -- 1912-31433).
\label{no:aknowledgements}
\end{nolabel}
\section{Preliminaries on Vertex Algebras}\label{sec:prelim}
In this section we collect the basic definitions and notations on vertex algebras. We start by describing the notation for Laurent expansions of multi-variable functions. We give the definition of vertex algebra in \ref{no:definition-vertex-algebra}, and of modules in \ref{no:modules}. We recall the exponential change of coordinates in \ref{no:zhu-alter-vertex}. 
\begin{nolabel}
Let $f(z,w) \in \mathbb{C}[ [z,w]][z^{-1},w^{-1},(z-w)^{-1}]$. We denote by $i_{z,w} f(z,w) \in \mathbb{C}( (z))( (w))$ and $i_{w,z} f(z,w) \in \mathbb{C}( (w))( (z))$ the expansions of $f(z,w)$ in the domains $|z| > |w|$ and  $|w| > |z|$ respectively. More generally, for $f(z_1,\dots,z_n) \in \mathbb{C}[ [z_1,\dots,z_n]][z_i^{-1}, (z_i-z_j)^{-1}]_{1 \leq i \neq j \leq n}$ we denote by
\begin{align*}
i_{z_1,\dots,z_n} f(z_1,\dots,z_n) \in \mathbb{C}( (z_1))\dots ( (z_n)), \end{align*}
the expansion of $f(z_1,\dots,z_n)$ in the domain $|z_1| > \dots > |z_n|$. This expansion may be done iteratively; first expand all instances of $(z_i - z_n)^n$ in positive powers of $z_n$, obtaining a Laurent series in $z_n$ whose coefficients are functions of $z_1,\dots,z_{n-1}$, then repeat with $z_{n-1}$, etc. 

Sometimes it will be required to expand $f(z_1,\dots,z_n)$ as above in powers of $z_i -z_j$. We will only need such expansions in the domain $|z_i| > |z_i - z_j|$, where $1 \leq i < j \leq n$. To perform such an expansion means to replace all instances of $z_i^{-k}$, where $k \geq 0$, by $\left((z_i - z_j) + z_j\right)^{-k}$ and expand in positive powers of $z_i - z_j$. Therefore such an expansion is a series in $z_1,\dots,\hat{z_i},\dots,z_n,z_i-z_j$, where we have omitted $z_i$ from the list. 
\label{no:expansions}
\end{nolabel}
\begin{ex}
For $f = f(z,w,t)$ the expansion 
\begin{align*}
i_{w,w-t,z-w} f(z,w,t) \in \mathbb{C} ( (w))( (w-t)) ( (z-w)),
\end{align*}
is computed as follows. We take the Laurent series expansion $f(z,w,t) = \sum_{k} f_k (w,t) (z-w)^k$ of $f$ around $z-w = 0$. Then we take the Laurent series expansion $f_k(w,t) = \sum_{\ell} f_{\ell, k}(w) (w-t)^\ell$ of each $f_k$ around $w-t = 0$. Finally we write
\begin{align*}
i_{w,w-t,z-w} f(z,w,t) = \sum_{\ell, k} f_{\ell, k}(w) (w-t)^\ell (z-w)^k.
\end{align*}
\label{ex:expansion12}
\end{ex}
\begin{nolabel}
For $f(z,w) \in \mathbb{C} [ [z,z^{-1},w,w^{-1}]]$ we write $\res_z dz f(z,w)  \in \mathbb{C} [ [w,w^{-1}]]$ for its coefficient of $z^{-1}$. For $f(z,w) \in \mathbb{C}[ [z,w]][z^{-1},w^{-1},(z-w)^{-1}]$ the notation $\res_z dz f(z,w)$ is ambiguous. We must choose an embedding $\mathbb{C}[ [z,w]][z^{-1},w^{-1},(z-w)^{-1}] \subset \mathbb{C}[ [z,z^{-1},w,w^{-1}]]$ in order to compute this residue. For example
\begin{align*} \res_z dz i_{z,w} \frac{1}{z-w} = 1 \neq 0 = \res_z dz i_{w,z} \frac{1}{z-w}. \end{align*}

By $\res_{z=w}dz f(z,w)$ we mean the coefficient of $(z-w)^{-1}$ in the expansion of $f(z,w) = f( (z-w) + w, w)$ in the domain $|w| > |z-w|$. 

Notice that for any Laurent series $f(t) \in \mathbb{C}( (t))$ we have
\begin{equation}
\res_t dt  \frac{d}{dt} f(t) = 0, \qquad \res_t dt \left( t \frac{d}{dt} + 1 \right) f(t) = 0.
\label{eq:residue-trading-exp}
\end{equation}
\end{nolabel}
\begin{nolabel}Throughout this article $(V, \vac, \omega, Y(\cdot,z))$ will denote a conformal vertex algebra (or vertex operator algebra). In particular $V = \bigoplus_{n \in \Z_+} V_n$ is a $\Z_+$-graded complex vector space $V$ together with two distinguished vectors $\vac$ and $\omega$ and a bilinear map 
\begin{align*} V \otimes V \rightarrow V( (z)), \qquad a \otimes b \mapsto Y(a,z)b = \sum_{n \in \mathbb{Z}} z^{-1-n} a_{(n)}b, \qquad a_{(n)}b = 0 \quad \forall n \gg 0, \end{align*}
satisfying the following axioms
\begin{enumerate}
\item $Y(\vac, z) = \id_V$. 
\item $a_{(n-1)} \vac = \delta_{n,0} a$ for all $a \in V$ and $n \geq 0$.
\item For any $f(z,w) \in \mathbb{C}[ [z,w]][z^{-1},w^{-1},(z-w)^{-1}]$ and $a,b,c \in V$, the following Borcherds identity holds
\begin{multline}
\res_z dz Y(a,z) Y(b,w) c i_{z,w} f(z,w)
- \res_z dz Y(b,w)Y(a,z)c i_{w,z} f(z,w) \\
= Y\left( \res_{z=w}dz f(z,w) Y(a,z-w)b,w \right)c. 
\label{eq:borcherds-def}
\end{multline}
\item Put $L(z) = \sum_{n \in \mathbb{Z}} L_n z^{-n-2} = Y(\omega, z)$. Then 
\begin{align*}
[ L_m, L_n] = (m-n) L_{m+n} + \frac{c}{12} (m^3 -m) \delta_{m,-n} \id_Vm
\end{align*}
where $c \in \mathbb{C}$ is a constant called the \emph{central charge} of $V$. 
\item $L_{-1} a = a_{(-2)} \vac$ for every $a \in V$. 
\item For each $n \in \Z_+$ the restriction of $L_0$ to $V_n$ is $n \id_{V_n}$ and $V_n$ is finite dimensional.
\end{enumerate}

\label{no:definition-vertex-algebra}
\end{nolabel}
\begin{nolabel}A vertex algebra $V$ satisfies the following skew-symmetry relation:
\begin{align*} Y(a,z)b = e^{z L_{-1}} Y(b,-z) a, \end{align*}
multiplying by $z^{k}$ and taking residues we get
\begin{equation} 
 a_{(k)} b = - (-1)^k \sum_{j \geq 0} \frac{(-L_{-1})^j}{j!} b_{(k+j)}a. 
\label{eq:skew-symmetry}
\end{equation} 
\label{no:skew-syum}
\end{nolabel}
\begin{nolabel} For us a \emph{module} over the conformal vertex algebra $V$ will be what is normally called a positive energy module or a $\mathbb{N}$-gradable weak module in the literature. This is a vector space $M$ equipped with a bilinear operation 
\begin{align*} V \otimes M \rightarrow M( (z)), \quad a\otimes m \mapsto Y^M(a,z)m = \sum_{n \in \mathbb{Z}} z^{-n-1} a^M_{(n)} m, \quad a^M_{(n)}m = 0 \quad \forall n \gg 0, \end{align*}
satisfying the following analog of the Borcherds identity for $a, b \in V$, $m \in M$ and $f(z,w) \in \mathbb{C}[ [z,w]][z^{-1},w^{-1},(z-w)^{-1}]$: 
\begin{multline}
\res_z dz Y^M(a,z) Y^M(b,w)m \,i_{z,w} f(z,w) - \res_z dz Y^M(b,w)Y^M(a,z)m\, i_{w,z} f(z,w)  = \\
Y^M\left( \res_{z=w}dz f(z,w) Y(a,z-w)b,w \right)m. 
\label{eq:module-def}
\end{multline}
We put $L^M(z) = Y^M(\omega,z)$ and require that $L^M_0$ act with finite dimensional generalized eigenspaces, and that the real part of its spectrum be bounded below. That is $M = \bigoplus_{\lambda \in \mathbb{C}} M_\lambda$ with $ (L^M_0|_{M_\lambda} - \lambda)^n = 0$ for $n \gg 0$. Below we will often drop the superscript $M$ on the operators whenever no confusion should arise. 

\label{no:modules}
\end{nolabel}
\begin{nolabel} Let $\tau \in \mathbb{H}$ and $q = e^{ 2 \pi i \tau }$. For a module $M$ the operator $q^{L_0} = e^{2\pi i \tau L_0}$ is well-defined since the generalized eigenspaces of $L_0$ are finite dimmensional by hypothesis. In Section \ref{sec:higher.traces} we will deal with formal sums like $\tr_M q^{L_0}$ by which we mean, in this context, the expression
\begin{align*} \sum_\lambda q^{\lambda} \sum_{n \geq 0} \frac{(2 \pi i \tau)^n}{n!} (L_0 - \lambda)^n, \end{align*}
where $q$ is a formal variable. Convergence of such series under appropriate hypotheses is then treated in Sections \ref{sec:series.exp} and \ref{sec:formal-to-conformal}.

\label{no:ql0}
\end{nolabel}
\begin{nolabel} For a vertex operator algebra $(V, \vac, \omega, Y(\cdot,z))$ we will consider another isomorphic vertex algebra $(V, \vac, \tilde{\omega}, Y[\cdot,z])$ introduced by Zhu \cite{zhu} where $\tilde{\omega} = (2 \pi i)^2 ( \omega - \tfrac{c}{24} \vac )$ and 
\begin{align*} Y[a,z] = Y\left( e^{2 \pi i z L_0}a, e^{2 \pi i z}-1\right). \end{align*}
We have for homogeneous $a \in V$ and $m \in \mathbb{Z}$
\begin{equation}
a_{[m]} = (2 \pi i)^{-m-1} \res_z \left( Y(a,z) \left( \log(1+z) \right)^m (1+z)^{\deg a -1}  \right).
\label{eq:zhu-bracket-mode}
\end{equation}
\label{no:zhu-alter-vertex}
\end{nolabel}
\begin{nolabel}For a Laurent series $f(t) = \sum_{n \geq N} f_n t^n$, $a,b \in V$ we will use the notation 
\begin{align*} a_{(f)} b = \res_{t} f(t) Y(a,t)b = \sum_{n \geq N} f_n a_{(n)} b,  \qquad a_{[f]} b = \res_t f(t) Y[a,t]b = \sum_{n \geq N} f_n a_{[n]}b. \end{align*}
\begin{rem}There is an ambiguous situation in that for a integer number $n$, $a_{(n)}$ could mean the $n$-product as defined in \ref{no:definition-vertex-algebra} or the $f$-product for the constant Laurent series $f(t) = n t^0$. We will always mean the $n$-th product by this notation and believe that no confusion should arise. 
\label{rem:f-vs-n-product}
\end{rem}

By integration by parts and noting that the translation operator for $(V,\vac, \tilde{\omega},Y[\cdot,z])$ is given by $ (2 \pi i) (L_0 + L_{-1})$ these operations satisfy
\begin{equation}
(L_{-1} a)_{(f)} b = - a_{(f')} b
\label{eq:integration-parts}
\end{equation}
and
\begin{equation}
(2 \pi i) \left(\left( L_0 + L_{-1} \right) a \right)_{[f]} b = - a_{[f']}b.
\label{eq:integration-parts-alt}
\end{equation}
\label{no:f-product}
\end{nolabel}
We have the following 
\begin{lem}\label{lem:skew-symmetry}
Let $f \in \mathbb{C}((x))$. Then
\begin{equation}
a_{(f(x))}b + b_{(f(-x))}a = \sum_{j \geq 0} \frac{(-1)^j}{(j+1)!} T^{j+1}\left(a_{(x^{j+1} f(x))}b\right), \qquad \text{for all $a, b \in V$}.
\label{eq:lemma-sym}
\end{equation}
\end{lem}
\begin{nolabel}\label{no:C2.def}
Let $V$ be a vertex algebra. We denote by $V_{(-2)}V$ the subspace of $V$ spanned by elements of the form $a_{(-2)}b$ for $a, b \in V$. The $C_2$-algebra $R_V = V / V_{(-2)}V$ of $V$ carries the natural structure of Poisson algebra with commutative product $a \cdot b = a_{(-1)}b$ and Poisson bracket $\{a, b\} = a_{(0)}b$. The vertex algebra $V$ is said to be $C_2$-cofinite if $R_V$ is finite dimensional.
\end{nolabel}

\section{Homology}\label{sec:homology}
In this section we recall basic homological constructions that will be used in Section \ref{sec:series.exp} to codify finiteness conditions on the vertex algebra $V$ guaranteeing finite dimensionality of its chiral homology. We recall the construction of the \emph{Koszul complex of a differential algebra} and the \emph{arc algebra} of a commutative algebra. We describe a relation between the first homology group of the Koszul complex of a graded differential algebra and finite generation of a certain ideal in an associated arc algebra. We then recall the definition of Poisson homology of a Poisson algebra and give an explicit description of the differentials in low degree.  
\begin{nolabel}
Let $k$ be a field and $A$ a commutative $k$-algebra. Denote by $\Om_A = \Om_{A/k}$ the $A$-module of K\"{a}hler differentials on $A$. It is isomorphic as a $k$-vector space to the Hochschild homology group $\Hoch_1(A)$. A derivation $T$ of $A$ induces a morphism
\begin{align*}
\iota_T : \Om_{A} \rightarrow A,
\end{align*}
of $A$-modules, which in turn extends uniquely to a differential $d$ on the Koszul complex $\wedge^\bullet \Om_{A}$. We refer to this as the Koszul complex of the differential algebra $(A, T)$ and denote it by $K_\bullet(A)$. Explicitly, the differentials of $K_\bullet(A)$ in low degree are
\begin{align}
d_1 : \Om_{A} &\rightarrow A, & a \, dx &\mapsto a \cdot (Tx), \label{eq:d1.HP} \\
d_2 : \wedge^2 \Om_{A} &\rightarrow \Om_{A}, & a \, dx &\wedge dy \mapsto a
\cdot (Tx) \, dy - a \cdot (Ty) \, dx. \label{eq:d2.HP2}
\end{align}
We denote the homology groups of this complex by $\HK_i(A)$ for $i \in \Z_+$.
\end{nolabel}
\begin{nolabel}
Let $R$ be a commutative $k$-algebra. The arc algebra of $R$ is the differential algebra $(JR, \partial)$ characterized by the universal property that any morphism of algebras $R \rightarrow A$ from $R$ to a differential algebra $(A, T)$, lifts uniquely to a morphism of differential algebras $(JR, \partial) \rightarrow (A, T)$ making the following diagram commute
\begin{align} \label{eq:jet-univ}
\begin{split}
\xymatrix{
& (JR, \partial) \ar[d] \\ 
R \ar[ur] \ar[r] & (A, T). \\
}
\end{split}
\end{align}
The arc algebra is naturally $\Z_+$-graded with $JR^0 = R$, and $\partial$ has degree $+1$.
\end{nolabel}
\begin{prop}\label{prop:Koszul.finiteness}
Let $A = \bigoplus_{n \in \Z_+} A^n$ be a graded differential $k$-algebra with derivation $T$ of degree $+1$. Assume $A^0$ is of finite type as a $k$-algebra, and that $A$ is generated by $A^0$ as a differential algebra. The canonical map $\pi : JA^0 \rightarrow A$ is a surjection and we write $I$ for its kernel, which is a differential ideal of $J = JA^0$. Then
\begin{enumerate}
\item $\HK_1(A) = 0$ if and only if $\pi$ is an isomorphism, and

\item in general there exists an injection
\[
\HK_1(A) \rightarrow I / (J \cdot TI).
\]
\end{enumerate}
\end{prop}
\begin{proof}
Part a) is {\cite[Thm. 13.7]{eh2018}}. Part b) is proved by a diagram chase. Indeed we have a morphism of complexes
\[
\xymatrix{
\wedge^2 \Om_J \ar[d] \ar[r] & \Om_J \ar[d] \ar[r] & J \ar[d] \ar[r] & 0 \\ 
\wedge^2 \Om_A \ar[r] & \Om_A \ar[r] & A \ar[r] & 0,
}
\]
where the upper row is exact by a). Let $\omega = \sum_i a_i \, db_i \in \Om_A$ be a $1$-cocycle, and let $\tilde{a}_i$ and $\tilde{b}_i$ be lifts of $a_i$ and $b_i$ to $J$, and write $\tilde{\omega} = \sum_i \tilde{a}_i \, d\tilde{b}_i \in \Om_J$. Clearly
\[
d\tilde{\omega} = \sum_i \tilde{a}_i \cdot T\tilde{b}_i \in I,
\]
since $\omega$ is closed. We claim that $\omega \mapsto d\tilde{\omega}$ defines a morphism
\[
\delta : \HK_1(A) \rightarrow I / (J \cdot TI).
\]
Indeed if $\omega' - \omega = d\theta$ where $\theta = \sum_i a_i \, db_i \wedge dc_i$, then
\[
\tilde{\omega}' - \tilde{\omega} - d\sum_i \tilde{a}_i \, d\tilde{b}_i \wedge d\tilde{c}_i \in I \, dJ + J \, dI,
\]
and so
\[
\delta(\omega') - \delta(\omega) \in I \cdot (TJ) + J \cdot (TI) \subset J \cdot (TI).
\]
Obviously $\delta$ is injective: $\delta(\omega) \in J \cdot TI$ implies $\omega$ possesses a lift $\tilde{\omega}$ of the form $\sum_i \tilde{a}_i \, d(\tilde{x}_i)$ which implies $\omega = 0$.
\end{proof}
\begin{nolabel}
We briefly recall the definition of Poisson homology groups. These were
introduced in \cite{lichnerowicz} in a geometric context, the presentation here
follows the algebraic formulation of \cite{Huebschmann}. First we recall the
notion of Lie-Rinehart pair. Let $k$ be a field (though the constructions work
with $k$ a commutative ring in general) and $\OO$ a commutative $k$-algebra.
Then a Lie-Rinehart pair, or a $\cO$-Lie algebroid consists of an $\OO$-module
$\cL$ together with a $k$ Lie algebra structure on $\cL$ and an action $\omega :
\cL \rightarrow \Der(\OO)$ of $\cL$ on $\OO$ by derivations. These data satisfy
the conditions:
\begin{align*}
\omega(f x)(g) = f \omega(x)(g) \quad \text{and} \quad [x, fy] = f[x, y] + (\omega(x)f) y
\end{align*}
for all $f, g \in \OO$ and all $x, y \in \cL$. 
Given a $\cO$--Lie algebroid $\cL$, we have its Chevalley-Eilenberg complex
\begin{equation} \label{eq:ce1} C(\cL)_{>0} = \wedge^{>0}_{\cO} \cL. 
\end{equation}
The differential is constructed as follows. In degree $2$ we let 
\[ d (x \otimes y) = [x,y], \]
and extend $d$ by the Leibniz rule to be a derivation of degree $-1$ of the super-commutative algebra $C(\cL)$. 

Let $A$ be a Poisson algebra over $k$. The $A$-module $\Om_{A/k} = \Hoch_1(A)$ of K\"{a}hler differentials on $A$ carries the structure of an $A$--Lie algebroid. The $k$--Lie algebra bracket is defined by
\begin{align*}
[f\,dx, g\,dy] = f\{x, b\}\,dy + g\{f, y\}\,dx + fg\,d\{x, y\}, \qquad f,g,x,y \in A,
\end{align*}
and the action is given by
\[ \omega( f dx) (g) = f \{x,g\}. \]

The Eilenberg-Chevalley complex defined in \eqref{eq:ce1} is augmented, by putting $A$ in degree $0$ and letting the differential in degree $1$, be given by 
\begin{equation}
d_1 : \Om_{A/k} \rightarrow A, \qquad  a \, dx \mapsto \{a, x\}, \label{eq:d1.HP2} 
\end{equation}
We let $C(\Omega^1_A)$ be this augmented complex. Its homology is known as the \emph{Poisson homology} $\HP_\bullet(-)$ of $A$: 
\begin{align*}
\HP_i(A) = H^{i} C(\Omega^1_A).
\end{align*}
We record here the differential in degree $2$ that will be used below:
\begin{equation}
d_2 : \wedge^2 \Om_{A/k} \rightarrow \Om_{A/k}, \qquad  a \, dx \wedge dy \mapsto \{a, x\} \, dy - \{a, y\} \, dx - a \, d(\{x, y\}). \label{eq:d2.HP}
\end{equation}
\end{nolabel}
\section{Elliptic Functions}\label{sec:elliptic}
In this section we recall some material about elliptic functions. We adopt the conventions of \cite{zhu}. We recall the Eisenstein series in \ref{no:eisenstein} and the Weierstrass elliptic functions and their Fourier series expansions in \ref{no:ellitptic-functions}. We introduce the rings of modular forms in \ref{no:modular-foms}, Jacobi forms in \ref{no:jacobi-forms} and meromorphic elliptic functions in \ref{no:meromorphic-functions}. We analyze domains of convergence of Fourier series expansions in \ref{no:fourier-n-variable} and define a derivation with respect to the modular parameter in \ref{lem:deriv-sum}.

\begin{nolabel}For $k \geq 1$ we consider the Eisenstein series
\begin{equation}
G_{2k}(q) = 2 \xi (2 k) + \frac{2 (2 \pi i)^{2k}}{(2k-1)!} \sum_{n =1}^\infty \sigma_{2k -1}(n)q^n,
\label{eq:eisenstein-definition}
\end{equation}
as a formal power series in the variable $q$. Here 
\begin{align*}
\xi(2k) = \sum_{n \geq 1} n^{-2k} \quad \text{and} \quad \sigma_k(n) = \sum_{d|n} d^k.
\end{align*}
This series converges uniformly and absolutely in the domain $|q| < 1$, and thereby defines a function $g_{2k}(\tau) = G_{2k}(e^{2 \pi i \tau})$ analytic in the upper half complex plane $\HH$. For $k \geq 2$ the function $g_{2k}(\tau)$ is a holomorphic modular form of weight $2k$. The function $g_2(\tau)$ transforms under the action of $\SL_2(\mathbb{Z})$ as 
\begin{equation}\label{eq:g2-non-mod}
g_2 \left( \frac{a \tau + b}{c \tau +d} \right) = (c\tau+d)^{2} g_2(\tau) - 2 \pi i c (c\tau+d).
\end{equation}
\label{no:eisenstein}
\end{nolabel}

\begin{nolabel} For $k \geq 1$ we introduce the formal Laurent series
\begin{equation}
\wp_k(t) = \frac{1}{t^k} + (-1)^k \sum_{n \geq 1} \binom{2n+1}{k-1} G_{2n + 2}(q) t^{2n + 2 - k},
\label{eq:wp-def}
\end{equation}
as an element of $\mathbb{C}[[q]]((t))$. These series are related to each other by
\begin{equation}
\wp_{k+1} (t)  =  - \frac{1}{k} \frac{d}{dt} \wp_k(t). 
\label{eq:differential-equation}
\end{equation} Putting $q = e^{2 \pi i \tau}$ in $\wp_k(t)$ yields a Laurent series in $t$ which converges uniformly and absolutely in a neighborhood of $t=0$. The sum can be analytically continued to a meromorphic function $\wp_k(t,\tau)$ having poles at $t = m + n \tau$ for $m, n \in \mathbb{Z}$. For $k \geq 2$ this meromorphic function is doubly periodic:
\begin{align*} \wp_k(t + m + n \tau,\tau) = \wp_k(t,\tau), \quad m,n \in \mathbb{Z}, \end{align*}
while for $k=1$ one has
\begin{equation}\label{eq:wp-not-per}
\wp_1(t + m + n \tau, \tau) = \wp_1(t,\tau) + (m +n \tau) g_2(\tau) - 2 \pi i n. 
\end{equation}
It is convenient to introduce the functions
\begin{align*}
\zeta(t, \tau) = \wp_1(t,\tau) - g_2(\tau)t
\end{align*}
which, in light of \eqref{eq:wp-not-per}, satisfies
\begin{equation}\label{eq:zeta-not-per}
\zeta(t+m+n\tau, \tau) = \zeta(t,\tau) - 2\pi i n.
\end{equation}
It will also be convenient to introduce the normalized functions
\[ 
\owp_k(t,\tau) = 
\begin{cases}
\zeta(t,\tau) + \pi i & k = 1 \\ 
\wp_2(t,\tau) + g_2(\tau) & k = 2 \\ 
\wp_k(t,\tau) & k \geq 3.
\end{cases}
\]
They satisfy the same differential equation \eqref{eq:differential-equation}. 
We will often drop $\tau$ from the notation denoting $\wp_k(t)$ instead of
$\wp_k(t,\tau)$ when no confusion should arise. 

We also consider the formal series $P_k(z,q) \in \mathbb{C}[[q]] [[z,z^{-1}]]$:
\begin{equation}
P_k(z,q) = \frac{(2 \pi i)^k}{(k-1)!} \sum_{n = 1}^\infty \left( \frac{n^{k-1}z^n}{1 -q^n} + \frac{(-1)^k n^{k-1} z^{-n} q^n}{1-q^n} \right),
\label{eq:pk-def}
\end{equation}
(where $(1-q^n)^{-1}$ stands for the geometric series $\sum_{j \geq 0} q^{nj}$). These series are related to each other by
\begin{equation} \label{eq:pkdiff}
P_{k+1}(z,q) = \frac{2 \pi i }{k} z \frac{d}{dz} P_k(z,q).
\end{equation}
The series $P_k(z, q)$ converges uniformly and absolutely in every closed subset of the domain $|q| < |z| < 1$, and can be analytically continued to a meromorphic function of $z \in \mathbb{C} \backslash \{0\}$ with poles at $z = q^{n}$ for $n \in \mathbb{Z}$. The series $P_k(z,q)$ should be thought of as the Fourier series of the function $\wp_k(t,\tau)$. More precisely one has
\begin{equation}\label{eq:Pk.wpk.relation}
P_k (e^{2 \pi i t},e^{2 \pi i \tau}) = (-1)^k \owp_k(t, \tau),
\end{equation}
whenever $0 < \Im(t) < \Im(\tau)$. The series $P_k(qz,q)$ converges uniformly and absolutely in every closed subset of the domain $1 < |z| < |q|^{-1}$. For $k \geq 2$ the analytic continuation of the sum coincides with that of $P_k(z,q)$. On the other hand, by (\ref{eq:zeta-not-per}), the analytic continuation of $P_1(qz, q)$ coincides with that of $P_1(z,q)+2\pi i$.

Purely formal analogues of these statements also hold, namely
\begin{equation}
P_1(z,q) - P_1(zq,q) + 2 \pi i = 2 \pi i \sum_{n \in \mathbb{Z}} z^n
\label{eq:delta-log}
\end{equation}
and
\begin{equation}
P_k(z,q) - P_k(zq,q) = \frac{(2 \pi i)^k}{(k-1)!} \sum_{n \in \mathbb{Z}} n^{k-1} z^n \quad \text{for $k \geq 2$},
\label{eq:higher-delta-log}
\end{equation}
as formal power series.
\label{no:ellitptic-functions}
\end{nolabel}
\begin{lem} The series $P_k(z,q)$ satisfy 
\begin{equation}
P_k\left( z^{-1},q \right) = (-1)^k P_k(zq, q).
\label{eq:P-symmetric}
\end{equation}
\label{lem:p-sym}
\end{lem}

\begin{nolabel} As in \ref{no:f-product}, for $a,b \in V$ we have the following elements of $V[ G_4(q),G_6(q)] \subset V[[q]]$:
\begin{align*} a_{[\wp_k]} b = a_{[k]}b + (-1)^k \sum_{n \geq 1} \binom{2n+1}{k-1} G_{2n+2}(q) a_{[2n + 2 -k]}b.  \end{align*}
\label{no:wpk-propduct}
\end{nolabel}

\begin{nolabel} We will consider the graded ring of modular forms $\bM_*$ which is generated by $G_{4}(q)$ and $G_6(q)$. We will also need the ring of quasi-modular forms $\bQM_*$ generated by $G_{2}(q)$, $G_{4}(q)$ and $G_6(q)$. Both of these rings are Noetherian, indeed they are polynomial rings. We may consider these rings as subrings of $\mathbb C[[q]]$. 
\label{no:modular-foms}
\end{nolabel}
\begin{nolabel}  For each $n \geq 2$ we consider the following graded rings $\bJ^n_* = \oplus_{k \geq 0} \bJ^n_k$ of weak Jacobi modular  forms of index $0$. By definition $\bJ^n_k$ consists of functions $f(t_1,\cdots,t_n,\tau)$ satisfying  
\begin{equation}
\begin{aligned}
 f(t_1,\cdots, t_i + m + l \tau,\cdots,t_n, \tau) &= f(t_1, \cdots, t_n, \tau), & 1 \leq i \leq n, \qquad m,l \in \mathbb{Z},\\
 f\left( \frac{t_1}{c \tau + d}, \cdots, \frac{t_n}{c \tau + d}, \frac{a \tau + b}{c \tau + d} \right) &= (c \tau + d)^k f\left( t_1,\cdots,t_n,\tau \right), & ad -cb =1, \: a,b,c,d \in \mathbb{Z}. 
\end{aligned}
\end{equation}
The functions $f(t_1,\cdots,t_n,\tau)$ are meromorphic in $\mathbb{C}^n \times \mathbb{H}$ with possible poles at $t_i - t_j = m + l \tau$, $i \neq j$, $m, l\in \mathbb{Z}$. 

For every pair of non negative integers $n \geq m$ we have an inclusion $\bJ^m_* \rightarrow \bJ^n_*$ given by $f \mapsto \tilde{f}$ where $\tilde{f}(t_1,\ldots, t_n; \tau) = f(t_{n-m+1},\ldots, t_n; \tau)$. In this way $\bJ^n_*$ is an algebra over $\bJ^m_*$, and in particular all $\bJ^n_*$ are algebras over $\bM_*$.


Notice that $\bJ^0_* = \bJ^1_* = \bM_*$. The description of $\bJ^2_*$ is as follows: the function $\wp_k(t_1-t_2,\tau)$ is a weak Jacobi form of index $0$ and weight $k$ for $k \geq 2$, and $\bJ^2_*$ consists of linear combinations of $1, \wp_{k}(t_1-t_2,\tau)$, $k \geq 2$ with coefficients in $\bM_*$. It follows from \cite[Prop. 2.8]{libgober2009elliptic} that the ring  $\bJ^2_*$ is isomorphic to the polynomial ring $\mathbb{M}[ \wp_2, \wp_3, \wp_4]$.

Since every function $f(t_1,t_2,\tau)$ in $\bJ^2_*$ can be written in the form $f(t_1-t_2,\tau)$ for some $f$, we shall often abuse notation and write $f(t,\tau) \in \bJ^2_*$, it being implicit that $f(t_1,t_2,\tau) = f(t_1-t_2,\tau)$.
\label{no:jacobi-forms}
\end{nolabel}
\begin{nolabel}For $n \geq 3$ an inductive description of $\bJ^n_*$ can be given. Any $f(t_1,\dots,t_n,\tau) \in \bJ^n_*$ can be written as a linear combination of the form \cite{eh2018} 
\begin{align*}
f(t_1,\dots,t_n,\tau) &= g(t_2,\dots,t_n,\tau)  \\ &\quad +\sum_{i=3}^n g_i(t_2,\dots,t_n,\tau) \left( \wp_1(t_{i-1} - t_1,\tau) - \wp_1(t_{i} - t_1,\tau) + \wp_1 (t_{i} - t_{i-1},\tau) \right) \\  &\quad + \sum_{i=2}^n \sum_{m \geq 2} \wp_{m} (t_i - t_1,\tau) g_{im} (t_2,\dots,t_n,\tau),
\end{align*}
where the functions $g(t_2,\dots,t_n,\tau)$, $g_i(t_2,\dots,t_n,\tau)$ for $i = 3,\dots,n$, and $g_{im} (t_2,\dots,t_n,\tau)$ for $i =1,\dots,n$ and $m \geq 2$, are elements of $\bJ^{n-1}_*$. Inductively one obtains an expression for $f(t_1,\dots,t_n,\tau) \in \bJ^n_*$ as a linear combination of products of $\wp_m(t_i - t_j, \tau)$ for $1 \leq j < i \leq n$, and the functions 
\begin{align*} \zeta\left( t_i,t_{j-1},t_j,\tau \right) = \wp_1(t_{j-1} - t_i, \tau) - \wp_1\left( t_{j} - t_i,\tau \right) + \wp_1 \left( t_j - t_{j-1},\tau \right), \qquad 1 \leq i < j-1 \leq n.\end{align*}
These are weak Jacobi forms in three variables of index $0$ and weight $1$. The function $\wp_1(t,\tau)$ is not a weak Jacobi modular form, as we have seen in \ref{no:ellitptic-functions} it is not periodic, however the combination $\zeta\left( t_1,t_{2},t_3,\tau \right)$ is a Jacobi form.
\label{no:general-function}
\end{nolabel}
\begin{nolabel}
For each $n$ we shall denote by $\cF_n$ the ring of meromorphic functions $f(t_1,\dots,t_n,\tau)$ on $\mathbb{C}^n \times \mathbb{H}$ such that
\begin{align*}
f(t_1,\dots,t_i + m + l \tau, \dots,t_n,\tau) = f(t_1,\dots,t_n,\tau), \quad m,l \in \mathbb{Z},
\end{align*}
and $f$ is allowed to have poles at $t_i - t_j = m + l \tau$, $1 \leq i \neq j \leq n$, $m,l \in \mathbb{Z}$. We have $\bJ^n_* \subset \cF_n$. Notice also that a similar recursive description as in \ref{no:general-function} is valid for $\cF_n$. The difference being in the first step of the induction as $\cF_0 = \cF_1$ consists of holomorphic functions of $\tau$ as opposed to modular forms. In particular, functions in $f(t_0,t_1,\tau) \in \cF_2$ are in fact functions of one variable $f = f(t_0-t_1, \tau)$. More generally we have
\label{no:meromorphic-functions}
\end{nolabel}
\begin{lem} Let $f(t_0,\dots,t_n, \tau) \in \cF_{n+1}$, then 
\[ \sum_{i = 0}^n \frac{\partial}{\partial t_i} f(t_0,\dots,t_n,\tau) = 0. \]
\label{lem:4.sum-deriv}
\end{lem}
\begin{rem} Notice that in the definition of $\zeta(t_1,t_2,t_3,\tau)$ we may
replace  $\wp_1(t,\tau)$ by $\owp_1(t,\tau)$ without effect. It follows that the Fourier series expansion of $\zeta(t_1,t_2,t_3,\tau)$ in the domain $|qz_1|<|z_2|<|z_3| < |z_1|$ is given by
\begin{align*} - P_{1} \left( \frac{z_2}{z_1},q \right) + P_1 \left( \frac{z_3}{z_1},q \right) - P_1\left( \frac{z_3}{z_2},q \right) - \pi i. \end{align*}
\end{rem}
The following is standard
\begin{lem} Let $f(t_1,\dots,t_n,\tau) \in \bJ^n_p$, then for each $1 \leq i < j \leq n$ the Laurent series expansion 
\begin{align*} f (t_1,\dots,t_n, \tau) = \sum_{k \geq N} f_k(t_1,\dots,\hat{t}_i,\dots,t_n, \tau) (t_i - t_j)^k, \end{align*}
where the variable $\hat{t}_i$ is missing in
the coefficient $f_k \in \bJ^{n-1}_{p+k}$. These coefficients are given as an integral
\begin{align*} f_k (t_1,\dots, \hat{t_i}, \dots,t_n,\tau) = \oint_{C_{j}} dt_i \frac{f(t_1,\dots,t_n,\tau)}{(t_i -t_j)^{k+1}}, \end{align*}
where the contour $C_j$ is a small circle in the $t_i$-plane around $t_j$ and not containing any other point $t_i = t_l$ for $l \neq j$. 
\label{lem:Laurent-exp}
\end{lem}
\begin{nolabel} Being periodic, any element $f(t_1,\cdots,t_n,\tau) \in \cF^n$ admits a Fourier series. That is a formal power series
\begin{equation}
 F(z_1,\cdots,z_n,q) =  \sum_{m_1,\ldots,m_n\in \mathbb{Z}} F_{m_1,\ldots,m_n}(q) z_1^{m_1}\cdots z_n^{m_n}, 
\label{eq:fourier-n-variable}
\end{equation}
where each coefficient $F_{m_1,\ldots,m_n}(q) \in \mathbb{C}[ [q]]$. This series converges uniformly and absolutely in any closed subset of the domain $|qz_1| < |z_n| < \cdots < |z_1|$ and it can be analytically continued to a meromorphic function with poles at $z_i = z_j q^m$, $i \neq j$, $m \in \mathbb{Z}$. In its domain of convergence we have
\begin{align*} F\left( e^{2 \pi i t_1},\cdots,e^{2 \pi i t_n},e^{2 \pi i \tau} \right) = f(t_1,\cdots,t_n,\tau).\end{align*}
Notice that for each $1 \leq i \leq n$ the series $F(z_1,\cdots,  z_{i-1} ,q z_{i}, \cdots, q z_n,q)$ converges absolutely and uniformly in any closed subset of the domain 
\[ |q z_i| < |z_{i-1}| < \dots < |z_1| < |z_n| < \dots < |z_i|, \]
and can be analytically extended to the same meromorphic function as $F(z_1,\cdots,z_n,q)$. That is, we have ``cyclically'' shifted the order of the modulus of the variables in the domain of expansion. In the same vein, we will use the fact that 
\begin{equation}
 F(z_1,\dots,z_n, q) = F(q z_1,\dots, q z_n, q),
\label{eq:domain-change}
\end{equation}
which follows since both are Fourier expansions of the same elliptic function on the same domain. 
The Fourier expansions of the functions in $\bJ^2_*$ have been described in \ref{no:ellitptic-functions}. 
\label{no:fourier-n-variable}
\end{nolabel}

\begin{nolabel}
For $f(t,\tau) \in \bJ^2_k$ the function
\begin{equation} \label{eq:q-derivatives}
\left( (2 \pi i) \frac{d}{d\tau} - \zeta(t,\tau) \frac{d}{dt} \right) f(t,\tau),
\end{equation} 
is periodic in $t$ with periods $1$, $\tau$. It is not, however, a Jacobi form. Indeed if we write
\begin{align*}
D = (2 \pi i) \tfrac{d}{d\tau} - \zeta(t,\tau) \frac{d}{dt},
\end{align*}
then it follows from \cite[2.15]{libgober2009elliptic} that
\begin{equation}\label{eq:D-non-mod}
\bigl(D f\bigr) \left( \frac{t}{c\tau + d},\frac{a\tau+b}{c\tau+d} \right) = (c\tau + d)^{k+2} \bigl(D f\bigr)(t,\tau) + 2 \pi i k c(c\tau + d)^{k+1} f(t,\tau).
\end{equation}
Rather $Df \in \bJ^2_* \otimes_{\mring}\qmring$. 
If $F(z,q)$ is the Fourier series of $f(t,\tau)$ in the domain $|q|<|z| < 1$  
\begin{equation}
 \left[ (2 \pi i)^2 q \frac{d}{dq} + \left(P_1\left(z,q \right) + \pi i
\right)(2 \pi i) z \frac{d}{dz}  \right] F(z,q),
\label{eq:q-der-fourier}
\end{equation}
converges absolutely in the domain $|q| < |z| < 1 $ and is the Fourier series of $Df$.
\label{no:p1-mult-der}
\end{nolabel}
\begin{lem}
Let $f  = f(t_0,\dots,t_n, \tau) \in \cF_{n+1}$. Define $D f$ by 
\begin{equation} \label{eq:di-def}
Df(t_0,\ldots,t_n,\tau) = 2 \pi i \frac{\partial f}{\partial \tau }(t_0,\dots,t_n,\tau) + \sum_{i = 0}^{n-1}  \zeta(t_n - t_i) \frac{\partial f}{ \partial t_i}(t_0,\dots,t_n,\tau).
\end{equation}
Then $Df \in \cF_{n+1}$. 
\label{lem:deriv-sum}
\end{lem}
\begin{proof}
Fix $0 \leq j \leq n$, we take the total derivative with respect to $\tau$ of
\begin{align*}
f(t_0,\dots, t_j+\tau,\ldots,t_n,\tau) - f(t_0,\dots, t_j, \dots,t_n,\tau) = 0,
\end{align*}
obtaining
\begin{align*}
\frac{\partial f}{\partial \tau}(t_0,\dots, t_j+\tau,\ldots,t_n,\tau) -  \frac{\partial f}{\partial \tau}(t_0,\dots, t_i,\ldots,t_n,\tau) = -\frac{\partial f}{\partial t_j}(t_0,\dots, t_j+\tau,\ldots,t_n,\tau).
\end{align*}
Recall that $\zeta(t+\tau, \tau) = \zeta(t, \tau)-2\pi i$, by \eqref{eq:zeta-not-per}. So if we put
\begin{align*}
\widehat{f}(t_0,\ldots,t_n,\tau) = \sum_{i=0}^{n-1}
\zeta(t_n-t_i,\tau)\frac{\partial f}{\partial t_i}(t_0,\ldots,t_n,\tau),
\end{align*}
then for $0 \leq j \leq n-1$ we have
\begin{align*}
\widehat{f}(t_0,\dots, t_j+\tau,\ldots,t_n,\tau) &- \widehat{f}(t_0,\dots, t_j,\ldots,t_n,\tau) \\
&= 2\pi i\frac{\partial f}{\partial t_j}(t_0,\dots, t_j+\tau,\ldots,t_n,\tau).
\end{align*}
In the case $j = n$ we have 
\[
\begin{multlined}
 \widehat{f}(t_0,\dots,t_n + \tau, \tau) = \sum_{i = 0}^{n-1} \left( \zeta(t_n - t_i) - 2 \pi i \right) \frac{\partial f}{\partial t_i} (t_0,\dots,t_n,\tau) \\ =  \widehat{f}(t_0,\dots,t_n, \tau) -  2 \pi i \sum_{i=0}^{n-1} \frac{\partial f}{\partial t_i} (t_0,\dots,t_n, \tau)  = \widehat{f}(t_0,\dots,t_n, \tau) +  2 \pi i  \frac{\partial f}{\partial t_n} (t_0,\dots,t_n, \tau),
\end{multlined}
\]
where we used Lemma \ref{lem:4.sum-deriv} in the last equation. It follows that
\begin{align*}
\left[  2\pi i \frac{\partial}{\partial \tau} + \sum_{i=0}^{n-1}
\zeta(t_n-t_i,\tau)\frac{\partial}{\partial t_i} \right] f(t_0,\dots,
t_i,\ldots,t_n,\tau),
\end{align*}
is $\tau$-periodic. 
It is obviously also $1$-periodic. Hence $Df \in \cF_{n+1}$ as desired.
\end{proof}

\begin{cor} Let $f \in \cF_n$ and let $F(z_1,\dots,z_n,q)$ be its Fourier series in the domain $|q z_1| < |z_n| < \dots <|z_2| < |z_1|$. Then the series
\begin{align} 
\left( (2 \pi i)^2 q \frac{d}{dq} - \sum_{i=1}^{n-1} \left( P_1\left(
\frac{z_n}{z_i},q \right) + \pi i \right) 2 \pi i z_i \frac{d}{d z_i} \right)
F(z_1,\dots,z_n,q) ,
\label{eq:4.14b}
\end{align}
converges absolutely in the same domain, and is the Fourier series of $Df \in \cF_n$. 
\label{cor:deriv-sum}
\end{cor}
\begin{nolabel}Lemma \ref{lem:deriv-sum} can be interpreted in the following
way. 
Consider the semidirect product $\SL(2, \mathbb{Z}) \ltimes
\mathbb{Z}^{2n}$, where $\mathbb{Z}^{2n}$ consists of
$n$-copies of the standard $\SL(2, \mathbb{Z})$ representation. 
 The coarse moduli space of $n$-marked elliptic curves is the quotient of
$\mathbb{H} \times \mathbb{C}^n $ by the action of this group
given as follows. 
\[ 
\begin{aligned}
(k_1,l_1,\dots,k_n,l_n) \cdot (\tau, t_1,\dots,t_n) &= \left(\tau, t_1 + k_1 + l_1
\tau, \dots, t_n + k_n + l_n \tau \right), \\ 
\begin{pmatrix}
a & b\\  
c & d
\end{pmatrix} \cdot  \left( \tau, t_1,\dots,t_n \right) &= \left(
\frac{a \tau + b}{c \tau + d}, \frac{t_1}{c \tau + d}, \dots,
\frac{t_n}{c \tau + d}
\right). 
\end{aligned}
 \]
The vector fields $\tfrac{\partial}{\partial t_i}$, $i = 1,\dots,n-1$ are well
defined in this quotient, but the vector field $\tfrac{\partial}{\partial \tau}$
is not. Instead, the vector field 
\begin{equation} 
 D_\tau = 2 \pi i \frac{\partial}{\partial \tau} + \sum_{i = 1}^{n-1} \zeta(t_n -
t_i, \tau) \frac{\partial}{\partial t_i},
\label{eq:vector-field}
\end{equation}
is well defined in the open locus where the points are distinct $t_i - t_j
\not\in \mathbb{Z} + \tau \mathbb{Z}$. Notice that the non-trivial
commutators among these vector fields are 
\begin{equation}
\left[ D_\tau, \frac{\partial}{\partial t_i} \right] = - \owp_2(t_n - t_i)
 \frac{\partial}{\partial t_i}, \qquad i = 1,\dots,n-1. 
\label{eq:commutators-D}
\end{equation}
\label{no:vector-field}
\end{nolabel}
\section{First chiral homology groups}\label{sec:first.chiral}
In this section we define complexes computing chiral homology and their
duals in genus $1$. In \ref{no:complex-definition}--\ref{no:complex-def} we define the complex computing chiral homology of the universal elliptic curve. In  \ref{no:supports}--\ref{defn:chiral-homology-with-supports} we define the
complex that computes chiral homology on the universal elliptic curve with $n$
marked points. In \ref{no:moving-points} we study the connection obtained by moving the marked points. In \ref{no:connectionsa}--\ref{rem:extend-scalars} we describe the connection obtained as the base elliptic curve moves in the moduli space. We study the dual spaces of chiral homology in \ref{defn:definition-cohomologya} and their flat sections with respect to variation of the marked points in \ref{defn:cohomology-def-2b}. Finally we define flat sections of the dual of chiral homology in \ref{defn:chiral-cohomology-flat-sections}. 
\begin{nolabel}
Let $V$ be a vertex algebra. We define a complex
\begin{align*}
V^{\otimes 3} \otimes \bJ_*^3 \xrightarrow{d_2} V^{\otimes 2} \otimes \bJ_*^2
\xrightarrow{d_1} V \otimes \bJ_*^1,
\end{align*}
as follows. For $f(t_1, t_2, \tau) = f(t_1 - t_2,\tau) \in \bJ_*^2$ we put
\begin{equation}
d_1 \left( a \otimes b \otimes f(t_1, t_2) \right) = a_{(f)}b,
\label{eq:d1-def}
\end{equation}
where the $(f)$-product is as defined in \ref{no:f-product}. 
In this expression, and in the sequel, we omit explicit reference to $\tau$ and we abide by the conventions of Section \ref{no:expansions} and Example \ref{ex:expansion12}. 

Now we let $f(t_1,t_2,t_3) \in \bJ^{3}_*$ and we consider the Laurent series expansions
\begin{equation}
f(t_1,t_2,t_3)
= \sum_{k} f_{12k}(t_2, t_3) (t_1 - t_2)^k 
= \sum_{k} f_{13k}(t_2, t_3) (t_1 - t_3)^k
= \sum_{k} f_{23k}(t_1, t_3) (t_2 - t_3)^k,
\end{equation} 
where $f_{ijk} \in \bJ^2_*$ and the sum runs over $k \geq N_{ij}$ for some fixed $N_{ij} \in \mathbb{Z}$. For instance when $i=1$, $j=2$ we have
\begin{align*}
f(t_1, t_2, t_3) = f(t_2 + (t_1-t_2), t_2, t_3) = \sum_{k \geq {N_{12}}} f_{12k}(t_2, t_3) (t_1-t_2)^k.
\end{align*}
We now put
\begin{align}\label{eq:boundary}
\begin{split}
d_2 \left( a \otimes b \otimes c \otimes f(t_1,t_2,t_3)  \right)
= {} & \sum_{k} a \otimes b_{(k)}c \otimes f_{23k}(t_1,t_2) \\
&- \sum_k b \otimes a_{(k)}c \otimes f_{13k}(t_1,t_2) - \sum_k a_{(k)}b \otimes
c \otimes f_{12k}(t_1,t_2). 
\end{split}
\end{align}
\label{no:complex-definition}
The following lemma is a straightforward consequence of the Borcherds identity.
\end{nolabel}
\begin{lem} $d_1 \circ d_2 = 0$. 
\label{lem:square-zero}
\end{lem}
\begin{proof}
Let $f(t_1,t_2,t_3) \in \bJ^3_*$. In the notation of Example \ref{ex:expansion12} we have
\begin{align*}
\sum_k \bigl( a_{(k)}b \bigr)_{(f_{12k})} c = \res_{t_2-t_3} \res_{t_1-t_2}  Y\Bigl(  Y(a,t_1-t_2)b,t_2-t_3 \Bigr)c i_{t_2,t_2-t_3, t_1-t_2} f(t_1,t_2,t_3),
\end{align*}
and similar expressions for the other terms of \eqref{eq:boundary}. Thus
\begin{align}\label{eq:borcherds-res}
\begin{split}
& d_1 \left( d_2 \left(a \otimes b \otimes c \otimes f(t_1,t_2,t_3) \right) \right) \\
= {} & \res_{t_1-t_3} \res_{t_2-t_3} Y(a,t_1-t_3) Y(b,t_2-t_3) c \,\, i_{t_1,t_1-t_3,t_2-t_3} f( t_1,t_2,t_3) \\
&- \res_{t_2-t_3} \res_{t_1-t_3} Y(b,t_2-t_3) Y(a,t_1-t_3) c \,\, i_{t_2,t_2-t_3,t_1-t_3} f(t_1,t_2,t_3) \\
&- \res_{t_2-t_3} \res_{t_1-t_2} Y\left( Y(a,t_1-t_2)b,t_2-t_3 \right) c \,\, i_{t_2,t_2-t_3, t_1-t_2} f(t_1,t_2,t_3). 
\end{split}
\end{align}
The lemma follows from Borcherds identity \eqref{eq:borcherds-def} with $z=t_1-t_3$ and $w =t_2 -t_3$.
\end{proof}
\begin{lem} For any $a,b \in V$ and $f(t_1-t_2) \in \bJ^2_*$ we have
\begin{align}
\begin{split}
d_1 \left( L_{-1} a \otimes b \otimes f + a \otimes b \otimes f' \right) &= 0, \\
d_1 \left( a \otimes L_{-1} b \otimes f - a \otimes b \otimes
f'  \right) &=  L_{-1} d_1 (a \otimes b \otimes f) .
\end{split} 
\label{eq:derivatives-d1}
\end{align}
\label{lem:derivtives-d1}
\end{lem}
\begin{proof}
This is simply an application of \eqref{eq:integration-parts}.
\end{proof}

\begin{nolabel}
We now define $C_0 =V \otimes \bJ^1_*$ and
\begin{equation} \label{eq:c1-def}
C_1 = V^{\otimes 2}  \otimes \bJ^2_* \Bigl/ \Bigl< L_{-1} a \otimes b \otimes f
+ a \otimes b \otimes f'\Bigr>,
\end{equation}
and we define $C_2$ to be the quotient of $V^{\otimes 3} \otimes \bJ^3_*$ by the subspace spanned by the terms
\begin{align*} 
\begin{aligned}
L_{-1} a \otimes b \otimes c \otimes f +& a \otimes b \otimes c \otimes  \left( \partial f / \partial t_1 \right), \\
a \otimes L_{-1} b \otimes c \otimes f +& a \otimes b \otimes c \otimes  \left( \partial f / \partial t_2 \right), \\
a \otimes b \otimes c \otimes f(t_1,t_2,t_3) +& b \otimes a \otimes c \otimes f(t_2,t_1,t_3).
\end{aligned}
\end{align*}
The differentials $d_1$ and $d_2$ descend to define a complex $C_\bullet$ as we shall see below. Let us denote by $\pi$ the quotient map $V^{\otimes 2} \otimes \bJ^2_* \rightarrow C_1$. 
\end{nolabel}

\begin{lem} For any $a,b,c \in V$ and $f \in \bJ^3_*$ we have 
\begin{equation}
\begin{gathered}
\pi \circ d_2 \Bigl( L_{-1} a \otimes b \otimes c \otimes f(t_1,t_2,t_3) + 
a \otimes b \otimes c \otimes \frac{d}{dt_1} f(t_1,t_2,t_3)
\Bigr) = 0, \\
\pi \circ
d_2 \Bigl(  a \otimes L_{-1} b \otimes c \otimes f(t_1,t_2,t_3) + 
a \otimes b \otimes c \otimes \frac{d}{dt_2} f(t_1,t_2,t_3)
\Bigr) = 0, \\
{\multlinegap=120pt
\begin{multlined}
d_2 \Bigl(  a \otimes b \otimes L_{-1} c \otimes f(t_1,t_2,t_3) + 
a \otimes b \otimes c \otimes \frac{d}{dt_3} f(t_1,t_2,t_3)
\Bigr) = \\ 
  \left( L_{-1}^{(2)} + \frac{d}{d t_2} \right) d_2 \Bigl( a \otimes b \otimes c \otimes f(t_1,t_2,t_3) \Bigr),
\end{multlined}
}
\end{gathered}
\label{eq:d2-diff}
\end{equation}
where $L_{-1}^{(2)}$ is the operator $L_{-1}$ applied to the second factor of $V \otimes V \otimes \mathbb{J}^2_*$. 
\label{lem:d2-diff}
\end{lem}
\begin{proof}
If we write $g = \partial f / \partial t_1$ then the coefficients of the Laurent series expansions of $f$ and $g$ are related by
\begin{align*}  
g_{23k} = f'_{23k}, \quad g_{13,k-1} = k f_{13k} \quad \text{and} \quad g_{12,k-1} = k f_{12k}.
\end{align*}
Therefore
\begin{multline} d_2 \Bigl (a \otimes b \otimes c \otimes \frac{d}{dt_1}f(t_1,t_2,t_3)  \Bigr) = \sum_{k} a \otimes b_{(k)}c \otimes f'_{23k} \\ - \sum_k (k+1) b \otimes a_{(k)}c \otimes f_{13,k+1} - \sum_k (k+1) a_{(k)}b \otimes c \otimes f_{12,k+1}. 
\end{multline}
On the other hand, replacing $a$ by $L_{-1} a$ in \eqref{eq:boundary} and applying $\pi$, we have
\begin{multline}
\pi \circ d_2 \Bigl (L_{-1} a \otimes b \otimes c \otimes f(t_1,t_2,t_3)  \bigr) = \sum_{k} L_{-1} a \otimes b_{(k)}c \otimes f_{23k} \\ - \sum_k b \otimes (L_{-1}a)_{(k)}c \otimes f_{13k} - \sum_k (L_{-1}a)_{(k)}b \otimes c \otimes f_{12k} = 
\\ 
-\sum_{k} a \otimes b_{(k)}c \otimes f'_{23k}  + \sum_k k b \otimes a_{(k-1)}c \otimes f_{13k} + \sum_k k a_{(k-1)}b \otimes c \otimes f_{12k}. 
\label{eq:pid2}
\end{multline}
Thus the first equation of \eqref{eq:d2-diff} is proved. The second equation is proved in the same way. In order to prove the third equation we consider $h = \partial f/ \partial t_3$, then the coefficients of the Laurent series expansions of $f$ and $h$ are related by 
\[ h_{12k} =- f'_{12k}, \qquad h_{13k} = - \left( f'_{13k} + (k+1) f_{13,k+1}\right), \qquad h_{23k} = - \left(f'_{23k} + (k+1) f_{23,k+1}\right). \]
Therefore
\begin{equation}
\begin{multlined}
d_2 \left( a \otimes b \otimes c \otimes \frac{d}{d t_3} f(t_1,t_2,t_3) \right) = - \sum_{k} a \otimes b_{(k)}c \otimes \left( f'_{23k} + (k+1)f_{23,k+1} \right) \\ 
+ \sum_k b \otimes a_{(k)} c \otimes \left( f'_{13k} + (k+1) f_{13,k+1} \right) + \sum_{k} a_{(k)}b \otimes c \otimes f'_{12k} \\ 
= - \sum_k a \otimes b_{(k)}c \otimes f'_{23k} + \sum_{k} b \otimes a_{(k)}c \otimes f'_{13k} 
+ \sum_{k} a_{(k)}b \otimes c \otimes f'_{12k} \\
+  \sum_{k} a \otimes \left( L_{-1} b \right)_{(k)} c \otimes f_{23k} - \sum_{k} b \otimes \left( L_{-1}a \right)_{(k)} c \otimes f_{13k}.
\end{multlined}
\label{eq:d2-2}
\end{equation}
The sum of the first three terms equals $\tfrac{d}{d t_2} d_2\left( a \otimes b \otimes c \otimes f(t_1,t_2,t_3) \right)$. 
On the other hand replacing $c$ by $L_{-1}c$ in \eqref{eq:boundary} we have
\begin{equation}
\begin{multlined}
d_2 \Bigl( a \otimes b \otimes L_{-1} c \otimes f(t_1,t_2,t_3) \Bigr)  = \sum_{k} a \otimes b_{(k)} L_{-1} c \otimes f_{23k} \\ 
- \sum_k b \otimes a_{(k)}L_{-1}c \otimes f_{13k} - \sum_{k} a_{(k)}b \otimes L_{-1} c \otimes f_{12k} \\
= \sum_{k} a \otimes L_{-1} \left( b_{(k)}c \right) \otimes f_{23k} - \sum_k a \otimes \left(L_{-1}b\right)_{(k)} c \otimes f_{23k} \\ 
- \sum_{k} b \otimes L_{-1} \left( a_{(k)}c \right) \otimes f_{13k} + \sum_k b
\otimes \left( L_{-1} a \right)_{(k)} c \otimes f_{13k} - \sum_k a_{(k)}b
\otimes L_{-1}c \otimes f_{12k} .
\end{multlined}
\label{eq:d2-3}
\end{equation}
The sum of the first, third and last term of the right hand side equal $L_{-1}^{(2)} \left( d_2 (a \otimes b \otimes c \otimes  f)\right)$, the remaining terms cancel the last two terms of \eqref{eq:d2-2}.
\end{proof}
\begin{lem} We have the identity
\begin{equation}
\pi \circ d_2 \Bigl( a \otimes b \otimes c \otimes f(t_1,t_2,t_3) \Bigr) =- \pi
\circ d_2 \Bigl( b \otimes a \otimes c \otimes f(t_2,t_1,t_3) \Bigr).
\label{eq:d2sym}
\end{equation}
\label{lem:d2sym}
\end{lem}
\begin{proof}
Let $g(t_1,t_2,t_3) = f(t_2,t_1,t_3)$. Then the Laurent expansions of $g$ and of $f$ are related by
\begin{align*}
f_{23k} = g_{13k}, \quad f_{13k} = g_{23k}, \quad f_{12k} = \sum_{j \geq 0} (-1)^{k+j}  g^{(j)}_{12,k-j}. 
\end{align*}
The last one is obtained from 
\begin{align*}
g(t_2,t_1,t_3) = \sum_j f_{12j} (t_2,t_3) (t_1-t_2)^j = \sum_j (-1)^j
g_{12j}(t_1,t_3) (t_1 -t_2)^j,
\end{align*}
after multiplying by $(t_1-t_2)^{-k-1}$ and taking residues. By \eqref{eq:boundary} we have
\begin{equation} d_2 \Bigl (b \otimes a \otimes c \otimes g(t_1,t_2,t_3,\tau)  \bigr) = \sum_{k} b \otimes a_{(k)}c \otimes f_{13k}  - \sum_k a \otimes b_{(k)}c \otimes f_{23k} - \sum_k  b_{(k)}a \otimes c \otimes g_{12k}. 
\label{eq:boundary-twist}
\end{equation}
We use the skew-symmetry identity \eqref{eq:skew-symmetry} to write
\begin{multline*}
\pi \sum_k  b_{(k)}a \otimes c \otimes g_{12k}
= - \pi \sum_{k} \sum_{j \geq 0} (-1)^{k} \frac{(-L_{-1})^{j}}{j!} a_{(k+j)} b \otimes c \otimes g_{12k} = \\
= - \pi \sum_{k}\sum_{j \geq 0}  (-1)^{k} a_{(k+j)}b \otimes c \otimes g^{(j)}_{12k} = - \pi \sum_{k} a_{(k)} b \otimes c \otimes f_{12k}.
\end{multline*}
This proves the Lemma. 
\end{proof}
\begin{nolabel}Due to Lemmas \ref{lem:d2-diff} and \ref{lem:d2sym} we may define the following complex 
\begin{equation} \label{eq:c-complex}
C_2 \xrightarrow{d_2} C_1 \xrightarrow{d_1} C_0 \rightarrow 0.
\end{equation}
Notice that this is a complex of $\mathbb{M}_*[\partial]$-modules, where $\partial$ acts by $L_{-1}^{(k+1)} + \tfrac{d}{d t_{k+1}}$ on $C_k$. 

For $i=0, 1$ the homology $H_i(C_\bullet)$ is called the $i^\text{th}$ \emph{chiral homology group} of the vertex algebra $V$ in genus $1$. We will denote these groups by $H^{\text{ch}}_i(V)$, leaving the $\tau$ dependence implicit. They correspond to the chiral homology groups of the universal elliptic curve with coefficients in the chiral algebra associated to $V$ as defined in \cite{beilinsondrinfeld} and explained in \cite{eh2018}. 
$H^{ch}_0(V)$ is also known as the space of \emph{coinvariants} of $V$ in genus $1$ and its dual is closely related to the space of \emph{conformal blocks} of $V$ in genus one as defined and studied in \cite{zhu}. We will comment more on this below when we study the connection on this complex. 
\label{no:complex-def}
\end{nolabel}
\begin{nolabel}[Homology with supports] More generally we may define, for fixed $n \geq 1$, a complex
\begin{align}\label{eq:homol.with.supp}
V^{\otimes (n+2)} \otimes \bJ^{n+2}_* \xrightarrow{d_2} V^{\otimes (n+1)}
\otimes \bJ^{n+1}_* \xrightarrow{d_1} V^{\otimes n} \otimes \bJ^n_* \rightarrow
0,
\end{align}
as follows. Let $f(t_0,\dots,t_n) \in \bJ^{n+1}_*$ and let $1 \leq i \leq n$. We consider the Laurent series expansion
\begin{align*}
f(t_0,\dots,t_n) = \sum_k f_{ik}(t_1,\dots,t_n) (t_0 - t_i)^k.
\end{align*}
By Lemma \ref{lem:Laurent-exp} we have $f_{ik} \in \bJ^{n}_*$ for each $i=1,\ldots, n$ and for all integers $k$, and we define
\begin{equation}
d_1 \left( a\otimes a_1 \otimes \dots \otimes a_{n} \otimes f(t_0,\dots,t_n) \right) = \sum_{i=1}^n \sum_k a_1 \otimes \dots \otimes \left(a_{(k)}a_i\right) \otimes \dots \otimes a_n \otimes f_{ik}.
\label{eq:d1-general-def}
\end{equation}
For $f(u,v,t_1, \dots,t_{n}) \in \bJ^{n+2}_*$ we will consider similarly the following expansions for $i = 1,\ldots, n$:
\begin{equation}
\label{eq:5.12b}
\begin{multlined}
f(u,v,t_1,\dots,t_n) = 
\sum_{k} f_{0ik}(v,t_1,\dots,t_{n}) (u - t_i)^k, \\ 
= \sum_k f_{1ik} (u,t_1,\dots,t_{n}) (v - t_i)^k = \sum_k f_{00k}
(v,t_1,\dots,t_n) (u-v)^k.
\end{multlined}
\end{equation} 
We define
\begin{equation}
\begin{split}
d_2 \left( a \otimes b \otimes a_1 \otimes \dots \otimes a_n \otimes f \right)
&= \sum_{i = 1}^n \sum_k a \otimes a_1 \otimes \dots \otimes \left( b_{(k)} a_i \right) \otimes \dots \otimes a_n \otimes f_{1ik}(t_0,\dots,t_n) \\
& \quad - \sum_{i=1}^n \sum_k b \otimes a_1 \otimes \dots \otimes \left( a_{(k)}a_i \right) \otimes \dots \otimes a_n \otimes f_{0ik}(t_0,\dots,t_n) \\ & \quad - \sum_{k} a_{(k)} b \otimes a_1 \otimes \dots \otimes a_n \otimes f_{00k}(t_0,\dots,t_n). 
\label{eq:d2-general-def}
\end{split}
\end{equation}
\label{no:supports}
\end{nolabel}
\begin{lem} $d_1 \circ d_2 = 0$
\label{lem:complex-supports}
\end{lem}
\begin{proof}
The proof of this lemma is similar to that of Lemma \ref{lem:square-zero}. Let $f(u,v,t_1,\dots,t_n) \in \bJ^{n+2}_*$. We have
\begin{align*}
& d_1 \left( d_2 \left( a \otimes b \otimes a_1 \otimes \dots \otimes a_n \otimes f (u,v,t_1,\dots,t_n) \right) \right) \\
= {} & \sum_{i=1}^n \sum_{j =1}^{i-1} \res_{u-t_j} \res_{v-t_i} a_1 \otimes \dots \otimes Y(a,u-t_j)a_j \otimes \dots \otimes Y(b,v-t_i)a_i \otimes \dots \otimes a_n \otimes i_{u,u-t_j} i_{v,v-t_i} f 
\\
&+ \sum_{i=1}^n \res_{u-t_i} \res_{v-t_i} a_1 \otimes \dots \otimes Y(a,u-t_i)Y(b,v-t_i)a_i \otimes \dots \otimes a_n \otimes i_{t_i,u-t_i,v-t_i} f \\
&+ \sum_{i=1}^n \sum_{j =i+1}^{n} \res_{u-t_j} \res_{v-t_i} a_1 \otimes \dots \otimes Y(b,v-t_i)a_i \otimes \dots \otimes Y(a,u-t_j) a_j\otimes \dots \otimes a_n \otimes i_{u,u-t_j} i_{v,v-t_i} f  \\
&- \sum_{i=1}^n \sum_{j =1}^{i-1} \res_{v-t_j} \res_{u-t_i} a_1 \otimes \dots \otimes Y(b,v-t_j)a_j \otimes \dots \otimes Y(a,u-t_i)a_i \otimes \dots \otimes a_n \otimes i_{u,u-t_i} i_{v,v-t_j} f 
\\
&- \sum_{i=1}^n \res_{v-t_i} \res_{u-t_i} a_1 \otimes \dots \otimes
Y(b,v-t_i)Y(a,u-t_i)a_i \otimes \dots \otimes a_n \otimes i_{t_i,v-t_i,u-t_i} f
\stepcounter{equation}\tag{\theequation}   
\label{eq:square-zero-support} \\ 
&-\sum_{i=1}^n \sum_{j =i+1}^{n} \res_{v-t_j} \res_{v-t_i} a_1 \otimes \dots \otimes Y(a,u-t_i)a_i \otimes \dots \otimes Y(b,v-t_j)a_j \otimes \dots \otimes a_n \otimes i_{u,u-t_j} i_{v,v-t_i} f  \\
&- \sum_{i = 1}^n \res_{v-t_i} \res_{u-v}  a_1 \otimes \dots \otimes Y\bigl( Y(a,u-v)b,v-t_i \bigr)a_i \otimes \dots \otimes a_n \otimes i_{t_i,v-t_i,u-v} f
\\
= {} & \sum_{i=1}^n \res_{u-t_i} \res_{v-t_i} a_1 \otimes \dots \otimes Y(a,u-t_i)Y(b,v-t_i)a_i \otimes \dots \otimes a_n \otimes i_{t_i,u-t_i,v-t_i} f \\
&- \sum_{i=1}^n \res_{v-t_i} \res_{u-t_i} a_1 \otimes \dots \otimes Y(b,v-t_i)Y(a,u-t_i)a_i \otimes \dots \otimes a_n \otimes i_{t_i,v-t_i,u-t_i} f \\
&- \sum_{i = 1}^n \res_{v-t_i} \res_{u-v}  a_1 \otimes \dots \otimes Y\bigl(
Y(a,u-v)b,v-t_i \bigr)a_i \otimes \dots \otimes a_n \otimes i_{t_i,v-t_i,u-v} f,
\end{align*}
and this vanishes because of Borcherds identity \eqref{eq:borcherds-def}.
\end{proof}
From now on, it will be convenient to denote by $L_{-1}^{(i)}$ the operator $1 \otimes \cdots \otimes L_{-1} \otimes \cdots \otimes 1$ acting on $a_1 \otimes \ldots \otimes a_n \in V^{\otimes n}$ or $a_0 \otimes a_1 \otimes \ldots \otimes a_n \in V^{\otimes (n+1)}$ or $a \otimes b \otimes  a_1 \otimes \ldots \otimes a_{n} \in V^{\otimes (n+2)}$ with $L_{-1}$ acting in the component labelled `$i$'.
\begin{lem}
For $a_0, a_1, \ldots, a_n \in V$ and $f(t_0,t_1,\dots,t_n) \in \bJ^{n+1}_*$ we have
\begin{equation} \label{eq:5.9.1}
d_1 \left( \left[ L_{-1}^{(0)} + \frac{\partial}{\partial t_0} \right] a_0 \otimes \dots \otimes a_n \otimes f(t_0,\dots,t_n) \right) = 0,
\end{equation}
and for $j=1,2,\ldots,n$ we have
\begin{equation} \label{eq:5.9.2}
d_1 \left( \left[ L_{-1}^{(j)} + \frac{\partial}{\partial t_j} \right] a_0 \otimes \dots \otimes a_n \otimes f(t_0,\dots,t_n) \right) 
= \left( L_{-1}^{(j)} + \frac{\partial}{ \partial t_j} \right)  d_1 \left( a_0 \otimes \dots \otimes a_n \otimes f(t_0,\dots,t_n) \right).
\end{equation}
\label{lem:d1-diff-trading}
\end{lem}

\begin{proof}
Recalling the Laurent series expansion $f = \sum_k f_{0jk} (t_1,\dots,t_n) (t_0-t_j)^k$ and the definition of $d_1$, we have
\begin{align*}
d_1\left( L_{-1} a_0 \otimes \dots \otimes a_n \otimes f(t_0,\dots,t_n) \right) 
&= \sum_{j=1}^n \sum_k a_1 \otimes \dots \otimes \left( (L_{-1} a_0)_{(k)} a_j \right) \otimes \dots \otimes a_n \otimes f_{jk} \\
&= -\sum_{j=1}^n \sum_k a_1 \otimes \dots \otimes \left(  a_{0(k-1)} a_j \right) \otimes \dots \otimes a_n \otimes k f_{jk} \\
&= - d_1\left(a_0 \otimes \dots \otimes a_n \otimes \frac{\partial f}{\partial t_0}(t_0,\dots,t_n)\right),
\end{align*}
and the first claim follows. Similarly the second claim follows from the equality
\begin{multline*}
\sum_k a_1 \otimes \dots \otimes \bigl( a_{0(k)}L_{-1} a_j \bigr) \otimes \dots
\otimes a_n \otimes  f_{0jk} (t_0 -t_j)^k\\   =  \sum_k  a_1 \otimes \dots
\otimes  L_{-1} \bigl(  a_{0(k)} a_j \bigr) \otimes \dots \otimes a_n \otimes
f_{0jk} (t_0-t_j)^k \\ - \sum_k a_1 \otimes \dots \otimes \bigl( a_{0 (k-1)} a_j
\bigr) \otimes \dots \otimes a_n \otimes k f_{0jk}  (t_0-t_j)^k.
\end{multline*}

\end{proof}
We define $\tilde{C}_0$ to be $V^{\otimes n} \otimes \bJ^{n}_*$ 
and $\tilde{C}_1$ to be the quotient of $V^{\otimes (n+1)}\otimes \bJ^{n+1}_*$ by the linear span of the terms
\begin{align*} 
\left( L_{-1}^{(0)} + \frac{\partial}{\partial t_0} \right) a_0 \otimes \dots \otimes a_n \otimes f(t_0,\dots,t_n).
\end{align*}
We denote by $\pi$ the corresponding quotient map $V^{\otimes n+1} \otimes \bJ^{n+1}_* \rightarrow \tilde{C}_1^{n}$. 
We notice that $d_1 : \tilde{C}_1 \rightarrow \tilde{C}_0$ is well defined by \eqref{eq:5.9.1} and it is a map of $k[\partial_1,\dots,\partial_n]$--modules by \eqref{eq:5.9.2}. In addition, recalling that we view $J^{n+1}_*$ as a $J^n_*$--module as in \ref{no:jacobi-forms} we see that the map $d_1$ is $J^n_*$--linear. These two actions make the complex $\tilde{C}_1 \rightarrow \tilde{C}_0$ a complex of $D$-modules. 
The following lemma is proved in the same way as above.
\begin{lem}
Let $a,b, a_1 \ldots a_n \in V$ and $f(u,v,t_1,\dots,t_{n}) \in \bJ^{n+2}_*$. Then
\begin{gather}
\pi \circ d_2 \left( L_{-1} a \otimes b \otimes a_1 \otimes \dots \otimes a_n \otimes f \right) + \pi \circ d_2 \left( a \otimes b \otimes a_1 \otimes \dots \otimes a_n \otimes \frac{d}{d u} f(u,v,t_1,\dots,t_n) \right) = 0 \\ 
\pi \circ d_2 \left( a \otimes L_{-1} b \otimes a_1 \otimes \dots \otimes a_n \otimes f \right) + \pi \circ d_2 \left( a \otimes b \otimes a_1 \otimes \dots \otimes a_n \otimes \frac{d}{d v} f(u,v,t_1,\dots,t_n) \right) = 0 \\ 
\begin{multlined}
d_2 \left( \left( L_{-1}^{(j)} + \frac{d}{d t_j} \right) a \otimes b \otimes a_1 \otimes \dots \otimes a_n \otimes f(u,v,t_1,\dots,t_n) \right) \\
= 
 \left( L_{-1}^{(j)} + \frac{d}{d t_j} \right) d_2 \left( a \otimes b \otimes a_1 \otimes \dots \otimes a_n \otimes f(u,v,t_1,\dots,t_n) \right), \qquad 1 \leq j \leq n.
\end{multlined}
\end{gather}
\label{lem:d2-diff-trading}
\end{lem}
We also have as in Lemma \ref{lem:d2sym}:
\begin{lem} Let $a, b, a_1, \ldots a_n \in V$ and $f(u,v,t_1,\dots,t_{n}) \in \bJ^{n+2}_*$. Then
\[ \pi d_2\bigl(a \otimes b \otimes a_1 \otimes \dots \otimes a_n \otimes f(u,v,t_1,\dots,t_{n}) \bigr) +   
\pi  d_2\bigl(b \otimes a \otimes a_1 \otimes \dots \otimes a_n \otimes f(v,u,t_1,\dots,t_{n}) \bigr) = 0. \]
\label{lem:d2-sym-general}
\end{lem}
\begin{nolabel} \label{no:symmetric-group}
The complex \eqref{eq:homol.with.supp} carries an action of the symmetric group $S_n$. Indeed for $\sigma \in S_{n}$ a permutation we have
\begin{align*}
 \sigma \bigl( a_1 \otimes \dots \otimes a_n \otimes f(t_1,\dots,t_n) \bigr) &= a_{\sigma(1)} \otimes \dots \otimes a_{\sigma(n)} \otimes f(v, t_{\sigma(1)}, \dots, t_{\sigma(n)}),\\
 \sigma \bigl( a \otimes a_1 \otimes \dots \otimes a_n \otimes f(v, t_1,\dots,t_n) \bigr) &= a \otimes a_{\sigma(1)} \otimes \dots \otimes a_{\sigma(n)} \otimes f(v, t_{\sigma(1)}, \dots, t_{\sigma(n)}), \\
 \sigma \bigl( a \otimes b \otimes a_1 \otimes \dots \otimes a_n \otimes f(u,v, t_1,\dots,t_n) \bigr) &= a \otimes b \otimes a_{\sigma(1)} \otimes \dots \otimes a_{\sigma(n)} \otimes f(u,v, t_{\sigma(1)}, \dots, t_{\sigma(n)}). 
\end{align*}
It is immediate that the differentials $d_1$ and $d_2$ commute with this action, making the complex \eqref{eq:homol.with.supp} a complex in the category of $S_n$--modules.
\end{nolabel}
\begin{defn}[of chiral homology] Let $V$ be a vertex algebra and $n \geq 1$. 
\begin{enumerate}
\item Let $C_0^n$ be $V^{\otimes n} \otimes \bJ^n_*$.
\item Let $C_1^n$ be the quotient of $V^{\otimes (n+1)} \otimes \bJ^{n+1}_*$ by the linear span of the terms
\begin{equation}
\label{eq:c1-quotb2}
L_{-1} a_0 \otimes \dots \otimes a_{n} \otimes f(t_0, \dots,t_{n}) + a_0 \otimes
\dots \otimes a_n \otimes \frac{d}{d t_0} f(t_0,\dots,t_n).
\end{equation}
\item Let $C_2^n$ be the quotient of $V^{\otimes (n+2)} \otimes \bJ^{n+2}_*$ by the linear span of the terms
\begin{equation} \label{eq:c2-quot}
\begin{gathered}
a \otimes b \otimes a_1 \otimes \dots \otimes a_n \otimes f(u,v,t_1,\dots,t_n) + b \otimes a \otimes a_1 \otimes \dots \otimes a_n \otimes  f(v,u,t_1,\dots,t_n), \\
L_{-1} a \otimes b \otimes a_1 \otimes \dots \otimes a_{n} \otimes f(u,v ,t_1 \dots,t_{n}) + a \otimes b \otimes a_1 \dots \otimes a_n \otimes \frac{d}{d u} f(u,v,t_1,\dots,t_n).
\end{gathered}
\end{equation}
\end{enumerate}
Let 
\begin{equation}  C^n_2 \xrightarrow{d_2} C^n_1 \xrightarrow{d_1} C^n_0 \rightarrow 0, 
\label{eq:homology-def}
\end{equation}
be the complex with differentials $d_1$ and $d_2$ as in \eqref{eq:d1-general-def} and \eqref{eq:d2-general-def} above. For $i=0, 1$ we define $i^{\text{th}}$ \emph{chiral homology} of $V$ in
genus $1$, with coefficients in $V^{\otimes n}$ to be the homology
$H_i(C_\bullet^n)$. We denote this by $H^{\text{ch}}_i(V^{\otimes n})$.
\label{defn:chiral-homology-with-supports}
\end{defn}
\begin{nolabel}[Moving the points]\label{no:moving-points} The complex $C^n_\bullet$ is a complex of
$\mathbb{J}^n_*$--modules and a complex of
$k[\partial_1,\dots,\partial_n]$--modules. The action of $\mathbb{J}^n_*$ is by
multiplication on the right and the action of $\partial_j$ is given by
\begin{equation} \label{eq:quotient-moving}
\begin{aligned}
 \left( L_{-1}^{(j)} + \frac{d }{d t_j} \right) a_1 \otimes \dots \otimes a_n
\otimes f(t_1,\dots,t_n), && \text{in degree $0$}, \\
 \left( L_{-1}^{(j)} + \frac{d }{d t_j} \right) a_0 \otimes a_1 \otimes \dots \otimes a_n
\otimes f(t_0,t_1,\dots,t_n), && \text{in degree $1$}, \\
 \left( L_{-1}^{(j)} + \frac{d }{d t_j} \right) a \otimes b \otimes a_1 \otimes \dots \otimes a_n
\otimes f(u,v,t_1,\dots,t_n), && \text{in degree $2$},
\end{aligned}
\end{equation} 
where, in each case, $j = 1, \ldots, n$.
These two actions are compatible making
$C^n_\bullet$ a complex of $D$-modules. We let $H^{\ch}_i(V^{\otimes
n})_{\text{tr}}$ be the coinvariants by this action, namely the quotient of
$H^{\ch}_i(V^{\otimes n})$ by elements of the form \eqref{eq:quotient-moving}. 
\end{nolabel}
\begin{nolabel} In the definition of $H^{ch}_i \left( V^{\otimes n}
\right)_{\text{tr}}$ given in \ref{no:moving-points} above, the same result is obtained if we quotient the terms
\eqref{eq:quotient-moving} for $j = 1,\dots,n-1$ rather than $j = 1,\dots,n$.
Indeed let us write $\tilde{H}^{ch}_i(V^{\otimes n})_{\text{tr}}$ for the quotient of
$H^{\text{ch}}_i(V^{\otimes n})$ by \eqref{eq:quotient-moving}
for $j = 1,\dots,n-1$. We have obvious surjections $\pi_i:
\tilde{H}^{\text{ch}}_i\left( V^{\otimes n} \right)_{\text{tr}} \twoheadrightarrow
H^{\text{ch}}_i \left( V^{\otimes n} \right)_{\text{tr}}$. 

Let $\alpha = a_1 \otimes \dots \otimes a_{n} \otimes f(t_1,\dots,t_n, \tau) \in
C^n_0$. Put $g(t_0,\dots,t_n, \tau) = f(t_1,\dots,t_n, \tau)$ so that $g$ is
independent of $t_0$. Put $\beta = \omega \otimes a_1 \otimes \dots \otimes a_n
\otimes g \in C^n_1$. We have 
\begin{equation*}
d_1 \beta = \sum_{i = 1}^n a_1 \otimes \dots \otimes L_{-1} a_i \otimes \dots
\otimes a_n \otimes f(t_1,\dots,t_n, \tau) \in C^n_0.
\end{equation*}
Therefore in $\tilde{H}^{\text{ch}}_0(V^{\otimes n})_{\text{tr}}$ we
have 
\begin{multline*}
 a_1 \otimes \dots a_{n-1} \otimes L_{-1} a_n \otimes f = - \sum_{i = 1}^{n-1} a_1
\otimes \dots \otimes L_{-1} a_i \otimes \dots \otimes a_n \otimes f =  \\ 
\sum_{i=1}^{n-1} a_1 \otimes \dots \otimes a_n \otimes \frac{\partial}{\partial
t_i} f = - a_1 \otimes \dots \otimes a_n \otimes \frac{\partial}{\partial t_n}
f,
\label{eq:htilde0}
\end{multline*}
where in the last equality we used Lemma \ref{lem:4.sum-deriv}. Equivalently
\[ \left( L_{-1}^{(n)} + \frac{\partial}{\partial t_n} \right) \alpha = 0. \]It follows that
$\pi_0$ is an isomorphism. 

Similarly let $\alpha = a_0 \otimes \dots \otimes a_{n+1} \otimes
f(t_0,\dots,t_n, \tau) \in C^n_1$,
be a cycle, that is $d_1 \alpha = 0$. Put $g(u,v,t_1,\dots,t_n, \tau) = f(v,
t_1,\dots,t_n, \tau)$, so that $g$ is independent of $u$ and let 
$ \beta = \omega \otimes a_0 \otimes a_1 \otimes \dots \otimes a_n \otimes g
\in C^n_2$.
We have 
\[ d_2 \beta = - \sum_{i = 0}^{n-1} a_0 \otimes \dots \otimes L_{-1} a_i \otimes
\dots \otimes a_n \otimes f(t_0,\dots,t_n, \tau) \in C^n_1, \]
where we have used that $d_1 \alpha = 0$ so that the first term in the right hand side of
\eqref{eq:d2-general-def} vanishes.  Therefore in $\tilde{H}^{\text{ch}}_1\left(
V^{\otimes n} \right)_{\text{tr}}$ it holds
\begin{multline*}
 a_0 \otimes \dots a_{n-1} \otimes L_{-1} a_n \otimes f = - \sum_{i = 0}^{n-1} a_0
\otimes \dots \otimes L_{-1} a_i \otimes \dots \otimes a_n \otimes f =  \\ 
\sum_{i=0}^{n-1} a_0 \otimes \dots \otimes a_n \otimes \frac{\partial}{\partial
t_i} f = - a_0 \otimes \dots \otimes a_n \otimes \frac{\partial}{\partial t_n}
f.
\label{eq:htilde0}
\end{multline*}
It follows that $\pi_1$ is an isomorphism.
In the case of $n=1$ we obtain in particular
\begin{equation}
\label{eq:one-coinv-does-nothing}
H^{\text{ch}}_i(V) \cong H^{\text{ch}}_i(V)_{\text{tr}}, \qquad i = 0,1.
\end{equation}
\label{no:one-coinv-does-nothing}
\end{nolabel}
\begin{rem}\label{nolabel:A.C.comparison}
The homology groups $H^{\text{ch}}_i(V) \cong H^{\text{ch}}_i(V)_{\text{tr}}$ recover
the chiral homology (of degree $i=0, 1$) of the torus with coefficients in the
translation invariant chiral algebra $\cA_V$ associated with the conformal
vertex algebra $V$ \cite{beilinsondrinfeld,eh2018}.
To compare the construction above with that given in \cite[Proposition 7.13]{eh2018} we temporarily denote by $\mathring{\mathbb{J}^n_*}$ the ring of mermorphic elliptic functions $f(t_1, \ldots, t_n, \tau)$ with Jacobi dependence on the modular parameter $\tau$ and with possible poles at $t_i = t_j$ for $i \neq j$ and $t_i = 0$ for $i=1,\ldots n$. Due to the group structure of the elliptic curve $E_\tau =
\mathbb{C} / \Z + \Z \tau$ there is no essential difference between $\mathring{\mathbb{J}^n}_*$ and $\mathbb{J}^{n+1}_*$, indeed the map $g \mapsto f$ defined by $f(t, \ldots, t_n) = g(t_1-t,\ldots, t_n-t)$ is a ring isomorphism $\mathring{\mathbb{J}^n}_* \rightarrow \mathbb{J}^{n+1}_*$.


We specialize equations \eqref{eq:d1-general-def} and \eqref{eq:d2-general-def} to the case of $n=1$ and we identify $f(t_0, t_1) = g(t_0-t_1) \in \mathbb{J}^{2}_*$ with $g(t_0) \in \mathring{\mathbb{J}^1}_*$ and $f(u, v, t_1) = g(u-t_1, v-t_1) \in \mathbb{J}^{3}_*$ with $g(u, v) \in \mathring{\mathbb{J}^2}_*$ as above. With these identifications the differentials may be rewritten as:
\begin{align}\label{eq:d1.in.A}
d_1\left( f(t_0) \cdot a \otimes a_1 \right) = a_{(f)}a_1,
\end{align}
and
\begin{align}\label{eq:d2.in.A}
\begin{split}
d_2\left( f(u, v) \cdot a \otimes b \otimes a_1 \right)
= {} & \res_{v} f(u, v) a \otimes b(v) a_1 - \res_{v - u} f(v, u) a(v-u)b \otimes a_1 \\
&- \res_{v} f(v, u) b \otimes a(v)a_1.
\end{split}
\end{align}
In \cite[Proposition 7.13]{eh2018} we have written down a complex $A^\bullet$ which computes chiral homology $H^{\text{ch}}_i(E_\tau, \cA_V)$ for $i=0, 1$ for the elliptic curve $E_\tau = \mathbb{C} / \Z + \Z \tau$ with coefficients in the the translation invariant chiral algebra $\cA_V$ over $E_\tau$ associated with the conformal vertex algebra $V$. Comparison of $A^\bullet$ with $C_\bullet^{n=1}$, particularly with equations \eqref{eq:d1.in.A} and \eqref{eq:d2.in.A}, reveals that
\[
H_i^{\text{ch}}(V) \cong H^{\text{ch}}_i(X, \cA_V),
\]
for $i = 0, 1$. In fact there are two small differences between $A^\bullet$ and $C_\bullet^{n=1}$: the first is that coinvariants by the action of $S_2$ are omitted from the definition of $A^\bullet$ in \cite{eh2018}, but this does not alter the homology groups in degrees $0$ and $1$. The second difference is conventional, with $A^\bullet$ being a cohomological complex concentrated in non positive degrees and $C^n_\bullet$ a homological complex concentrated in non negative degrees.
\end{rem}
\begin{nolabel}[Connections] \label{no:connectionsa}
We extend scalars in the complex $C^n_*$ from the ring $\bJ_*^n$ of Jacobi forms
to the ring $\cF_n$ of meromorphic elliptic functions. Recall from \ref{no:jacobi-forms} that we have chosen a $\mathbb{J}^n_*$-module structure on
$\mathbb{J}^{n+k}_*$ and therefore on $\cF_{n+k}$. We consider $C^n_k
\otimes_{\bJ^n_*} \cF_{n+k}$. In addition to the action of
$\tfrac{\partial}{\partial t_i}$ for $i = 1,\dots,n$, these complexes carry a
canonical action of the vector field $D_\tau$ of \eqref{eq:vector-field}. Define $\nabla_\tau : C^n_0 \otimes_{\bJ^n_*} \cF_n \rightarrow C^n_0 \otimes_{\bJ^n_*} \cF_n$ by
\begin{equation}
\begin{split}
\nabla_\tau \left( a_1 \otimes \dots \otimes a_n \otimes f \right)
= {} & a_1 \otimes \dots \otimes a_n \otimes D f \\
&+ a_1 \otimes \dots \otimes a_{n-1} \otimes \left( \omega_{(-1) }a_n \right)  \otimes f \\ 
&- \sum_{k \geq 1} a_1 \otimes \dots \otimes a_{n-1} \otimes \left( \omega_{(2k-1)} a_n \right) \otimes g_{2k}(\tau) f \\
&- \sum_{i=1}^{n-1} \sum_{m \geq 0} a_1 \otimes \dots \otimes \left(
\omega_{(m+1)} a_i \right) \otimes \dots \otimes a_n \otimes \owp_{m+2}(t_n-t_i) f.
\end{split}
\label{eq:connection_c0}
\end{equation}
where $Df = \left(2 \pi i \tfrac{\partial}{\partial \tau} + \sum_{i=1}^{n-1} \zeta(t_n - t_i) \tfrac{\partial}{\partial t_i} \right) f$ is as defined in Lemma \ref{lem:deriv-sum}. 
Every function appearing on the right hand side of \eqref{eq:connection_c0} is elliptic
therefore  $\nabla_\tau \in \End \left(C_0^n \otimes_{\mathbb{J}^n_*} \cF_n \right)$ is well defined.

In the algebro-geometric setting, the chiral homology acquires the structure of a twisted $D$-module, indeed a vector bundle with projectively flat connection, over the moduli space of curves with marked points. In our setting this
connection is actually flat and it corresponds to the following
\begin{lem*} The operators $\nabla_\tau$ and $L_{-1}^{(j)} + \tfrac{d}{dt_j}$
satisfy, for $1 \leq j \leq n-1$, the same commutation relations as in
\eqref{eq:commutators-D}: 
\begin{equation}
\left[\left( L_{-1}^{(j)} + \frac{d}{d t_j}  \right), \nabla_\tau \right] = \owp_2(t_n-t_j) \left( L_{-1}^{(j)} + \frac{d}{d t_j}  \right).
\label{eq:non-flat-connection}
\end{equation}
\end{lem*}
\begin{proof}
Follows easily from $\left[ \tfrac{d}{dt_j},D \right] =  \owp_2(t_n - t_j) \tfrac{d}{d t_j}$ and $[L_{-1},L_{m}] = -(m+1)L_{m-1}$ for $m \geq 0$. 
\end{proof}

We let $H^{(0)}$ denote the endomorphism of $C_0^n$ defined by
\begin{align}\label{eq:H0.op.def}
H^{(0)} \left( a_1 \otimes \dots \otimes a_n \otimes f(t_1,\dots,t_n, \tau) \right) = d_1 \Bigl( \omega \otimes a_1 \otimes \dots \otimes a_n \otimes f(t_1,\dots,t_n, \tau) \zeta(t_n-t_0) \Bigr),
\end{align}
here abusing notation since $f(t_1,\dots,t_n, \tau) \zeta (t_0 - t_n) \notin \cF_{n+1}$ in general. A direct calculation shows that
\[
\nabla_\tau = 2 \pi i \frac{\partial}{\partial\tau} + \sum_{i=1}^{n-1}
\zeta(t_n - t_i) \left[ L_{-1}^{(i)} + \frac{\partial}{\partial t_i} \right]-
H^{(0)} .
\]
The non-ellipticity of $\zeta$ and of $\partial f / \partial \tau$ cancel due to \eqref{eq:q-derivatives}. In the right hand side of \eqref{eq:connection_c0} the terms $\zeta$ and $\partial f / \partial \tau$ are combined in the operator $D$.
\end{nolabel}
\begin{nolabel}
Similarly we define $\nabla_\tau$ on $C^n_1 \otimes_{\bJ^n_*} \cF_{n+1}$ by 
\begin{align}
\begin{split}
\nabla_\tau \left( a_0 \otimes\dots \otimes a_n \otimes f \right)
= {} &  a_0 \otimes \dots \otimes a_n \otimes Df (t_0,\dots,t_n,\tau) \\
&+ a_0 \otimes \dots \otimes \omega_{(-1)} a_n \otimes f - \sum_{k \geq 1} a_0 \otimes \dots \otimes \omega_{(2k-1)}a_n \otimes g_{2k}(\tau) f \\
&+  a_0 \otimes \dots \otimes a_n \otimes  \owp_2(t_n-t_0)   f \\ 
&- \sum_{i=0}^{n-1} \sum_{m \geq 0}  a_0 \otimes a_1 \otimes \dots \otimes
\omega_{(m+1)} a_i \otimes \dots \otimes a_n \otimes \owp_{m+2}(t_n-t_i) f \\ 
&- \sum_{i=1}^n \sum_k \omega \otimes a_1 \otimes \dots \otimes a_{0(k)} a_i \otimes \dots \otimes a_n \otimes f_{0ik}(t_1,\dots,t_n,\tau) \zeta(t_0 - t_n,\tau).
\end{split}
\label{eq:connection_2}
\end{align}

When $a_0 \otimes \dots \otimes a_n \otimes f$ is a cycle (that is if it lies in
the kernel of $d_1$) then the last term of \eqref{eq:connection_2} vanishes. The
remaining terms are manifestly elliptic so that the image of $\nabla_\tau$ is
indeed in $C_1^n \otimes_{\mathbb{J}^n_*} \cF_{n+1}$. 
For elements of $\ker d_1$ we have, in the same way as in
\eqref{eq:non-flat-connection} the commutation relations
\eqref{eq:commutators-D}:
\begin{equation}
\left[\left( L_{-1}^{(j)} + \frac{d}{d t_j}  \right), \nabla_\tau \right] = \owp_2(t_n-t_j) \left( L_{-1}^{(j)} + \frac{d}{d t_j}  \right).
\label{eq:non-flat-connection-3}
\end{equation}
for $1 \leq j \leq n-1$. Meanwhile for $j=0$ we have
\begin{equation} \label{eq:j0commut} 
\left[\left( L_{-1}^{(0)} + \frac{d}{d t_0}  \right), \nabla_\tau \right] =
\left( L_{-1}^{(0)} + \frac{d}{d t_0} \right) \owp_2(t_n - t_0). 
\end{equation}
It follows that $\nabla_\tau$ descends to the quotient by elements of the form
\eqref{eq:c1-quotb2}. 

In general we can write
\begin{equation}
\nabla_\tau =  2 \pi i \frac{\partial}{\partial\tau} + \sum_{i=0}^{n-1}
\zeta(t_n - t_i) \left[ L_{-1}^{(i)} + \frac{\partial}{\partial t_i} \right] -
H^{(1)} + \owp_2(t_n-t_0),
\label{eq:h1-generala}
\end{equation}
where $H^{(1)}$ is defined by
\begin{align}\label{eq:H1.op.def}
H^{(1)}\left( a_0 \otimes \dots \otimes a_{n} \otimes f(t_0,\dots,t_{n}, \tau) \right) &= d_2 \Bigl( \omega \otimes  a_0 \otimes \dots \otimes a_{n} \otimes f(t_0, t_1, \dots,t_{n}, \tau) \zeta(t - t_{n}, \tau) \Bigr),
\end{align}
where on the right hand side the function $f \cdot \zeta$ is considered as a function on the $n+3$ variables $t, t_0, t_1,\dots,t_{n}, \tau$. 
Once again the non-ellipticity of $\partial f / \partial \tau$ is corrected by the terms containing $\zeta$, making the elliptic function $Df$ appear in the right hand side of \eqref{eq:connection_2}. 

Notice that by commuting $\zeta(t_n-t_0)$ with $L_{-1}^{0} + \tfrac{d}{dt_0}$ in
\eqref{eq:h1-generala}, it is seen that $\nabla_\tau$ has the following simple
form in the quotient $C^n_1$:
\begin{equation}
\nabla_\tau =  2 \pi i \frac{\partial}{\partial\tau} + \sum_{i=1}^{n-1}
\zeta(t_n - t_i) \left[ L_{-1}^{(i)} + \frac{\partial}{\partial t_i} \right] -
H^{(1)} 
\label{eq:h1-generalb}
\end{equation}
\label{no:connections}
\end{nolabel}
\begin{nolabel}
We record here some particular cases of the operators $\nabla_\tau$ when $n=1$. Equation \eqref{eq:connection_c0} reduces to 
\begin{equation}
\nabla_\tau \bigl( a \otimes f \bigr) = a \otimes \frac{d}{d\tau} f  + \omega_{(-1)}a \otimes f - \sum_{k \geq
1} \omega_{(2k-1)}a \otimes g_{2k}(\tau) f .
\label{eq:nabla-0-n=1}
\end{equation}
Similarly, equation \eqref{eq:connection_2}, applied to a cycle $a \otimes b \otimes f$ reduces to
\begin{equation}
\begin{split}
\nabla_\tau \bigl(a \otimes b \otimes f(t_1-t_0,\tau) \bigr) &=  a \otimes b \otimes Df (t_1 - t_0, \tau)
+a \otimes \omega_{(-1)}b \otimes f(t_1-t_0,\tau) 
\\ & \quad - \sum_{k \geq 1} a \otimes \omega_{(2k-1)}b \otimes g_{2k}(\tau) f(t_1-t_0,\tau) 
\\ & \quad+ a \otimes b \otimes \owp_2(t_1 - t_0) f(t_1-t_0,\tau) 
\\ & \quad - \sum_{m \geq 0} \omega_{(m+1)}a \otimes b \otimes \owp_{m+2}(t_1 -
t_0) f(t_1-t_0,\tau).
\label{eq:nabla-1-n=1}
\end{split}
\end{equation}
\end{nolabel}
\begin{prop}
The connection operator induces well defined endormorphisms of $H^{\text{ch}}_k(V^{\otimes n})$ for $k=0,1$. 
\label{prop:dnabla-kills}
\end{prop}
\begin{proof}
We want to show that $\nabla_\tau$ preserves the spaces of cycles and
boundaries. In degree $0$ it is therefore enough to prove that $\nabla_\tau$
preserves boundaries. It is instructive to do the case $n=1$ first, the proof
for general $n$ follows the same pattern. So let $\alpha = a \otimes b \otimes f
\in C^{n=1}_1$. We have 
\[ d_1 \alpha = a_{(f)}b = \sum_k a_{(k)}b \otimes f_k, \quad \text{where }\quad f =
\sum_{k} f_k (t_0-t_1)^k, \]
and 
\[ \nabla_\tau d_1 \alpha = \sum_k a_{(k)}b \otimes \frac{d}{d\tau} f_k +
\sum_{k} \omega_{(-1)} a_{(k)}b \otimes f_k - \sum_{j \geq 1} \sum_{k}
\omega_{(2j-1)} a_{(k)}b \otimes g_{2j}(\tau) f_k. \]
Using Borcherds identity \eqref{eq:borcherds-def} we express the right hand side as 
\[
\sum_k a_{(k)}b \otimes \frac{d}{d\tau} f_k + \sum_k a_{(k)}
\omega_{(\zeta)}b \otimes f_k +  \left(L_{-1} a \right)_{(\zeta f)} b + \left(
\omega_{(1)}a  \right)_{(\owp_2 f)}b +  \sum_{m \geq 2} \left( \omega_{(m)} a
\right)_{(\wp_{m+1} f)} b
\]
The third term equals $a_{( \owp_2 f)}b -a_{(\zeta f')}b$ by
\eqref{eq:integration-parts}. Collecting we obtain
\begin{equation}
\nabla_\tau d_1 \alpha = d_1 \Bigl( a \otimes b \otimes Df + a \otimes
\omega_{(\zeta)}b \otimes f + a \otimes b \otimes 
\owp_2 f + \sum_{m \geq 1} \omega_{(m)}a \otimes b \otimes \owp_{m+1} f \Bigr).
\label{eq:n=1-boundaries-deg0}
\end{equation}
Therefore $\nabla_\tau$ preserves homology in degree $0$ in the case $n=1$. The
general case is proved in the same fashion by an application of Borcherds
identity.  Let $\alpha = a \otimes a_1 \dots \otimes a_n \otimes f(t_0,\dots,t_n) \in
C^n_1$. We have 
\[ 
d_1 \alpha = \sum_{i = 1}^n  \sum_k a_1 \otimes \dots \otimes a_{(k)}a_i \otimes \dots \otimes
a_n \otimes f_{0ik}(t_1,\dots,t_n). \]
We compute therefore
\begin{multline*}
\nabla_\tau d_1 \alpha = \sum_{i=1}^n \sum_k a_1 \otimes \dots \otimes a_{(k)}a_i
\otimes \dots \otimes a_n \otimes Df_{0ik} + \sum_{i=1}^{n-1} \sum_k a_1 \otimes \dots
\otimes a_{(k)}a_i \otimes \dots \otimes \omega_{(\zeta)}a_n \otimes f_{0ik} \\ 
+ \sum_k a_1 \otimes \dots \otimes a_{n-1} \otimes \omega_{(\zeta)}a_{(k)}a_n
\otimes f_{0nk} \\ 
- \sum_{i=1}^{n} \sum_{k} \sum_{j = 1}^{i-1} \sum_{m \geq 1} a_1 \otimes \dots
\otimes \omega_{(m)} a_j \otimes \dots \otimes a_{(k)}a_i \otimes \dots \otimes
a_n \otimes \overline{\wp}_{m+1}(t_n - t_j) f_{0ik} \\ 
- \sum_{i=1}^{n-1} \sum_{k} \sum_{m \geq 1} a_1 \otimes \dots
\otimes \omega_{(m)} a_{(k)}a_i \otimes \dots \otimes 
a_n \otimes \overline{\wp}_{m+1}(t_n - t_i) f_{0ik} \\ 
- \sum_{i=1}^{n-2} \sum_{k} \sum_{j = i+1}^{n-1} \sum_{m \geq 1} a_1 \otimes \dots
\otimes a_{(k)}a_i \otimes \dots \otimes  \omega_{(m)} a_j \otimes \dots \otimes
a_n \otimes \overline{\wp}_{m+1}(t_n - t_j) f_{0ik}.
\end{multline*}
We add and subtract the terms with $m = 0$ as in the last three sums and use
Borcherds identity in the third and fifth summand on the right hand side to obtain 
\begin{multline*}
\nabla_\tau d_1 \alpha = d_1 \Bigl( a \otimes a_1 \otimes \dots \otimes a_n
\otimes D f + a \otimes a_1 \otimes \dots \otimes a_{n-1} \otimes
\omega_{(\zeta)} a_n \otimes f \\ 
- \sum_{j=1}^{n-1}  \sum_{m \geq 1} a \otimes a_1 \otimes \dots \otimes
\omega_{(m)} a_j \otimes \dots \otimes a_n \otimes \owp_{m+1}(t_n - t_j) f \\  + 
\sum_{m \geq 1} \omega_{(m)} a \otimes a_1 \otimes \dots \otimes a_n \otimes
\owp_{m+1}(t_n-t_0) f 
 \Bigr). 
\end{multline*}
This proves that $\nabla_\tau$ preserves homology in degree $0$. The degree $1$
case is proved in the same way. 
\end{proof}
\begin{rem} \label{rem:extend-scalars} In order to define $\nabla_\tau$, it is
enough to extend scalars by adding $g_2(\tau)$ instead of all
holomorphic functions, ie. the connection $\nabla_\tau$ is well defined on $C^n_\bullet \otimes_{\bM_*}
\bQM_* \subset C^n_\bullet \otimes_{\bJ^n_*} \cF_*$. We can therefore consider the complexes
\begin{equation} 
 C^n_2 \otimes_{\bM_*} \bQM_* \xrightarrow{d_2} C^n_1 \otimes_{\bM} \bQM_*
\xrightarrow{d_1} C^n_0  \otimes_{\bM} \bQM_* \rightarrow 0,
\label{eq:complejo-plano}
\end{equation}
and the $\nabla_\tau$-- and $k[\partial_1,\dots , \partial_n]$--coinvariants of the corresponding homology. 
 We will denote these coinvariants by $H^{ch}_{i}(V^{\otimes n})_\nabla$.

Now let
$a_0,\dots,a_n$ be of homogeneous conformal weight and let $f \in \bJ^n_k$. Let
$\alpha = a_0 \otimes \dots \otimes a_n \otimes f \in C^n_1$ and put $\deg \alpha = k +
\sum_{i = 0}^n \deg a_i$.  This defines a grading on $C^n_1$ that we may call
the \emph{conformal grading}. We let $(C^n_1)_p$ consist of elements of
conformal degree $p$. Suppose $d_1 \alpha= 0$. Then every element in \eqref{eq:connection_2} lies in
$(C^n_1)_{\deg \alpha+2}$ except   
\[ a_0 \otimes \dots \otimes a_n \otimes \left( Df - \sum_{i=0}^n \deg a_i
g_2(\tau) f \right). \]
Let $g(t_0,\dots,t_n, \tau) = D(f) - \sum_{i=0}^n \deg a_i g_2(\tau) f$. It
follows from \eqref{eq:D-non-mod} and \eqref{eq:g2-non-mod} that $g$ is not modular but rather
satisfies 
\begin{equation} g\left( \frac{t_0}{c \tau + d}, \dots,\frac{t_n}{c \tau + d},
\frac{a \tau + b}{c \tau + d} \right) = \left( c \tau + d \right)^{k+2}
g(t_0,\dots,t_n, \tau) +  2 \pi i  c \deg \alpha \left( c \tau + d \right)^{k+1}
f(t_0,\dots,t_n, \tau). 
\label{eq:nabla-non-mod}
\end{equation}
\end{rem}
\begin{nolabel}[Specializing the curve] We can restrict the complex
$C^n_\bullet$ (or rather its extension of scalars to $\cF_n$) by fixing the points $t_1,\dots,t_n \in \mathbb{C}$ with $t_i \neq
t_j$ for $i \neq j$ and $\tau \in \mathbb{H}$. The homology of the resulting
complex is denoted by $H^{\ch}_i(V^{\otimes n}, t_1,\dots,t_n, \tau)$ and it
corresponds to the chiral homology of the elliptic curve $E_\tau = \mathbb{C} / \mathbb{Z} +
\tau \mathbb{Z}$ with insertions of the vacuum module $V$ at the points
$t_1,\dots,t_n$, as defined in \cite[\S 4.2.19]{beilinsondrinfeld}. Explicitly,
$C^n_0$ specializes to $V^{\otimes n}$, $C^n_1$ specializes to the quotient of
$V^{\otimes n+1} \otimes \Gamma_1$, where $\Gamma_1$ consists of meromorphic
functions $f(t_0)$, biperiodic with periods $1$ and $\tau$, with possible poles at
$t_0 = t_i + \mathbb{Z} + \tau \mathbb{Z}$, $i = 1,\dots,n$,  by the relation
\eqref{eq:c1-quotb2}. An similarly $C^n_2$ specializes to the quotient of
$V^{\otimes n+2} \otimes \Gamma_2$, where $\Gamma_2$ consists of meromorphic
functions $f(u,v)$, biperiodic, with possible poles at $u,v = t_i + \mathbb{Z} +
\tau \mathbb{Z}$, $i = 1,\dots,n$, and $u = v + \mathbb{Z} + \tau \mathbb{Z}$, by elements of
the form \eqref{eq:c2-quot}. 

The chiral homology $H^{\ch}_i (V^{\otimes n})$ induces a vector bundle with a
flat connection over the
moduli space of $n$-marked elliptic curves, its fiber over the curve $E_\tau$
with markings $t_1,\dots,t_n$ is $H^{\ch}_i(V^{\otimes
n},t_1,\dots,t_n,\tau)$. We remark that in the general case the connection is
not flat but rather projectively flat, but in the genus $1$ case (where there is
only one modular parameter) this connection is actually flat as we saw in
\eqref{eq:non-flat-connection} and \eqref{eq:non-flat-connection-3}. 

In the case $n=1$, as we saw in \ref{no:one-coinv-does-nothing} we have an
isomorphism $H^{\ch}_i(V) \simeq H^{\ch}_i(V)_{\text{tr}}$ therefore
$H^{\ch}_i(V, t_1,\tau)$ does not depend on $t_1$, we denote this homology by
$H^{ch}_i (E_\tau, V)$.  
\label{no:supports-as-bd}
\end{nolabel}
\begin{nolabel} We will be interested in the dual vector spaces
$H^{\ch}_\bullet(V^{\otimes n})^*$ and their flat sections as we
vary the points and the modular parameter.  Recall that we have
fixed a $\bJ_*^n$-module structure on $\bJ_*^m$ for all $m \geq n$ in
\ref{no:jacobi-forms} and that the complexes $C^n_\bullet$ are complexes of
$\bJ_*^n$-modules. Let $Z^n_i \subset C^n_i$ consist of cycles, that is $Z^n_i =
\ker d_i$. We will consider linear functionals 
\begin{equation}
 S^n_i \in \Hom_{\bJ^n_*} \Bigl( Z^n_i, \cF_{n} \Bigr), \qquad S^n_i(d_{i+1}
\alpha) = 0, \quad  \forall \:  \alpha \in C^{n}_{i+1}. 
\label{eq:si-domain-def}
\end{equation}
Thus,
$S^n_i$ gives a vector in the dual to chiral homology $H^{\ch}_i(V^{\otimes
n})^*$. 
 
We remark that these are not cohomology classes of the dual complex of
$C^n_\bullet$ as these linear functionals are only defined on cycles and not on
all chains.
\label{defn:definition-cohomologya}
\end{nolabel}
\begin{nolabel}[Moving the points] 
Recall that the complex $C^n_\bullet$ is a complex of
$k[\partial_1,\dots,\partial_n]$-modules. $\cF_{n}$ is a
$k[\partial_1,\dots,\partial_n]$-module where $\partial_i$ acts as
$\tfrac{\partial}{\partial t_i}$. Hence $\Hom_{\bJ^n_*}(Z^n_i,\cF_n)$ is a
$k[\partial_1,\dots,\partial_n]$--module where the action of $\partial_j$ is given
by
\begin{equation} 
\Bigl( \partial_j S^n_i \Bigr) \left( \alpha \right) = \frac{\partial}{\partial
t_j} S^n_i \left( \alpha \right) -  S^n_i \left( \left[ L_{-1}^{(j)} +
\frac{\partial}{\partial t_j} \right]\alpha \right), \qquad S_i^n
\in \Hom_{\bJ^n_*}\left( Z^n_i, \cF_n \right), \quad \alpha \in Z_{i}^{n}.
\label{eq:connection-tn}
\end{equation}
We denote by $H^{ch}_i (V^{\otimes n})^{* \text{tr}}$ the invariants of this action, namely
the subspace consisting on classes $S^n_i$
such that the right hand side of \eqref{eq:connection-tn} vanishes for every $1 \leq j \leq
n$.  
\label{defn:cohomology-def-2b}
\end{nolabel}
\begin{defn} A \emph{flat section of the dual of chiral homology}, of degree $i
= 0, 1$ is the cohomology class of a linear functional $S^n_i \in
H^{\ch}_i(V^{\otimes n})^{*\text{tr}}$ which in addition is flat
with respect to the dual connection $\nabla_\tau$, that is 
\begin{equation}
D_\tau S^n_i(\alpha) = S^n_i( \nabla_\tau \alpha), \qquad
\text{for all $\alpha \in Z^i_n$},
\label{eq:flatness-0}
\end{equation}
where $D_\tau$ is the derivation defined in \eqref{eq:vector-field}. 
We will denote by $H^{\ch}_i(V^{\otimes n})^{*\nabla}$ the group of such flat
sections.
\label{defn:chiral-cohomology-flat-sections}
\end{defn}
\begin{rem} 
In degree $0$ the space of flat sections of the dual of chiral homology agrees
with the space of \emph{conformal blocks of $V$ in genus 1} as defined by Zhu
\cite{zhu}. It is clear that we also have $H^{ch}_0(V^{\otimes n})^{*\nabla} =
\left(H^{ch}_0(V^{\otimes n})_\nabla\right)^*$. 
\label{rem:zhu-connection}
\end{rem}
\section{Degree $1$ conformal blocks on the torus} \label{sec:degree-1-conf}
In this section we define the degree $1$ analog of the notion of \emph{conformal
block} as in \cite{zhu}. It consists of a compatible system of flat sections of
the dual of the first chiral homology with supports on $n$ points, for every $n$. We prove the modular invariance of the degree $1$ conformal block in \ref{no:modular-deg1}. 

\begin{defn} Let $(V,\vac, \omega, Y(\cdot,z))$ be a vertex algebra. For each $n
\geq 1$, consider
the complex $C^n_\bullet$ computing the chiral homology with coefficients in
$V^{\otimes n}$. We let $Z^n_1 \subset C^n_1$ be the set of cycles, that is the
kernel of $d_1$. A system of maps $S_1 = \left\{ S^n_1 \right\}_{n \geq 1}$ 
\[ S^n_1: Z^n_1 \rightarrow \cF_n, \qquad (a_0,\dots,a_n, f) \mapsto S^n_1(a_0,\dots,a_n,f;  t_1,\dots,t_n,\tau), \] 
is said to satisfy the \emph{genus-one property} if the following properties are satisfied.
\begin{enumerate}
\item $S^n_1$ defines a flat section of the dual of the first chiral homology
in the torus:  $S^n_1 \in H^\text{ch}_1(V^{\otimes n})^{*\nabla}$. 
\item For $a_1,\dots,a_n \in V$ and $f \in \bJ^{n+1}_*$: 
\begin{equation}
 S^n_1(\vac, a_1,\dots,a_n,f;t_1,\dots,t_n,\tau) = 0
\end{equation}
\item Let $\alpha = a_0 \otimes \dots \otimes a_n \otimes f \in Z^n_1$ and let
$\sigma \in S_n$. Recall the action of the symmetric group on $C^n_\bullet$
defined in \ref{no:symmetric-group}. We have
\begin{equation}
S^n_1\left( \sigma(\alpha); t_{\sigma(1)},\dots,t_{\sigma(n)} \right) = S^n_1(\alpha;
t_1,\dots,t_n).
\label{eq:s-symmetric}
\end{equation}
\item Let $b,a_0,\dots,a_n \in V$ and $f = f(u,v,t_1,\dots,t_n,\tau) \in
\bJ^{n+2}_*$. 
We can view 
\begin{equation} \label{eq:vintwocomplexes}
 v :=  a_0 \otimes b \otimes a_1 \otimes \dots \otimes a_n \otimes f  \in
C^{n+1}_1.
\end{equation}
Then if $v$ is a cycle, that is $d_1 v = 0$, using the Laurent expansions of $f$ as in \eqref{eq:5.12b}. We have 
\begin{multline}
S^{n+1}_1 \Bigl( a_0,b,a_1,\dots, a_n, f; v,t_1,\dots,t_n, \tau \Bigr) = \\
\sum_{i = 1}^{n} \sum_k  S^{n}_1 \Bigl( a_0, a_1,\dots, b_{(k)}a_i,\dots,a_n,
f_{1i k}; t_1,\dots,t_n,\tau)   \Bigr) \zeta(v - t_i) \\ 
+\sum_{i = 1}^{n} \sum_k  S^{n}_1 \Bigl( a_0, a_1,\dots, b_{(k)}a_i,\dots,a_n,
f_{1i k-1}; t_1,\dots,t_n,\tau   \Bigr) \left(\wp_{2}(v - t_i) + g_2(\tau) \right) \\ 
+\sum_{i = 1}^{n} \sum_k  \sum_{m \geq 0} S^{n}_1 \Bigl( a_0, a_1,\dots,
b_{(k)}a_i,\dots,a_n, f_{1i k-m-2} ;t_1,\dots,t_n,\tau  \Bigr)  \wp_{m+3}(v - t_i) \\ 
- \sum_{i = 1}^{n-1} \sum_k  S^{n}_1 \Bigl( a_0, a_1,\dots, b_{(k)}a_i,\dots,a_n, f_{1i k};t_1,\dots,t_n,\tau   \Bigr) \zeta(t_n - t_i) \\ 
- \sum_{i = 1}^{n-1} \sum_k  S^{n}_1 \Bigl( a_0, a_1,\dots, b_{(k)}a_i,\dots,a_n, f_{1i k-1};t_1,\dots,t_n,\tau \Bigr) \left(\wp_{2}(t_n - t_i)+g_2(\tau) \right)  \\ 
- \sum_{i = 1}^{n-1} \sum_k  \sum_{m \geq 0} S^{n}_1 \Bigl( a_0, a_1,\dots, b_{(k)}a_i,\dots,a_n, f_{1i k-2};t_1,\dots,t_n,\tau  \Bigr) \wp_{m+3}(t_n - t_i)  \\ 
+  \sum_k  S^{n}_1 \Bigl( a_0, a_1,\dots, b_{(k)}a_n, f_{1n k+1};t_1,\dots,t_n,\tau \Bigr)  \\ 
- \sum_{k} \sum_{j \geq 1} S^n_1 \Bigl( a_0,a_1,\dots,b_{(k)}a_n, f_{1n,k+1-2j};t_1,\dots,t_n,\tau \Bigr) g_{2j}(\tau)
.
\label{eq:insertion-conf-2}
\end{multline}
\end{enumerate}
\label{defn:conf-1-blocks}
\end{defn}
\begin{nolabel} We prove now that the right hand side of \eqref{eq:insertion-conf-2} is
well defined. Consider 
\[ \alpha_s = a_0 \otimes b \otimes a_1 \otimes \dots \otimes a_n \otimes
g_s(u,v,t_1,\dots,t_n), \]
where $g_s(u,v,t_1,\dots,t_n) = f(u,v,t_1,\dots,t_n)\left( \zeta(v - t_n) +
\zeta(s-v) \right)$ is a family of meromorphic functions parame\-trized by $s$.
The functions $g_s$ are not in $\cF_{n+2}$ since they have poles at $v = s$ in
addition to the diagonal divisors, and are not elliptic with respect to the
variable $t_n$. As such $\alpha$ cannot be viewed as an
$s$-family of elements in $C^n_2$. However, we may formally define $\beta_s = d_2
\alpha_s$ using \eqref{eq:d2-general-def}. It follows in the same way as in
Lemma \ref{lem:complex-supports} that $d_1 \beta_s = 0$. Note however that
$\beta_s$ is not a cycle in $C^{n}_1$ since the functions involved have poles at
$t_i = s$ for $i=1,\dots,n$ in addition to the diagonal divisor (there is no
pole at $t_0 = s$ due to our assumption that $v$ defined by
\eqref{eq:vintwocomplexes} is a cycle).
Equation \eqref{eq:insertion-conf-2} is equivalent to 
\begin{equation} S^{n+1}_1 \Bigl(a_0,b,a_1,\dots,a_n,f;v,t_1,\dots,t_n, \tau \Bigr) = S^n_1
\Bigl( \beta_s; t_1,\dots,t_n,\tau \Bigr)|_{s = v}.
\label{eq:insertion-d2-exp}
\end{equation}
However, the $s$-dependent part of 
$\beta_s$ can be written as a linear combination of elements of the form $\gamma
\otimes h \cdot \wp_k(t_i - s)$ where $\gamma \in V^{\otimes (n+1)}$, $h \in
\cF_{n+1}$, $k \geq 1$ and $1 \leq i \leq n$. It follows by the
$\cF_{n}$--linearity of the differential that each coefficient of $\wp_k(t_i
-s)$ is a cycle in $C^n_1$ and therefore in the domain of $S^n_1$. Thus, 
the right hand side of \eqref{eq:insertion-d2-exp} is well defined.

We now prove that the right hand side of \eqref{eq:insertion-conf-2} is an elliptic
function in 
$\cF_{n+1}$. Indeed every term except the first and the fourth are manifestly
elliptic. Notice that the sum of these two terms is also elliptic for the
variables $t_i$, $1 \leq i \leq n-1$. Denote by $S = S(v,t_1,\dots,t_n)$ the right hand side
of \eqref{eq:insertion-conf-2}. It follows that 
\begin{equation} \label{eq:sniselliptic}
\begin{multlined}
S(v,t_1,\dots,t_n + \tau) - S(v,t_1,\dots,t_n) = - S(v + \tau, t_1,\dots,t_n) + S(v,t_1,\dots,t_n) =  \\ 
 2 \pi i \sum_{i = 1}^{n} \sum_k  S^{n}_1 \Bigl( a_0, a_1,\dots, b_{(k)}a_i,\dots,a_n,
f_{1i k}(t_0,\dots,t_n,\tau)   \Bigr) .
\end{multlined}
\end{equation}
Consider the vector $v$ defined by \eqref{eq:vintwocomplexes} but viewed as a
vector in $C^{n}_2$. We compute 
\[ d_2 v = \sum_{i = 1}^n  \sum_k a_0 \otimes \dots \otimes b_{(k)}a_i \otimes \dots
\otimes a_n \otimes f_{1ik}(t_0,\dots,t_n,\tau), \]
where we used that $d_1 v = 0$ when viewed as an element of $C^{n+1}_1$. Thus,
the right hand side of \eqref{eq:sniselliptic}
equals $2 \pi i S^n_1 \left( d_2 v \right) = 0$.
\label{no:ins-well-defined}
\end{nolabel}
\begin{nolabel}
Given such a system, for each $a_0,\dots,a_n \in V$ and $f \in \bJ^{n+1}_*$, and a fixed $\tau \in \mathbb{H}$, we may view \[S^n_1(a_0,\dots,a_n, f; t_1,\dots,t_n, \tau),\] as an $n$-variable function on the complex torus $\mathbb{C}/\mathbb{Z} + \mathbb{Z}\tau$. These are the higher analogs of the \emph{correlation functions} on the torus defined by Zhu in \cite{zhu} in degree $0$. 

By definition it is clear that if $\left\{ S^n_1 \right\}$ and $\left\{ S'^n_1 \right\}$ satisfy the genus-one property then so does $\left\{ \alpha S^{n}_1 + S'^n_1 \right\}$ for any $\alpha \in \mathbb{C}$. 
\begin{defn*}The linear space of families $S_1 = \{S_1^n\}$ satisfying the genus one property is
called the space of \emph{degree $1$ conformal blocks on the torus}. 
\end{defn*}
\label{defn:1-conformal-block}
\end{nolabel}
\begin{rem}Notice that all the terms in the right hand side of \eqref{eq:insertion-conf-2}
involving $g_2(\tau)$ cancel each-other and can be therefore omitted from the
identity. 
\label{rem:g2-doesnt-count}
\end{rem}
\begin{nolabel}
Let $f(t_0,t,t_1,\dots,t_n, \tau) \in \bJ^{n+2}_*$ and consider its Laurent expansion near $t-t_1$:
\[ f = \sum_{k} f_k(t_0,t_1,\dots,t_n, \tau) (t-t_1)^k. \]
It follows from \eqref{eq:insertion-conf-2} 
that the following holds
\begin{equation}
\int_C S^{n+1}_1 \Bigl( a_0, a,a_1,\dots,a_n, f; t,t_1,\dots,t_n, \tau \Bigr)dt = \sum_{k} S^n_1 \Bigl( a_0, a_{(k)}a_1,a_2,\dots,a_n, f_{k}; t_1,\dots,t_n,\tau \Bigr),
\label{eq:integral}
\end{equation}
where $C$ is a small countour around $t_1$, not including any of the points
$t_i$, $i \neq 1$. This is the degree $1$ analog of \cite[(4.1.4)]{zhu}. To
prove \eqref{eq:integral} notice that from the right hand side of
\eqref{eq:insertion-conf-2} only the $i=1$ summand of the first term has a non-trivial residue at $s
= t_1$. 
\label{no:integral}
\end{nolabel}
\begin{nolabel}
Consider \eqref{eq:insertion-conf-2} with $b = \omega$ and $f =
f(t_0,\dots,t_n,\tau) \in \bJ^{n+1}_*$, that is, independent of $s$. It follows from \eqref{eq:flatness-0} that 
\begin{equation}
\begin{multlined}
S^{n+1}_1 \Bigl( a_0, \omega, a_1,\dots,a_n, f; t_1,\dots,t_n \Bigr) = \left( 2 \pi i \frac{d}{d \tau} - \sum_{i=1}^{n-1} \zeta(s-t_i) \frac{d}{d t_i} \right)  S^{n}_1 \Bigl( a_0,\dots,a_n, f; t_1,\dots,t_n,\tau \Bigr) \\ 
- S^{n}_1 \left( a_0,\dots,a_n, \left(  2 \pi i \frac{d}{d \tau} - \sum_{i=0}^{n-1} \zeta(s-t_i) \frac{d}{dt_i} \right)  f ; t_1,\dots,t_n,\tau \right) \\ 
- S^n_1 \Bigl( a_0,\dots,a_n, \left( \wp_2(s-t_0) + g_2(\tau) f \right);
t_1,\dots,t_n,\tau \Bigr).
\end{multlined}
\label{eq:new-4.0.1}
\end{equation}
This is a degree $1$ analog of \cite[(4.1.5)]{zhu}. 
\label{no:ins-diff-eq}
\end{nolabel}
\begin{thm}\label{no:modular-deg1}  Let $\{S^n_1\}$ satisfy the genus--one
property and let 
\[ 
\sigma = \begin{pmatrix}
a & b \\ c & d 
\end{pmatrix} \in SL(2, \mathbb{Z}). 
\]
Define $\{\sigma(S)^n_1\}$ as follows. For $a_i \in V$, $i=0,\dots,n$ of conformal
weight $\Delta_i$ and $f  \in \bJ^{n+1}_k$. Put
\begin{equation}
\sigma(S)^n_1\left( a_0,\dots,a_n,f;t_1,\dots,t_n,\tau \right) = \left( c \tau +
d\right)^{-k -\sum_{i=0}^n \Delta_i} 
 S^n_1\left(
a_0,\dots,a_n,f; \frac{t_1}{c \tau +d},\dots,\frac{t_n}{c \tau +d},
\frac{a \tau + b}{c \tau +d} \right). 
\label{eq:sl2action}
\end{equation}
Then $\{ \sigma(S^n_1)\}$ satisfies the genus--one property. 
\end{thm}
\begin{proof}
It is clear that $\sigma(S)^n_1$ satisfies \eqref{eq:s-symmetric}. 
Let now $\alpha = a_0 \otimes \dots \otimes a_n \otimes f \in Z^n_1$ with $f \in
\bJ^{n+1}_k$ and $a_0,\dots,a_n$ of homogeneous conformal weight. We put $\deg \alpha = k + \sum_{i=0}^n \deg a_i$. 
\begin{multline*} \sigma(S)^n_1 \left( L_{-1} a_0,\dots, a_n, f; t_1,\dots,t_n,
\tau  \right) \\
\begin{split}
&= \left( c \tau + d \right)^{-\deg \alpha - 1}
S^n_1 \left(L_{-1} a_0, \dots a_n, f; \frac{t_1}{c \tau +
d},\dots,\frac{t_n}{c\tau +d} , \frac{a  \tau + b}{c\tau +d}\right) \\
&= 
- \left( c\tau + d \right)^{-\deg \alpha - 1} S^n_1 \left(
a_0,\dots,a_n, \frac{\partial}{\partial t_0}f; \frac{t_1}{c\tau
+d},\dots,\frac{t_n}{c\tau +d}, \frac{a \tau +b}{c \tau +d} 
\right) \\ 
&= - \sigma(S)^n_1\left( a_0,\dots,a_n, \frac{\partial}{\partial t_0} f;
t_1,\dots,t_n, \tau \right),
\end{split}
\end{multline*}
so that $\sigma(S)^n_1$ vanishes on elements as in \eqref{eq:c1-quotb2} and it
is therefore well defined on $Z^n_1$. 

Next, we check that $\sigma(S)^n_1$ vanishes on coboundaries. 
 Let $a,b, a_1, \dots,a_n \in V$ be homogeneous elements, and $f = f(u,v,t_1,\dots,t_n,\tau)
\in \bJ^{n+2}_p$. 
Let $\alpha = a \otimes b \otimes a_1 \dots \otimes a_n \otimes f \in C^{n}_2$.
We we have
\begin{equation}
\begin{aligned}
\sigma(S)^n_1 \left( d_2\alpha;t_1,\dots,t_n,\tau \right) &= \sum_{i=1}^n \sum_k \sigma(S)^n_1 \left( a,
a_1,\dots,b_{(k)}a_i, \dots,a_n, f_{1ik};  t_1,\dots,t_n, \tau \right) \\ &\quad - 
\sum_{i=1}^n  \sum_k \sigma(S)^n_1 \left( b, a_1,\dots,a_{(k)}a_i, \dots,a_n, f_{0ik};
t_1\dots,t_n, \tau \right)\\  &\quad - \sum_k \sigma(S)^n_1 \left( a_{(k)}b,
a_1,\dots,a_n, f_{00k}; t_1,\dots,t_n, \tau \right). 
\label{eq:modularity1}
\end{aligned}
\end{equation}
Noting that $f_{1ik} \in \bJ^{n+1}_{p+k}$ and $\deg b_{(k)}a_i = \deg b + \deg
a_i - k - 1$ we have
that the first term on the right hand side of \eqref{eq:modularity1} equals 
\[ 
\left( c \tau + d \right)^{1 -p - \deg b - \deg a - \sum_{i=1}^n \deg a_i}
S^n_1 \left( a, a_1,\dots,b_{(k)}a_i, \dots,a_n, f_{1ik}; \frac{t_1}{c \tau +d
},\dots, \frac{t_n}{c \tau + d}, \frac{a \tau +b}{c \tau +d} \right). 
\]
The other two terms are treated similarly. It follows that
\[ \sigma(S)^n_1 \left( d\alpha; t_1,\dots,t_n, \tau \right)   = 
\left( c \tau + d \right)^{1-p - \deg b - \deg a - \sum_{i=1}^n \deg a_i}
S^n_1 \left( d_2 \alpha; \frac{t_1}{c \tau +d
},\dots, \frac{t_n}{c \tau + d}, \frac{a \tau +b}{c \tau +d} \right) = 0, \]
therefore $\sigma(S)^n_1$ defines a vector in $H^{\text{ch}}_1(V^{\otimes n},
t_1,\dots,t_n)^*$. 

We now check that $\sigma(S)^n_1$ vanishes on elements as in
\eqref{eq:quotient-moving}. We have
for $\alpha = a_0 \otimes \dots \otimes a_n \otimes f \in Z^n_1$ with $f \in
\bJ^{n+1}_k$
\begin{multline*} \sigma(S)^n_1 \left( a_0,\dots,L_{-1}a_i, \dots, a_n, f; t_1,\dots,t_n,
\tau  \right) \\
\begin{split}
&= \left( c \tau + d \right)^{-\deg \alpha - 1}
S^n_1 \left(a_0, \dots, L_{-1}a_i,\dots, a_n, f; \frac{t_1}{c \tau +
d},\dots,\frac{t_n}{c\tau +d} , \frac{a  \tau + b}{c\tau +d}\right) \\
&= 
\left( c\tau + d \right)^{-\deg \alpha -1}\frac{\partial}{\partial t_i}  S^n_1 \left(
a_0,\dots,a_n, f; \frac{t_1}{c\tau
+d},\dots,\frac{t_n}{c\tau +d}, \frac{a \tau +b}{c \tau +d} 
\right) \\ 
& \quad - \left( c\tau + d \right)^{-\deg \alpha -1} S^n_1 \left(
a_0,\dots,a_n, \frac{\partial}{\partial t_i}f; \frac{t_1}{c\tau
+d},\dots,\frac{t_n}{c\tau +d}, \frac{a \tau +b}{c \tau +d} 
\right) \\ 
&= \frac{\partial}{\partial t_i} \sigma(S)^n_1\left( a_0,\dots,a_n, f;
t_1,\dots,t_n, \tau \right) \\ 
& \quad - \sigma(S)^n_1\left( a_0,\dots,a_n, \frac{\partial}{\partial t_i} f;
t_1,\dots,t_n, \tau \right).
\end{split}
\end{multline*}
So that $\sigma(S)^n_1$ gives a well defined vector on
$H^{\text{ch}}_1(V^{\otimes n})^{*\text{tr}}$. 

We now check that $\sigma(S)_1^n$ satisfies \eqref{eq:flatness-0}. Let
$a_0,\dots,a_n \in V$ be homogeneous elements and $f = f(t_0,\dots,t_n, \tau)
\in \bJ_k^{n+1}$. Put $\alpha = a_0 \otimes \dots \otimes a_n \otimes f \in
Z^n_1$.  
We have 
\begin{equation}\label{eq:modularity-nabla0}
\begin{split}
\frac{d}{d \tau} \sigma(S)_1^n (\alpha; t_1,\dots,t_n, \tau) &=
\frac{d}{d \tau} \left( c \tau + d \right)^{- \deg
\alpha} S_1^n\left(\alpha; \frac{t_1}{c \tau + d}, \dots, \frac{t_n}{c \tau + d},
\frac{a \tau + b}{c \tau +d} \right) \\
&= - c \deg \alpha  \left( c \tau + d \right)^{-1 - \deg \alpha} S_1^n\left( 
\alpha; \frac{t_1}{c \tau + d}, \dots, \frac{t_n}{c \tau + d},
\frac{a \tau + b}{c \tau +d} 
\right) \\ 
&\quad - 
c \left( c \tau + d \right)^{-2 - \deg
\alpha} \sum_{i=1}^n  \left( \frac{\partial}{\partial t_i} S_1^n \right) \left(\alpha; \frac{t_1}{c \tau + d}, \dots, \frac{t_n}{c \tau + d},
\frac{a \tau + b}{c \tau +d} \right) \\ 
& \quad +  \left( c \tau + d \right)^{-2 - \deg \alpha}
\left( \frac{\partial }{\partial \tau} S_1^n \right) \left( \alpha; \frac{t_1}{c \tau + d}, \dots, \frac{t_n}{c \tau + d},
\frac{a \tau + b}{c \tau +d}  \right). 
\end{split}
\end{equation}
The second term on the right hand side vanishes because of Lemma \ref{lem:4.sum-deriv}. The
first term equals 
\begin{equation} \label{eq:modularity-nabla1}  -  \frac{c \deg \alpha}{c \tau + d}
\sigma(S)_1^n(\alpha; t_1, \dots,t_n, \tau). 
\end{equation}
Finally using that $S_1^n$ satisfies \eqref{eq:flatness-0}, the third term equals
\[
\frac{1}{2 \pi i}  \left( c \tau + d \right)^{-2 - \deg \alpha }
S_1^n \left( \nabla_\tau \alpha; \frac{t_1}{c\tau +d}, \dots,
\frac{t_n}{c \tau+d},\frac{a \tau+b}{c\tau+d} \right),
\]
It follows from Remark \ref{rem:extend-scalars} and \eqref{eq:nabla-non-mod}
that this term equals 
\[ \frac{1}{2 \pi i} \sigma(S)_1^n \left( \nabla_\tau \alpha; t_1,\dots,t_n,
\tau \right) + \frac{c \deg \alpha}{c \tau +d} \sigma(S)_1^n(\alpha;
t_1,\dots,t_n,\tau).  \]
The second term cancels \eqref{eq:modularity-nabla1} and the first term together
with the left hand side of \eqref{eq:modularity-nabla0} show that $\sigma(S)_1^n$ is a flat
section in $H^{\text{ch}}_1(V^{\otimes n})^{*\nabla}$

It is only left to prove that $\sigma(S)_1^n$ satisfies the insertion formula
\eqref{eq:insertion-conf-2}. We let $a_0,\dots,a_n,b$ be homogeneous elements in
$V$ and $f = f(u,v,t_1,\dots,t_n,\tau) \in \bJ^{n+2}_p$. We define $v$ as in
\eqref{eq:vintwocomplexes} and assume it is a cycle, that is $d_1 v = 0$. We
want to evaluate 
\begin{multline*} \sigma(S)_1^{n+1} \left( a_0,b,a_1,\dots,a_n,f; v,t_1,\dots,t_n,\tau \right)
= \\  \left( c \tau+d \right)^{-k - \deg b - \sum_{i=0}^n \deg a_i} S_1^{n+1}\left(
a_0,b,a_1,\dots,a_n,f;\frac{v}{c\tau+d},\frac{t_1}{c\tau+d},\dots,\frac{t_n}{c\tau+d},\frac{a\tau+b}{c\tau+d}
\right). 
\end{multline*}
We now notice that in the right hand side of \eqref{eq:insertion-conf-2} function appearing
transforms maniflestly as a Jacobi form except the first and the fourth terms involving a $\zeta$
function. The sum of these two terms is however manifestly modular and the
Theorem follows by simply applying \eqref{eq:insertion-conf-2} for $S_1^n$ and
the definition of $\sigma(S)_1^n$.
\end{proof}
\section{Modified Vertex Operators}\label{sec:modified.v.o}
In this section we prove a version of Borcherds identity adapted to elliptic curves via the exponential map. This identity \eqref{eq:bmf} involves the two different vertex algebra structures $Y(\cdot,z)$ and $Y[\cdot,z]$ on $V$, as well as the Laurent and Fourier expansions of an elliptic function in the coordinates $t$ and $z = e^{2 \pi i t}$. 
\begin{nolabel}\label{no:modified-operators}
We introduce the following modified vertex operators
\begin{align*}
X(a,z) = z^{-1} Y(z^{L_0} a,z) = \sum_{n \in \mathbb{Z}} a_n z^{-n-1}.
\end{align*}
\end{nolabel}
\begin{lem}
The modified vertex operators $(X\cdot, z)$ satisfy the following identities: 
\begin{equation}
X\left( (L_{-1} +  L_0)a,z \right) = z\frac{d}{dz} X(a,z) + X(a,z),
\label{eq:x-translation}
\end{equation}
and
\begin{equation}
q^{L_0} X(a,z) = q X(a, qz) q^{L_0}. 
\label{eq:q0-xoperators}
\end{equation}
\label{lem:modified-operators}
\end{lem}
\begin{thm}\label{prop:borcherds-modified-fourier}
Let $f(t)$ be a meromorphic function of a complex variable $t$ which is periodic, i.e., $f(t+1) = f(t)$. Let $\tau \in \mathbb{H}$ and suppose $f$ is analytic in the domain $-\im \tau < \im t < \im \tau$ except possibly at $t=0$ where it may have a pole. Put $q = e^{2 \pi i \tau}$, $z = e^{2 \pi i t}$ and let $F_+(z)$ (resp. $F_-(z)$ ) be the Fourier expansion of $f$ as in \ref{no:fourier-n-variable}, in the domain $|q|<|z|<1$ (resp. $1 < |z| < |q|^{-1}$). If either
\begin{itemize}
\item $f(t)$ is regular at $t=0$, or

\item $f$ satisfies $f(t+\tau) = f(t) - \alpha$ for all $t \in \mathbb{C}$ for some constant $\alpha \in \mathbb{C}$,
\end{itemize}
then the following identity holds
\begin{equation}\label{eq:bmf}
2\pi i X(a_{[f]} b, w)  = \res_z F_-\left( \frac{z}{w} \right) X(a,z) X(b,w) -
\res_z F_+\left( \frac{z}{w} \right) X(b,w) X(a,z),
\end{equation}
as an equality in $\Hom(V,V ((w)))$.
\end{thm}
Notice that two different vertex algebra structures of $V$ are in play here: the usual vertex algebra structure in the form of $X(\cdot, z)$, and also the $Y[\cdot, z]$ vertex algebra structure of \ref{no:zhu-alter-vertex} is used in the left hand side of \eqref{eq:bmf}.
\begin{proof}
It is enough to prove the statement for $a,b \in V$ of homogeneous conformal weight. We assume this throughout. We also fix $\tau \in \mathbb{H}$ and write $q = e^{2\pi i \tau}$.
Our hypotheses on $f(t)$ imply that it admits a Laurent expansion near $t =0$, namely 
\begin{align*}
f(t) = \sum_{n \geq N} f_n t^n,
\end{align*}
so the left hand side of (\ref{eq:bmf}) becomes
\begin{align}\label{eq:fourier-exponentiated-LHS}
X\left( \res_t f(t)Y[a,t]b dt,w\right) = \sum_{n \geq N} f_n X\left(a_{[n]}b,w \right) \in \Hom\bigl(V,V((w))\bigr).
\end{align}
Since $f(t+1)=f(t)$, the function $f$ has Fourier series expansions 
\begin{align*}
F_\pm (z) = \sum_{n\in \mathbb{Z}} F^\pm_n z^n, \quad \text{where} \quad F^\pm_n = \int_{t_0}^{t_0+1} f(t) e^{-2 \pi i n t} dt.
\end{align*}
Here $t_0$ is chosen so that $0 < \pm \Im t_0 < \Im\tau$ and the integral is taken along the path, parallel to the real axis in the $t$-plane, connecting $t_0$ to $t_0+1$.

The series $F_+(z)$ converges uniformly on closed subsets of the domain $|q| < |z| < 1$ and the series $F_-(z)$ converges uniformly on closed subsets of the domain $1 < |z| < |q|^{-1}$. In their corresponding domains of convergence we have
\begin{align*}
F_\pm (e^{2 \pi i t}) = f(t).
\end{align*}
If $f(t)$ is regular at $t=0$ then $F^-_n - F^+_n = 0$ for all $n \in \mathbb{Z}$ and if $f(t+\tau) = f(t) - \alpha$ then
\begin{equation}
F^-_n - F^+_n = \alpha \quad \text{for all $n \in \mathbb{Z}$.}
\label{eq:fourier-identity}
\end{equation}
To simplify the presentation we agree to put $\alpha = 0$ in case $f(t)$ is regular at $t=0$, so that (\ref{eq:fourier-identity}) holds in all cases.
\begin{multline}
\res_t f(t) Y[a,t]b dt = \res_t f(t) e^{2 \pi i \deg(a) t} Y(a,e^{2 \pi i t} -1)b dt \\
= \sum_{n \in \mathbb{Z}} a_{(n)}b \res_t f(t) e^{2 \pi i \deg(a) t} \left( e^{2 \pi i t} -1  \right)^{-n-1} dt.
\label{eq:residue-fourier-1}
\end{multline}
Since $f(t)$ is analytic with possible pole at $t=0$, we have
\begin{equation}
\res_t f(t) e^{2 \pi i t  \deg(a)} \left( e^{2 \pi i t} -1 \right)^{-n-1} dt = \frac{1}{2\pi i} \oint f(t)e^{2 \pi i t \deg(a)} \left( e^{2 \pi i t} -1 \right)^{-n-1} dt,
\label{eq:contour-2}
\end{equation}
where the integral is taken along an anticlockwise contour containing $t=0$ and no other pole of $f(t)$. One can always find a suitable contour of the form: $t_0$ to $t_0+1$ to $t_0+1+\tau$ to $t_0+\tau$ back to $t_0$. We denote the lower and upper segments of this contour by $C_1$ and $-C_2$ respectively. The integrals on the lateral segments cancel because $f(t+1)=f(t)$. For $n \leq -1$ we have
\begin{equation}
\begin{multlined}
\oint dt f(t) e^{2 \pi i \deg(a)} \left( e^{2 \pi i t} -1 \right)^{-n-1} =
 \sum_{j \geq 0} (-1)^j \binom{-n-1}{j} \oint dt f(t) e^{2 \pi i t (\deg(a) - n -1 -j)} \\
= 
 \sum_{j \geq 0} (-1)^j \binom{-n-1}{j} \left (\int_{C_1} dt f(t) e^{2 \pi i t (\deg(a) - n -1 -j)} 
- 
 \int_{C_2} dt f(t) e^{2 \pi i t (\deg(a) - n -1 -j)} \right)\\
=
\sum_{j \geq 0} (-1)^j \binom{-n-1}{j} \left( F^-_{n+1+j-\deg(a)} - F^+_{n+1+j-\deg(a)} \right). 
\label{eq:fourier-nnegative}
\end{multlined}
\end{equation}
By (\ref{eq:fourier-identity}) the right hand side of \eqref{eq:fourier-nnegative} equals $\alpha \delta_{n,-1}$.
For $n \geq 0$ we have the following convergent series expansions
\begin{equation}
\left( e^{2 \pi i t} -1 \right)^{-1-n} = 
\begin{cases}
(-1)^{n+1} \sum_{j \geq 0} \binom{n+j}{j} e^{2 \pi i jt},&  t \in  \mathbb{H}, \\
\sum_{j \geq n+1}  \binom{j -1}{n} e^{- 2 \pi i j t}, &  -{t}  \in \mathbb{H},
\end{cases}
\label{eq:different-expansions-exp}
\end{equation}
and hence
\begin{equation}
\begin{aligned}
\oint f(t) e^{2 \pi i \deg(a)} \left( e^{2 \pi i t} -1 \right)^{-n-1} dt &= 
\sum_{j \geq n+1} \binom{j-1}{n}  \int_{C_1} e^{2 \pi it (\deg(a)-j)} f(t) \\
& \quad + \sum_{j \geq 0} (-1)^n \binom{n+j}{j} \int_{C_2} f(t) e^{2 \pi i t (\deg(a) + j)} dt
\\ 
&= \sum_{j \geq n+1} \binom{j-1}{n} F^-_{j - \deg(a)}  + \sum_{j \geq 0} (-1)^n \binom{n+j}{j} F^+_{-j - \deg(a)}.
\end{aligned}
\label{eq:aux-2}
\end{equation}
Collecting (\ref{eq:fourier-nnegative}) and (\ref{eq:aux-2}) yields
\begin{equation}
\begin{multlined}
2\pi i X\left( a_{[f]}b, w \right) =  \alpha X(a_{(-1)}b,w) \\
+ \sum_{n \geq 0} \left( \sum_{j \geq n+1} \binom{j-1}{n} F^-_{j - \deg(a)} + \sum_{j \geq 0} (-1)^n \binom{n+j}{j} F^+_{-j - \deg(a)} \right)  w^{\deg(a) + \deg(b) -n-2}Y\left(a_{(n)}b,w\right).
\label{eq:borcherds-lhs}
\end{multlined}
\end{equation}
We now study the right hand side of \eqref{eq:bmf}:
\begin{equation}
\begin{multlined}
\res_z F_-\left( \frac{z}{w} \right) X(a,z)X(b,w) = \sum_{n \in \mathbb{Z}} F^-_n w^{-n +\deg(b) -1} a_n Y(b,w) \\ 
= \sum_{n \in \mathbb{Z}} F^-_{n+1 - \deg(a)} w^{\deg_a + \deg_b -n -2} a_{(n)} Y(b,w)
\label{eq:rhs-1-borcherds}
\end{multlined}
\end{equation}
Similarly we obtain
\begin{equation}
\res_z F_+\left( \frac{z}{w} \right) X(b,w)X(a,z) = \sum_{n \in \mathbb{Z}} F^+_{n+1 -\deg(a)} w^{\deg(a) + \deg(b) -n -2} Y(b,w) a_{(n)}\label{eq:rhs-2-borcherds}
\end{equation}
Subtracting (\ref{eq:rhs-2-borcherds}) from (\ref{eq:rhs-1-borcherds}) and using
\begin{align*}
[a_{(n)},Y(b,w)] = \sum_{j \geq 0} \binom{n}{j} w^{n-j} Y(a_{(j)}b,w),
\end{align*}
we obtain
\begin{equation}
\begin{multlined}
\res_z F_-\left( \frac{z}{w} \right) X(a,z)X(b,w) - \res_z F_+\left( \frac{z}{w} \right) X(b,w)X(a,z) \\
= \sum_{n < 0} F^-_{n+1 - \deg(a)} w^{\deg(a) + \deg(b) -n -2} a_{(n)} Y(b,w) \\
+\sum_{n \geq 0} F^{-}_{n+1-\deg(a)} \left[ \sum_{j \geq 0} \binom{n}{j} w^{\deg(a) + \deg(b) - j -2} Y\left( a_{(j)}b,w \right) + w^{\deg_a + \deg_b -n -2} Y(b,w) a_{(n)} \right] \\
- \sum_{n \geq 0} F^+_{n+1 -\deg(a)} w^{\deg(a) + \deg(b) -n -2} Y(b,w) a_{(n)} \\
+ \sum_{n < 0} F^+_{n + 1 - \deg(a)} \left[ \sum_{j \geq 0} \binom{n}{j} w^{\deg(a) + \deg(b) - j - 2} Y\left( a_{(j)}b,w \right) - w^{\deg(a) + \deg(b) -n -2} a_{(n)} Y(b,w) \right] \\ 
= \sum_{n,j \geq 0} F^{-}_{n+1-\deg(a)} \binom{n}{j} w^{\deg(a) + \deg(b) - j -2} Y\left( a_{(j)}b,w \right) \\
+ \sum_{n < 0, j \geq 0} F^+_{n + 1 - \deg(a)} \binom{n}{j} w^{\deg(a) + \deg(b) - j - 2} Y\left( a_{(j)}b,w \right) \\ 
+ \sum_{n < 0} \left( F^-_{n+1-\deg(a)} - F^{+}_{n+1 - \deg(a)}  \right) a_{(n)} Y(b,w) w^{\deg(a) + \deg(b) - n -2} 
\\ 
+ \sum_{n \geq  0} \left( F^-_{n+1-\deg(a)} - F^+_{n+1 - \deg(a)} \right) w^{\deg(a) + \deg(b) -n-2} Y(b,w) a_{(n)}.
\label{eq:multline-b}
\end{multlined}
\end{equation}
The sum of the final two terms of \eqref{eq:multline-b} is
\begin{multline}
\alpha \sum_{n < 0} a_{(n)} Y(b,w)w^{\deg(a) + \deg(b) - n -2} + \alpha \sum_{n \geq 0} w^{\deg(a) + \deg(b) -n -2} Y(b,w) a_{(n)} \\ = \alpha w^{\deg(a) + \deg(b) -1}  Y(a_{(-1)}b,w) = \alpha X(a_{(-1)}b,w).
\label{eq:-1-product}
\end{multline}
We use the identity $\binom{-n-1}{j} = (-1)^j \binom{n+j}{j}$ to rewrite the first two terms of the right hand side of (\ref{eq:multline-b}) as
\begin{align}\label{eq:final-lhs-1}
\sum_{j \geq 0} w^{\deg(a) + \deg(b)-j-2} Y\left( a_{(j)}b,w \right) \left[ \sum_{n \geq j+1} F^{-}_{n-\deg(a)} \binom{n-1}{j} + \sum_{n \geq 0} F^+_{-n - \deg(a)} (-1)^j \binom{n+j}{j} \right].
\end{align}
Comparing \eqref{eq:-1-product} and \eqref{eq:final-lhs-1} with \eqref{eq:borcherds-lhs} yields the desired result.
\end{proof}
We only use a limited number of cases of Theorem \ref{prop:borcherds-modified-fourier}, summarized in:
\begin{cor} For $a,b \in V$ the modified vertex operators satisfy
\begin{equation}
X\left( a_{[0]}b,w \right) =\frac{1}{2 \pi i} \left( \res_z X(a,z) X(b,w) - \res_z X(b,w) X(a,z) \right).
\label{eq:0-th-borcherds}
\end{equation}
\begin{equation}
\begin{gathered}
\begin{aligned}
X\left( a_{[\wp_1]} b,w \right) &- G_2(q) X(a_{[1]}b,w) + \pi i  X\left( a_{[0]},b,w \right) = \\ &\quad -\frac{1}{2 \pi i} \res_z \left( P_1 \left( \frac{zq}{w},q \right) - 2 \pi i  \right) X(a,z) X(b,w)  + \frac{1}{2 \pi i} 
\res_z P_1\left( \frac{z}{w},q \right) X(b,w) X(a,z).  
\end{aligned} \\ 
\begin{aligned}
X\left( a_{[\wp_2]} b,w \right) + G_2(q)  X\left( a_{[0]},b,w \right) &= \frac{1}{2 \pi i} \res_z P_2 \left( \frac{zq}{w},q \right)  X(a,z) X(b,w)\\ & \quad - \frac{1}{2 \pi i} 
\res_z P_2\left( \frac{z}{w},q \right) X(b,w) X(a,z). \\  
X\left( a_{[\wp_k]} b,w \right) &= \frac{(-1)^{k}}{2 \pi i} \res_z  P_k \left( \frac{zq}{w},q \right)  X(a,z) X(b,w)\\ & \quad - \frac{(-1)^{k}}{2 \pi i} 
\res_z P_k\left( \frac{z}{w},q \right) X(b,w) X(a,z) \qquad k \geq 3. 
\end{aligned}
\label{eq:borcherds-modified-1}
\end{gathered}
\end{equation}
\label{prop:borcherds-modified}
\end{cor}
\begin{proof}
Putting $f(t) = 1$ in Theorem \ref{prop:borcherds-modified-fourier} immediately yields \eqref{eq:0-th-borcherds}. By \eqref{eq:Pk.wpk.relation} and subsequent remarks, it follows that the Fourier expansions of $-\wp_1(t,\tau) + g_2(\tau)t - \pi i$ in the domains $|q| < |z| < 1$, resp. $1 < |z| < |q|^{-1}$, are $P_1(z, q)$, resp. $P_1(qz, q)-2\pi i$. The first equation of \eqref{eq:borcherds-modified-1} now follows from Theorem \ref{prop:borcherds-modified-fourier}. The proofs of the remaining equations are similar.
\end{proof}
Finally we will need the following slight generalization of the identities proved above.
\begin{prop} For each $a,b,c \in V$ the following identity of elements of $V[[u, u^{-1}, w, w^{-1}, q]]$ holds:
\begin{equation}
\begin{multlined}
\sum_{m \geq 0} (-1)^m \binom{k+m-1}{m} P_{k+m}\left( \frac{u}{w},q \right) X\left( a_{[m]}b,w \right) c 
= \frac{1}{2 \pi i} \res_z  P_k\left( \frac{u}{z},q \right) X( a, z) X(b,w) c \\ 
- \frac{1}{2 \pi i}\res_z  P_k\left( \frac{u}{z},q \right)  X(b,w) X(a,z) c.
\label{eq:borcherds-modified-2}
\end{multlined}
\end{equation}
\label{prop:borcherds-modified-2}
\end{prop}
\begin{proof}
Put $z = e^{2 \pi i t}$, $w = e^{ 2 \pi i v}$, $u = e^{2 \pi i s}$ and for each $k \geq 1$ and fixed $v,s, \tau$, consider the function 
\[ f(t, \tau) = (-1)^k \owp_k (t + v - s, \tau).
\]
We apply Theorem \ref{prop:borcherds-modified-fourier} with the above function. This function is periodic $f(t+1) = f(t)$ and analytic except at $t = s - v + m + n \tau$, $m,n \in \mathbb{Z}$, where it has a pole. This function does not satisfy the conditions for Theorem \ref{prop:borcherds-modified-fourier} unless $s - v =0$. This function is analytic on an appropriate strip containing $t=0$.  Without loss of generality we may assume $0 < \im (s-v) < \tau$. In the following picture we can see shaded the domain where the function $f(t)$ is analytic:
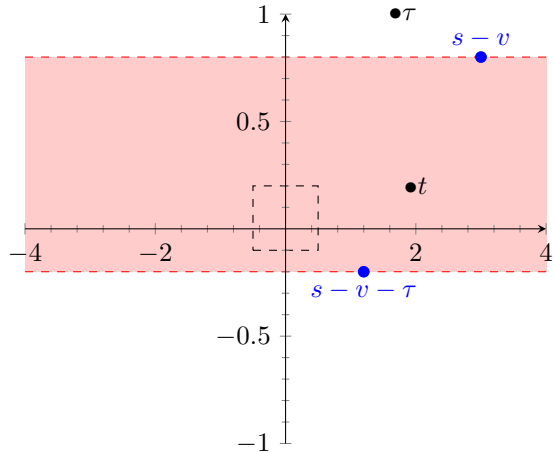
\begin{center}
\begin{tikzpicture}
\begin{axis}[axis lines=middle,minor tick num=4, xmin=-4, xmax=4, ymin=-1, ymax=1,set layers,
	declare function={ycrit(\x)=0.8; y2crit(\x)=-0.2;}]
	\begin{pgfonlayer}{axis background}
		\path[fill=red!20]
	(-4,0.8)
	edge[dashed,red]
	(4,0.8)
       	(4,0.8)
	--
	(4,-0.2)
	edge[dashed,red]
	(-4,-0.2)
	--
	(-4,-0.2)
        |- cycle;
        ;
	\node at (2,0.2) {$\bullet t$};
	\draw [blue,fill] (3,0.8) circle (2pt) node [above] {$s-v$};
	\draw [blue,fill] (1.2,-0.2) circle (2pt) node [below] {$s-v-\tau$};
	\node at (1.8,1) {$\bullet \tau$};
	\draw[dashed] 
	(-0.5,-0.1)
	--
	(0.5,-0.1)
	--
	(0.5,0.2)
	--
	(-0.5,0.2)
	--
	(-0.5,-0.1);
    \end{pgfonlayer}
  \end{axis}
\end{tikzpicture}
\end{center}
That is, $f(t)$ is analytic in the strip $\im(s-v - \tau) < \im t < \im (s-v)$.  The proof of Theorem \ref{prop:borcherds-modified-fourier} is adapted verbatim by replacing the contour of integration there by a contour contained  in this strip (depicted in dashed in the above picture around the point $0$).
Recall from \eqref{eq:Pk.wpk.relation} that the series $P_k(\tfrac{u}{zw}, q)$ converges to $f(t)$ in the domain $|q| < |\tfrac{u}{zw}| < 1$, that is precisely the shaded strip. It follows that the right hand side of \eqref{eq:bmf} is given by the right hand side of \eqref{eq:borcherds-modified-2} times $2 \pi i$. 

We can compute the left hand side of \eqref{eq:bmf} by Laurent expansion.
\begin{equation}
\begin{multlined}
X\bigl(a_{[f]}b,w\bigr) = \res_t f(t) X\bigl(Y[a,t]b,w\bigr) = \sum_{m \geq 0}
\frac{1}{m!} \frac{d^m f}{dt^m}\Bigr|_{t=0} t^m X \bigl( Y[a,t]b,w \bigr) \\ =
\sum_{m \geq 0} \frac{1}{m!} \frac{d^m f}{dt^m}\Bigr|_{t=0} X\bigl( a_{[m]}b,w
\bigr).
\end{multlined}
\label{eq:taylor-wp-shifted}
\end{equation}
Using \eqref{eq:differential-equation} and \eqref{eq:Pk.wpk.relation} we see that the series 
 $(-1)^{m}\tbinom{k+m-1}{m} P_{k+m} \left( \tfrac{u}{w},q \right)$ converges
uniformly in the domain $|q| < |u/w| < 1$, to $\tfrac{1}{m!} \tfrac{d^m f }{dt^m}|_{t=0}$. Therefore the left hand side of \eqref{eq:bmf} reads 
\begin{align*}
\sum_{m \geq 0} (-1)^m \binom{k+m-1}{m} P_{k+m}\left( \frac{u}{w},q \right) X\left( a_{[m]}b,w \right).
\end{align*}

Notice that if we consider the situation when $ - \tau < \im (s-v) < 0$, then the function $f(t)$ would be analytic in the strip $\im (s-v) < \im t < \im (s-v + \tau)$. Repeating the same argument we would obtain the equation:
\begin{equation}
\begin{multlined}
\sum_{m \geq 0}  \binom{k+m-1}{m} P_{k+m}\left( \frac{w}{u},q \right) X\left( a_{[m]}b,w \right) c 
= \frac{1}{2 \pi i} \res_z  P_k\left( \frac{z}{u},q \right) X( a, z) X(b,w) c \\ 
- \frac{1}{2 \pi i}\res_z  P_k\left( \frac{z}{u},q \right)  X(b,w) X(a,z) c,
\label{eq:borcherds-modified-3}
\end{multlined}
\end{equation}
which is equivalent to \eqref{eq:borcherds-modified-2} by \eqref{eq:P-symmetric} after replacing $u$ by $u q^{-1}$.  

\end{proof}
\section{Derivations and Self-Extensions}\label{sec:derivations}
In this section we study self extensions of a module $M$ of a vertex algebra $V$ and the corresponding $\End(M)$-valued derivation of $V$. We state an analogous to the Borcherds formula of Theorem \ref{prop:borcherds-modified-fourier} in this setting. 
\begin{nolabel}
Let $V$ be a vertex algebra and $M$ be a $V$-module. An $\End(M)$-valued derivation of $V$ is a bilinear map $\Psi: V \otimes M \rightarrow M( (z))$ such that for every $f \in \mathbb{C} [ [z,w]][z^{-1},w^{-1},(z-w)^{-1}]$ and every $a,b \in V$, $m \in M$ we have
\begin{align*}
\res_{z-w} \Psi\left( Y(a,z-w)b,w \right)m\, f(z,w)
= {} & \res_z  \bigl( Y(a,z) \Psi(b,w) + \Psi(a,z) Y(b,w) \bigr) m\, i_{z,w} f(z,w) \\ 
& - \res_z \bigl( \Psi(b,w) Y(a,z) + Y(b,w) \Psi(a,z) \bigr)m\, i_{w,z} f(z,w). 
\end{align*}
\end{nolabel}
\begin{ex}
Let $M$ be a $V$-module and $E$ a self extension of $V$-modules, i.e.,
\begin{align*}
0 \rightarrow M \rightarrow E \rightarrow M \rightarrow 0.
\end{align*}
Writing $Y^E(\cdot,z) : V \otimes E \rightarrow E( (z))$ for the $V$-module structure of $E$ and choosing a splitting of the sequence as vector spaces, allows us to write in matrix form 
\begin{equation}
Y^E(a,z) = 
\begin{pmatrix}
Y^M (a,z) & \Psi(a,z) \\ 
0 & Y^M(a,z)
\end{pmatrix}.
\label{eq:Psi-def}
\end{equation}
Then \eqref{eq:module-def} for the $V$-module $E$ implies that $\Psi$ is an $\End(M)$-valued derivation. Conversely, given an $\End(M)$-valued derivation $\Psi$, we define $Y^E(a,z)$ on $M \oplus M$ by \eqref{eq:Psi-def} and obtain a self extension of $M$. 
\end{ex}
\begin{nolabel}
As in Section \ref{no:modified-operators} we introduce the modified derivation $\sigma(a,z) = z^{-1} \Psi(z^{L_0}a,z)$ and we write
\begin{align*}
\Psi(a,z) = \sum_{n \in \mathbb{Z}} \Psi(a)_{(n)} z^{-1-n} \quad \text{and} \quad \sigma(a,z) = \sum_{n \in \mathbb{Z}} \sigma(a)_n z^{-n-1}.
\end{align*}
\label{no:modified-derivation}
In the same way as in Theorem \ref{prop:borcherds-modified-fourier}, Corollary \ref{prop:borcherds-modified} and Proposition \ref{prop:borcherds-modified-2}  we prove
\end{nolabel}
\begin{prop} Let $M$ be a $V$-module, $X(\cdot,z) = X^M(\cdot,z)$ the modified vertex operators. Let $\Psi(\cdot,z)$ be a derivation and $\sigma(\cdot,z)$ the corresponding modified operator as in \ref{no:modified-derivation}. For $a,b \in V$ we have:
\begin{equation}
\sigma\left( a_{[0]}b,w \right) = \frac{1}{ 2 \pi i} \res_z \Bigl( \sigma(a,z) X(b,w) + X(a,z) \sigma(b,w) \Bigr) - \frac{1}{ 2 \pi i} \res_z \Bigl(  \sigma(b,w) X(a,z)  + X(b,w) \sigma(a,z) \Bigr)
\label{eq:0-th-borcherds-b}
\end{equation}
\begin{equation}
\begin{split}
\sigma\left( a_{[\wp_1]} b,w \right) - G_2(q) \sigma\left( a_{[1]}b,w \right) &+ \pi i \sigma\left( a_{[0]}b,w \right)  
\\ 
&= 
 -\frac{1}{2 \pi i} \res_z \left( P_1 \left( \frac{zq}{w},q \right) -  2\pi i\right)  \Bigl( \sigma(a,z) X(b,w) + X(a,z) \sigma(b,w) \Bigr) \\
&\quad +\frac{ 1 }{2 \pi i}
res_z P_1\left( \frac{z}{w},q \right) \Bigl( \sigma(b,w) X(a,z) + X(b,w) \sigma(a,z) \Bigr),
\label{eq:borcherds-modified-1-ba}
\end{split}
\end{equation}
\begin{equation}
\begin{split}
\sigma\left( a_{[\wp_2]} b,w \right) + G_2(q) X\left( a_{[0]}b,w \right)  &=  \frac{1}{2 \pi i} \res_z P_2 \left( \frac{zq}{w},q \right)  \Bigl( \sigma(a,z) X(b,w) + X(a,z) \sigma(b,w) \Bigr) \\
&\quad -\frac{ 1 }{2 \pi i}
res_z P_2\left( \frac{z}{w},q \right) \Bigl( \sigma(b,w) X(a,z) + X(b,w) \sigma(a,z) \Bigr),
\label{eq:borcherds-modified-1-bb}
\end{split}
\end{equation}
\begin{equation}
\begin{split}
\sigma\left( a_{[\wp_k]} b,w \right) &=  \frac{(-1)^{k}}{2 \pi i} \res_z P_k \left( \frac{zq}{w},q \right)  \Bigl( \sigma(a,z) X(b,w) + X(a,z) \sigma(b,w) \Bigr) \\
&\quad -\frac{ (-1)^{k} }{2 \pi i}
res_z P_k\left( \frac{z}{w},q \right) \Bigl( \sigma(b,w) X(a,z) + X(b,w) \sigma(a,z) \Bigr), \qquad k \geq 3
\label{eq:borcherds-modified-1-bc}
\end{split}
\end{equation}
{\multlinegap=130pt
\begin{equation}
\begin{multlined}
\sum_{m \geq 0} (-1)^m \binom{k+m-1}{m} P_{k+m}\left( \frac{u}{w},q \right) \sigma\left( a_{[m]}b,w \right) c 
= \\ \frac{1}{2 \pi i} \res_z  P_k\left( \frac{u}{z},q \right) \Bigl( \sigma( a, z) X(b,w) +  
X(a,z) \sigma(b,w) \Bigr) c 
\\ - \frac{1}{2 \pi i}\res_z  P_k\left( \frac{u}{z},q \right)  \Bigl( X(b,w) \sigma(a,z) + \sigma(b,w) X(a,z)\Bigr) c.
\label{eq:borcherds-modified-2-b}
\end{multlined}
\end{equation}
}
\label{prop:borcherds-modified-b}
\end{prop}
\section{Higher Trace Functions}\label{sec:higher.traces}
In this section we define and study $n$--point linear functionals associated with a self extension of a $V$-module $M$. These functionals are formal power series, defined as graded traces of products of an $\End(M)$-valued derivation and $n$ vertex operators acting on $M$. We start with the $n=1$ case in  \ref{no:f1-def}--\ref{thm:cor-der-dq} and then treat the general case in \ref{prop:actual-trace-prop}--\ref{prop:8.15}. In Lemma \ref{lem:residue-trading} we prove that these linear functionals satisfy the formal differential equations \eqref{eq:connection-tn} which correspond to moving the marked points. We show that the functionals vanish on boundaries in Proposition \ref{prop:new-deriv-generic} in the $n=1$ case and in Proposition \ref{prop:8-trace-d2} in the general case. We show that they are formal solutions of the modular differential equation \eqref{eq:flatness-0} in Theorems \ref{thm:cor-der-dq} and \ref{thm:last-diff-eq}. We prove a recursive formula expressing an $(n+1)$--point linear functional in terms of $n$--point linear functionals in \ref{prop:8.15}, which we refer to as the \emph{insertion formula}. The motif in all proofs is similar: one uses Theorem \ref{prop:borcherds-modified-fourier} to obtain two terms involving different Fourier expansions of the same elliptic function, inside of the trace functional. Then one commutes one of these two terms using the cyclic property of the trace and the commutation relation \eqref{eq:q0-xoperators}, of the modified vertex operators with $L_0$. We often exploit the fact that most of the Fourier series appearing are expansions of elliptic functions, with the exception of $P_1$ which is an expansion of $\wp_1$. One needs to keep track of the different domains of convergence of the involved Fourier transforms. 

In this section we fix a vertex algebra $V$, its module $M$ and an $\End(M)$-valued derivation $\Psi$ as in \ref{no:modified-derivation}. We let $X(\cdot, z)$ and $\sigma(\cdot, z)$ be the associated modified operators.
\begin{nolabel}
Recall the ring $\cF_{n+1}$ of meromorphic elliptic functions as defined in \ref{no:meromorphic-functions}, and its subring $\bJ_*^{n+1}$ of weak Jacobi forms of index $0$. For elements $a_0,\dots,a_n \in V$ and $f(t_0,\dots,t_n,\tau) \in \cF_{n+1}$, we let $F(z_0,\dots,z_n,q)$ be the Fourier expansion of $f$ in the domain $|q z_0| < |z_n| < \dots < |z_1| < |z_0|$,
as in \ref{no:fourier-n-variable}, and we define 
\begin{equation}
F_1^n\left( a_0,\dots,a_n,F(z_0,\dots,z_n,q) \right) = \res_{z_0} z_1\dots z_n F(z_0,\dots,z_n,q) \tr_M \sigma(a_0,z_0) X(a_1,z_1)\dots X(a_n,z_n) q^{L_0}.
\label{eq:8.1}
\end{equation}
We will prove below (see Theorem \ref{thm:convergence1}) that, under certain finiteness conditions on $V$, this series converges in the domain $|qz_1| < |z_n| < \dots < |z_1|$, to an element of $\cF_{n}$, which we denote $F_1^n\left( a_0,\dots,a_n, f \right)$. 
\label{no:f1-def}
\end{nolabel}
\begin{rem} In this section we will make extensive use of Fourier expansions of elliptic functions in different domains. To avoid notational clutter, we will adhere to the following conventions. Whenever we write a Fourier series expansion $F(z_0,\dots,z_n, q)$ of an elliptic function $f(t_0,\dots,t_n, \tau)$ in $n+1$ variables, followed by $n+1$ field insertions in a given order, it is tacitly understood that $F$ is the expansion whose domain of convergence is given by the order the variables in the product of fields as it is in \eqref{eq:8.1}. Occasionally other domains of expansion play a role. We write $F_> (z_0,\dots,z_n,q)$ to mean the Fourier expansion in the domain given by decreasing order of variable modulus, in the same order as specified by the field insertion. For example, in the expression
\[ F_> \bigl(z_0, z_1,z_2, q\bigr) X(b,z_1) \sigma(a,z_0) X(c,z_2), \]
The series $F_>(z_0,z_1,z_2, q)$ converges in the domain $|q z_1| < |z_2| < |z_0|  < |z_1|$.

When in need of a more verbose notation, we will write a subindex with the explicit order of the expansion. Hence, in the example above we would have
\[ F_{z_1,z_0,z_2}(z_0,z_1,z_2,q) = F_>(z_0,z_1,z_2,q). \]
\label{rem:ordering}
\end{rem}
\begin{rem} The right hand side of \eqref{eq:8.1} is well defined for any formal series $F(z_0,\dots,z_n,q)$ not necessarily the Fourier expansion of an elliptic function. We will denote the corresponding functional also with the left hand side of \eqref{eq:8.1}, noting the points where the argument is not elliptic. 
\label{rem:9.3}
\end{rem}
\begin{lem} The series $F_1^n$ satisfies
\begin{gather} \label{eq:lem821}
\begin{split} 
z_i \frac{d}{dz_i} F_1^n\Bigl( a_0,\dots,a_n, F(z_0,\dots,z_n,q) \Bigr)
&= F_1^n\Bigl(  a_0,\ldots, (L_{-1} + L_0) a_i,\ldots,a_n, F(z_0,\cdots,z_n,q) \Bigr) \\
&\quad + F_1^n \Bigl( a_0,\cdots,a_n, z_i \frac{d}{dz_i} F(z_0,\ldots,z_n,q) \Bigr), \qquad 1 \leq i \leq n, 
\end{split} \\
F^n_1\left(   (L_{-1} + L_0) a_0,\ldots,a_n, F(z_0,\cdots,z_n,q) \right) + F^n_1
\left( a_0,\cdots,a_n, z_0 \frac{d}{dz_0} F(z_0,\ldots,z_n,q)  \right) = 0.
\label{eq:8.trading-2}
\end{gather}

\label{lem:residue-trading}
\end{lem}
\begin{proof}
Follows directly from \eqref{eq:residue-trading-exp} and \eqref{eq:x-translation}.
\end{proof}
\begin{nolabel}
For simplicity we first describe the situation in the case $n=1$. Let $a,b \in V$ and $f(t_0,t_1,\tau) \in \cF_2$. We have
\begin{align*}
a_{[f]}b = \res_{t_1 - t_0} dt_1 f(t_1 - t_0) Y[a,t_1 -t_0]b \in V \otimes \cF_1.
\end{align*}
If $f(t_0,t_1,\tau) \in \bJ^2_*$ then $a_{[f]}b$ is a linear combination of the $\wp_k$-products defined in \eqref{no:wpk-propduct}, and we have in fact $a_{[f]}b \in V \otimes \bQM_*$.

Let $f(t_0,t_1,\tau) \in \cF_2$ with Fourier series $F(z_0, z_1, q)$ convergent in the domain $|qz_0| < |z_1| < |z_0|$. Abusing notation, we have $f(t_0,t_1,\tau) = f(t_1 - t_0, \tau)$, so that $F(z_0,z_1,q) = F\left( \frac{z_1}{z_0},q \right)$ and
\begin{align*}
F_1^1 \left( a,b,F\left(z_0, z_1,q \right) \right) = \res_{z_0} z_1 F\left( \frac{z_1}{z_0},q \right) \tr_M \sigma(a,z_0) X(b,z_1) q^{L_0}.
\end{align*}
We will prove in Corollary \ref{cor:sum-of-omega} below that it is independent of $z_1$. We now prove a proposition which plays an important technical role in what follows. This is a degree $1$ analog of \cite[4.3.3]{zhu}. Recall that we have introduced $\zeta(t, \tau) = \wp_1(t, \tau) - g_2(\tau)t$ and that $P_1(z, q)$ is the expansion, in the domain $|q| < |z| < 1$, of $-\zeta(t, \tau) - \pi i$.
\end{nolabel}
\begin{prop} Let $a,b,c \in V$ and let $f(t_0,t_1,\tau) = f(t_1-t_0,\tau) \in \cF_2$ with Fourier expansion $F\left( \frac{z_1}{z_0},q \right)$ convergent in the domain $|qz_0| < |z_1| < |z_0|$. Then
\begin{align}
\begin{split}
-\res_{z_0} z_1 & F\left( \frac{z_1}{z_0},q \right) \tr_M a_0 \,\, \sigma(b,{z_0}) X(c,z_1) q^{L_0} \\
= {} &
\res_{z_0} z_1 P_1\left( \frac{z_1}{z_0},q \right) \tr_M \sigma(a,z_0) X\left( b_{[f]}c,z_1 \right) q^{L_0} - \res_{z_0} z_1 F\left( \frac{z_1}{z_0},q \right) \tr_M \sigma(b,z_0) X(a_{[\zeta]}c, z_1) q^{L_0} \\
&+ \pi i \res_{z_0} z_1 F\left( \frac{z_1}{z_0},q \right) \tr_M \sigma(b, z_0) X(a_{[0]}c, z_1) q^{L_0} \\
&- \sum_{m \geq 0} (-1)^m \res_{z_0} z_1 P_{m+1}\left( \frac{z_1}{z_0},q \right) F\left( \frac{z_1}{z_0},q \right) \tr_M \sigma\left( a_{[m]}b, z_0 \right) X(c,z_1) q^{L_0}.
\end{split}
\label{eq:recursion-1}
\end{align}
\label{prop:recursion}
\end{prop}
\begin{proof}
Recall from \ref{no:fourier-n-variable} that $F(\tfrac{z_1 q}{u},q)$ is the Fourier expansion of $f$ in the domain $|z_1 q| < |u| < |z_1|$. 
Using Theorem \ref{prop:borcherds-modified} we obtain
\begin{align}
\begin{split}
\res_{z_0} z_1 & P_1\left( \frac{z_1}{z_0},q \right) \tr_M \sigma(a,z_0) X\left( b_{[f]}c,z_1 \right) q^{L_0} \\
= {} & \frac{1}{ 2 \pi i} \res_{z_0} z_1 P_1\left( \frac{z_1}{z_0},q \right) \tr_M \sigma(a,z_0) \res_u X(b,u) X(c,z_1) F\left( \frac{z_1}{u},q \right) q^{L_0} 
\\
&- \frac{1}{ 2 \pi i} \res_{z_0} z_1 P_1\left( \frac{z_1}{z_0},q \right) \tr_M \sigma(a,z_0) \res_u X(c,z_1) X(b,u) F\left(\frac{z_1 q}{u},q \right) q^{L_0} \\
= {} & \frac{1}{ 2 \pi i} \res_{z_0} \res_u z_1 P_1\left( \frac{z_1}{u},q \right) F\left( \frac{z_1}{z_0},q \right)  \tr_M \sigma(a,u) X(b,z_0) X(c,z_1) q^{L_0} \\
&- \frac{1}{2 \pi i} \res_{z_0} \res_u z_1 P_1\left( \frac{z_1}{u},q \right) F\left(\frac{z_1q}{z_0},q \right) \tr_M \sigma(a,u) X(c,z_1) X(b,z_0) q^{L_0}.
\end{split}
\label{eq:rec-1-1}
\end{align}
Similarly, using that $P_1(zq,q)$ converges in the domain $1 < |z| < |q^{-1}|$ to $- \zeta(t,\tau) + \pi i$, while $P_1(z,q)$ converges in the domain $|q| < |z| <1$ to $- \zeta(t,\tau) - \pi i$ we obtain:
\begin{align}
\begin{split}
-\res_{z_0} z_1 & F\left( \frac{z_1}{z_0},q \right) \tr_M \sigma(b, z_0) X(a_{[\zeta]}c, z_1) q^{L_0} + \pi i \res_{z_0} z_1 F\left( \frac{z_1}{z_0},q \right)\tr_M \sigma(b, z_0) X(a_{[0]}c, z_1) q^{L_0} \\
= {} &
\frac{1}{ 2 \pi i}\res_{z_0} \res_u z_1 F\left( \frac{z_1}{z_0},q \right) P_1\left(\frac{uq}{z_1},q  \right) \tr_M \sigma(b,z_0) X(a,u) X(c,z_1) q^{L_0} \\
&-\frac{1}{2 \pi i}\res_{z_0} \res_u z_1 F\left( \frac{z_1}{z_0},q \right) \left( P_{1} \left(\frac{u}{z_1},q  \right) + 2 \pi i \right) \tr_M \sigma(b,z_0) X(c,z_1) X(a,u) q^{L_0}. 
\end{split}
\label{eq:rec-1-2}
\end{align}
And by \eqref{eq:borcherds-modified-2-b} we obtain
\begin{align}
\begin{split}
&-\sum_{m \geq 0} (-1)^m \res_{z_0} z_1 P_{m+1}\left( \frac{z_1}{z_0},q \right) F\left( \frac{z_1}{z_0},q \right) \tr_M \sigma(a_{[m]}b, z_0) X(c, z_1) q^{L_0} \\
= {} &
-\frac{1}{2 \pi i} \res_{z_0} z_1 F\left( \frac{z_1}{z_0},q \right) \tr_M \res_u P_1\left( \frac{z_1}{u},q \right) \sigma(a,u) X(b,z_0) X(c,z_1) q^{L_0} \\
&- \frac{1}{2 \pi i} \res_{z_0} z_1 F\left( \frac{z_1}{z_0},q \right)  \tr_M \res_u P_1\left( \frac{z_1}{u},q \right) X(a,u) \sigma(b,z_0) X(c,z_1) q^{L_0} \\
&+ \frac{1}{2 \pi i} \res_{z_0} z_1 F\left( \frac{z_1}{z_0},q \right) \tr_M \res_u P_1\left( \frac{z_1}{u},q \right) \sigma(b,z_0) X(a,u)  X(c,z_1) q^{L_0} \\
&+ \frac{1}{2 \pi i} \res_{z_0} z_1 F\left( \frac{z_1}{z_0},q \right) \tr_M
\res_u P_1\left( \frac{z_1}{u},q \right)  X(b,z_0) \sigma(a,u) X(c,z_1) q^{L_0}
.
\end{split}
\label{eq:rec-1-3}
\end{align}
The first term of the right hand side of \eqref{eq:rec-1-3} cancels the first term of the right hand side of \eqref{eq:rec-1-1}. Using \eqref{eq:P-symmetric}, the third term of the right hand side of \eqref{eq:rec-1-3} cancels the first term of the right hand side of \eqref{eq:rec-1-2}. In the fourth term of the right hand side of \eqref{eq:rec-1-3} we commute $X(b,z_0)$ with $q^{L_0}$ using \eqref{eq:q0-xoperators} and that the trace is a symmetric functional, so this term cancels the second term of the right hand side of \eqref{eq:rec-1-1}. In the second term of the right hand side of \eqref{eq:rec-1-3} we commute $X(a,u)$ with $q^{L_0}$ (using \eqref{eq:q0-xoperators} again) to obtain 
\begin{align*}
&- \frac{1}{2 \pi i} \res_{z_0} z_1 F\left( \frac{z_1}{z_0},q \right) \tr_M \res_u P_1\left( \frac{z_1}{u},q \right)\sigma(b,z_0) X(c,z_1) q^{L_0}  X(a,u) \\ 
= &- \frac{1}{2 \pi i} \res_{z_0} z_1 F\left( \frac{z_1}{z_0},q \right)  \tr_M \res_u P_1\left( \frac{z_1}{u},q \right)\sigma(b,z_0) X(c,z_1)  X(a,qu)q q^{L_0} \\ 
= &- \frac{1}{2 \pi i} \res_{z_0} z_1 F\left( \frac{z_1}{z_0},q \right)  \tr_M \res_u P_1\left( \frac{z_1q}{u},q \right)\sigma(b,z_0) X(c,z_1)  X(a,u) q^{L_0}. 
\end{align*}
Now, using $P_1( \tfrac{z_1q}{u},q) = - P_1(\tfrac{u}{z_1},q)$, this term plus the second term on the right hand side of \eqref{eq:rec-1-2} yields
\begin{align*}
&- \res_{z_0} \res_u z_1 F\left( \frac{z_1}{z_0},q \right) \tr_M \sigma(b,z_0) X(c,z_1) X(a,u) q^{L_0} \\
= &-\res_{z_0} z_1 F\left( \frac{z_1}{z_0},q \right) \tr_M a_0 \sigma(b,z_0) X(c,z_1) q^{L_0}. 
\end{align*}
This proves the proposition. 
\end{proof}
Recall the complex $C_\bullet$ defined in \ref{no:complex-def}, computing the chiral homology in genus $1$ of the vertex algebra $(Y,\vac,\tilde{\omega}, Y[\cdot,z])$. To lighten the burden of notation we write $a \otimes b \otimes f$ for a typical element of $\ker d_1$, implicitly we shall always mean finite sums of such terms. The following proposition shows that $F^1_1$ vanishes on coboundaries.
\begin{prop} Let $a,b,c \in V$ and let $f = f(t_0,t_1,t_2,\tau) \in \cF_3$. Let
$F(z_0,z_1,z_2,q)$ be the Fourier expansion of $f$ in the domain $|qz_0| < |z_2|
< |z_1| < |z_0|$. Let $f_{ijk}$ be the coefficients of the Laurent expansion of
$f$ near $t_i-t_j = 0$, that is $f = \sum_k f_{ijk} (t_i-t_j)^k$ with $f_{ijk}
\in \cF_2$. Let $F_{ijk} (z,q)$ be the Fourier expansion of $f_{ijk}$ in the
domain $|q|<|z|<1$. Recall the definition of $d_2 (a \otimes b \otimes c \otimes
f)$ from \eqref{eq:boundary}. Then the following identity holds
\begin{multline} \label{eq:8.19}
\sum_k F^1_1\left( a, b_{[k]} c, F_{12k}\left( \frac{z_1}{z_0},q \right) \right) - \sum_k F^1_1\left( b,a_{[k]}c,F_{02k}\left( \frac{z_1}{z_0},q \right) \right) \\ - \sum_k F^1_1\left( a_{[k]}b,c,F_{01k}\left( \frac{z_1}{z_0},q \right) \right) = 0.
\end{multline}
\label{prop:new-deriv-generic}
\end{prop}
\begin{proof} Using Theorem \ref{prop:borcherds-modified-fourier} and the convention of Remark \ref{rem:ordering}, the first term of \eqref{eq:8.19} can be written as
\begin{multline}
\frac{1}{2 \pi i} \res_{u} \res_{z_0} z_1 F_>(u,z_0,z_1) \tr_M \sigma(a,u)
X(b,z_0) X(c,z_1) q^{L_0} \\ - \frac{1}{2 \pi i}  \res_{u} \res_{z_0} z_1
F_>(u,z_0,z_1) \tr_M \sigma(a,u) X(c,z_1) X(b,z_0) q^{L_0} .
\label{eq:8.7.1}
\end{multline}
Similarly the second and third term are respectively written as 
\begin{multline}
-\frac{1}{2 \pi i } \res_{z_0} \res_u z_1 F_>(u, z_0,z_1) \tr_M \sigma(b,z_0) X(a,u) X(c,z_1) q^{L_0} \\ 
+ \frac{1}{2 \pi i} \res_{z_0} \res_u z_1 F_>(u,z_0,z_1 ) \tr_M \sigma(b,z_0) X(c,z_1) X(a,u) q^{L_0}.
\label{eq:8.7.2}
\end{multline}
\begin{multline}
-\frac{1}{2 \pi i} \res_{z_{0}} \res_u z_1 F_>(u,z_0,z_1) \tr_M \left( \sigma(a,u) X(b,z_0) + X(a,u) \sigma(b,z_0) \right)  X(c,z_1) q^{L_0} \\ 
+ \frac{1}{2\pi i} \res_{z_0} \res_u z_1 F_>(u,z_0,z_1 ) \tr_M \left( X(b,z_0) \sigma(a,u) + \sigma(b,z_0) X(a,u) \right) X(c,z_1) q^{L_0}.
\label{eq:8.7.3}
\end{multline}
The first term of \eqref{eq:8.7.3} cancels the first term of \eqref{eq:8.7.1} and the fourth term of \eqref{eq:8.7.3} cancels the first term of \eqref{eq:8.7.2}. Commuting $X(b,z_0)$ with $q^{L_0}$ using \eqref{eq:q0-xoperators} the third term of \eqref{eq:8.7.3} can be written as 
\[ \frac{1}{ 2 \pi i} \res_{z_0} \res_u F_{z_0,u,z_1}\left(u,q^{-1} z_0,z_1,q \right) \sigma(a,u)X(c,z_1)X(b,z_0)q^{L_0}.\]
Noting that 
\[ F_{z_0,u,z_1}(u,q^{-1}z_0,z_1,q) = F_{u,z_1,z_0}(u,z_0,z_1,q), \]
since they are two Fourier expansions of the same function in the same domain, we see that this term
 cancels the second term of \eqref{eq:8.7.1}. Finally by the same reasoning the second term of \eqref{eq:8.7.3} cancels the second term of \eqref{eq:8.7.2} after commuting $X(a,u)$ with $q^{L_0}$ using \eqref{eq:q0-xoperators}. 
\end{proof}
\begin{cor} Let $a \otimes b \otimes f \in \ker d_1$. The linear functional $F^1_1$ satisfies:
\begin{equation}
 z_1 \frac{d}{d z_1} F^1_1 \left( a,b,F\left( \frac{z_1}{z_0},q \right) \right) = 0.
\label{eq:sum-of-omega}
\end{equation}
\label{cor:sum-of-omega}
\end{cor}
\begin{proof}
Apply proposition \ref{prop:new-deriv-generic} to $\tilde{\omega}\otimes a \otimes b \otimes \tfrac{F(z_1,z_2,q)}{2 \pi i}$. We have
\[ d_2 \left( \tilde{\omega} \otimes a \otimes b \otimes \frac{f(t_1,t_2,\tau)}{2 \pi i} \right) = \tilde{\omega} \otimes a_{[f]}b \otimes \frac{1}{2 \pi i} - \tilde{\omega}_{[0]} a \otimes b \otimes \frac{f(t_1,t_2,\tau)}{2 \pi i} - a \otimes \tilde{\omega}_{[0]}b \otimes \frac{f(t_1,t_2,\tau)}{2\pi i}. \]
The first term in the right hand side vanishes by our assumption. 
Using that $\tilde{\omega}_{[0]} = 2 \pi i (L_0 + L_{-1})$ obtain
\begin{equation} F^1_1 \left( (L_0+L_{-1})a, b, F\left( \frac{z_1}{z_0},q \right) \right) + F^1_1 \left( a, (L_0+L_{-1}) b, F\left( \frac{z_1}{z_0},q \right) \right) = 0. 
\label{eq:cor-8.1}
\end{equation}
On the other hand, it follows from Lemma \ref{lem:residue-trading} that the left hand side of \eqref{eq:cor-8.1} equals 
\[ z_1 \frac{d}{d z_1} F^1_1 \left( a,b,F\left( \frac{z_1}{z_0},q \right) \right). \]
\end{proof}
\begin{thm}
Let $b,c \in V$ and let $F(z_0,z_1,q) = F\left(\frac{z_1}{z_0},q\right)$ be the Fourier series of $f(t_0,t_1,\tau) \in \cF_2$ in the domain $|qz_0| < |z_1| < |z_0|$. Recall the differential operator $D = (2 \pi i) \tfrac{d}{d \tau} + \zeta(t_1-t_0,\tau) \tfrac{d}{d t_0}$ of \eqref{eq:di-def}. 
Let $G\left( \frac{z_1}{z_0},q \right)$ be the Fourier series of $g(t_0,t_1,\tau) = Df(t_0,t_1,\tau) \in \cF_2$. Then the following differential equation holds
\begin{equation}
\begin{multlined}
(2 \pi i)^2 \left[ q \frac{d}{dq} - \frac{c}{24} \right] F^1_1\left( b, c, F\left( \frac{z_1}{z_0},q \right) \right) 
=   F^1_1 \left( b, \tilde{\omega}_{[\zeta]} c, F\left( \frac{z_1}{z_0},q \right) \right) \\
+ F^1_1 \left( b,c, P_2 \left( \frac{z_1}{z_0},q \right) F\left( \frac{z_1}{z_0},q \right) \right)
+ F^1_1\left( b,c, G\left( \frac{z_1}{z_0},q \right) \right) \\
- \sum_{k \geq 0} (-1)^k F^1_1 \left( \tilde{\omega}_{[k+1]}b, c, P_{k+2}\left( \frac{z_1}{z_0},q \right) F\left( \frac{z_1}{z_0},q \right)\right) 
- F^1_1 \left( \tilde{\omega},  b_{[f]}c, P_1 \left( \frac{z_1}{z_0},q \right) \right).
\end{multlined}
\label{eq:dq-der-1}
\end{equation}
Notice that the argument of $F^1_1$ in the last term on the right hand side is not the Fourier expansion of an elliptic function. This term however vanishes if $b \otimes c \otimes f \in \ker d_1$. 
\label{thm:cor-der-dq}
\end{thm}
\begin{proof}
We have from \eqref{eq:4.14b}:
\begin{align*}
G\left(\frac{z_1}{z_0}, q\right) = (2\pi i)^2 q \frac{d}{d q} F\left(\frac{z_1}{z_0}, q\right) - 2\pi i \left[P_1\left(\frac{z_1}{z_0}, q\right) + \pi i\right] z_0 \frac{d}{d z_0} F\left(\frac{z_1}{z_0}, q\right).
\end{align*}

Putting $a = \tilde{\omega} = (2 \pi i)^2 (\omega - \tfrac{c}{24} \vac)$ in Proposition \ref{prop:recursion} yields
\begin{align}
\begin{split}
-\res_{z_0} z_1 & F\left( \frac{z_1}{z_0},q \right) \tr_M \tilde{\omega}_{0} \,\, \sigma(b,{z_0}) X(c,z_1) q^{L_0} \\
= {} &
\res_{z_0} z_1 P_1\left( \frac{z_1}{z_0},q \right) \tr_M \sigma(\tilde{\omega},z_0) X\left( b_{[f]}c,z_1 \right) q^{L_0} - \res_{z_0} z_1 F\left( \frac{z_1}{z_0},q \right) \tr_M \sigma(b,z_0) X(\tilde{\omega}_{[\zeta]}c, z_1) q^{L_0} \\
&+ \pi i \res_{z_0} z_1 F\left( \frac{z_1}{z_0},q \right) \tr_M \sigma(b, z_0) X(\tilde{\omega}_{[0]}c, z_1) q^{L_0} \\
&- \sum_{k \geq 0} (-1)^k \res_{z_0} z_1 P_{k+1}\left( \frac{z_1}{z_0},q \right) F\left( \frac{z_1}{z_0},q \right) \tr_M \sigma\left( \tilde{\omega}_{[k]}b, z_0 \right) X(c,z_1) q^{L_0}.
\end{split}\label{eq:8.6-1}
\end{align}
The fourth term equals
\begin{equation}
\begin{multlined}
\sum_{k \geq 0} (-1)^k F^1_1 \left( \tilde{\omega}_{[k+1]}b,c, P_{k+2} \left( \frac{z_1}{z_0},q \right) F\left( \frac{z_1}{z_0},q \right) \right)\\
-\res_{z_0} z_1 P_{1}\left( \frac{z_1}{z_0},q \right) F\left( \frac{z_1}{z_0},q \right) \tr_M \sigma\left( \tilde{\omega}_{[0]}b, z_0 \right) X(c,z_1) q^{L_0}.
\label{eq:assuming-primary}
\end{multlined}
\end{equation}
Recalling that $\tilde{\omega}_{[0]} = (2 \pi i) (L_0 + L_{-1})$, and using Lemma \ref{lem:residue-trading}, we expand the second term of \eqref{eq:assuming-primary} as
\begin{equation*}
\begin{split}
&- (2 \pi i) \res_{z_0} z_1 P_1\left( \frac{z_1}{z_0},q \right) F\left( \frac{z_1}{z_0},q\right) \tr_M  \sigma \left( (L_0 + L_{-1})b,z_0 \right)X(c,z_1) q^{L_0} \\ 
=&(2 \pi i)  \res_{z_0} z_1 z_0\frac{d}{dz_0} \left[ P_1\left( \frac{z_1}{z_0},q \right) F\left( \frac{z_1}{z_0},q\right) \right] \sigma \left(b,z_0 \right)X(c,z_1) q^{L_0}.
\end{split}
\end{equation*}
Using \eqref{eq:pkdiff} this last expression equals:
\begin{equation}
\begin{split}
=&-F^1_1 \left( b, c, P_2 \left( \frac{z_1}{z_0},q \right) F\left( \frac{z_1}{z_0},q\right) \right) \\  
&+(2 \pi i) \res_{z_0}z_1 P_1\left( \frac{z_1}{z_0},q \right) z_0 \frac{d}{dz_0} \left( F\left( \frac{z_1}{z_0},q\right) \right) \sigma\left(b,z_0 \right) X(c,z_1) q^{L_0}.
\end{split}
\label{eq:zdz-2term}
\end{equation}
We rewrite the second term of the right hand side of \eqref{eq:zdz-2term} in terms of the function $G$, obtaining
\begin{align}
\begin{split}
-& F^1_1 \left(b, c, G\left( \frac{z_1}{z_0},q \right) \right)
- (2\pi i)(\pi i) F^1_1 \left(b, c, z_0 \frac{d}{d z_0} F\left( \frac{z_1}{z_0},q \right) \right) \\
+& (2 \pi i)^2 \res_{z_0} z_1 \left[ q \frac{d}{dq} F \left( \frac{z_1}{z_0},q \right) \right] \tr_M \sigma(b, z_0) X(c, z_1) q^{L_0}.
\end{split}
\label{eq:some-g-second}
\end{align}
Next we use $\tilde{\omega}_{(\deg \tilde{\omega} -1)} = (2 \pi i)^{2} (L_0 - \tfrac{c}{24})$ to see that:
\begin{multline}
(2 \pi i)^2 \left[ q  \frac{d}{dq} - \frac{c}{24} \right]  F^1_1 \left( b,c,F\left( \frac{z_1}{z_0},q \right) \right) \\ 
= \res_{z_0} z_1 F\left( \frac{z_1}{z_0},q \right) \tr_M \tilde{\omega}_{(\deg \tilde{\omega} -1)} \sigma(b,z_0) X(c,z_1) q^{L_0} \\
+ (2 \pi i)^2 \res_{z_0} z_1 \left[ q \frac{d}{dq} F\left( \frac{z_1}{z_0},q \right) \right] \tr_M \sigma(b,z_0) X(c,z_1) q^{L_0}.
\label{eq:theorem-trace-deriv-def}
\end{multline}
Finally we use \eqref{eq:lem821} and Corollary \ref{cor:sum-of-omega} to obtain
\[
F^1_1 \left( b,\tilde{\omega}_{[0]}c, F\left( \frac{z_1}{z_0},q \right) \right)
=
2\pi i F^1_1\left( b,c, z_0 \frac{d}{dz_0} F\left( \frac{z_1}{z_0},q \right) \right).
\]
Putting all this together gives the claimed differential equation.
\end{proof}
We now proceed to analyze the general $n \geq 1$ situation. The proofs are essentially the same as above with a little bit of notational clutter. Care must be taken when considering the domain of expansions of the respective Fourier series. 
The following proposition expresses the trace of the zero mode as a linear combination of $n$-point functions, this is the analog of \cite[Prop 4.3.3]{zhu} and a generalization of Proposition \ref{prop:recursion}. 
\begin{prop} Let $a,a_0,\dots,a_n \in V$, $n \geq 1$ and $f(t_0,\dots,t_n,\tau) \in \cF_{n+1}$. Let $F(z_0,\dots,z_n,q)$ be the Fourier expansion of $f$ in the domain $|qz_0|<|z_n|<\dots<|z_0|$. For each $1 \leq i \leq n$ consider the Laurent expansion of $f$ near $t_0=t_i$, that is 
\begin{align*} f(t_0,\dots,t_n,\tau) = \sum_k f_{0ik}(t_1,\dots,t_n,\tau) (t_0-t_i)^k. \end{align*}
Let $F_{0ik}(z_1,\dots,z_n,q)$ be the Fourier expansion of $f_{0ik}$ in the domain $|qz_1| < |z_n| < \dots < |z_1|$. The following holds
\begin{equation}
\begin{multlined}
+ \sum_{i=0}^{n-1} \sum_{m \geq 0} (-1)^m  F^n_1 \left( a_0,\dots,  a_{[m]}a_i,\dots,a_n, F(z_0,\dots,z_n,q)P_{m+1}\left( \frac{z_n}{z_i},q \right) \right) \\ 
+ F^n_1 \left( a_0,\dots,a_{n-1},a_{[\zeta(t) - \pi i]}a_n, F(z_0,\dots,z_n,q)   \right) \\ 
-\sum_{i=1}^n \sum_k F^n_1 \left( a,a_1,\dots,a_{0[k]}a_i,\dots,a_n,F_{0ik}(z_1,\dots,z_n,q) P_1\left( \frac{z_n}{z_0},q \right) \right) = \\ 
\res_{z_0} z_1\dots z_n F(z_0,\dots,z_n,q) \tr_M a_{(\deg a -1)}
\sigma(a_0,z_0)X(a_1,z_1)\dots X(a_n,z_n) q^{L_0}.
\label{eq:actual-trace-prop}
\end{multlined}
\end{equation}
\label{prop:actual-trace-prop}
\end{prop}
\begin{proof}
Using \eqref{eq:borcherds-modified-2-b} we express the term $i=0$ in the first summand of \eqref{eq:actual-trace-prop} as 
\begin{equation}
\begin{multlined}
\frac{1}{ 2 \pi i } \res_{z_0} \res_u z_1 \dots z_n F(z_0,\dots,z_n,q) P_1 \left( \frac{z_n}{u},q \right) \tr_M \sigma(a,u) X(a_0,z_0) \dots X(a_n,z_n)q^{L_0} \\ 
+ \frac{1}{ 2 \pi i } \res_{z_0} \res_u z_1 \dots z_n F(z_0,\dots,z_n,q) P_1 \left( \frac{z_n}{u},q \right)  \tr_M X(a,u) \sigma(a_0,z_0) X(a_1,z_1) \dots X(a_n,z_n)q^{L_0} \\ 
-\frac{1}{ 2 \pi i } \res_{z_0} \res_u z_1 \dots z_n F(z_0,\dots,z_n,q) P_1 \left( \frac{z_n}{u},q \right) \tr_M X(a_0,z_0) \sigma(a,u) X(a_1,z_1) \dots X(a_n,z_n) q^{L_0}  \\ 
-\frac{1}{ 2 \pi i } \res_{z_0} \res_u z_1 \dots z_n F(z_0,\dots,z_n,q) P_1 \left( \frac{z_n}{u},q \right) \tr_M \sigma(a_0,z_0) X(a,u) X(a_l,z_1) \dots X(a_n,z_n) q^{L_0}.
\label{eq:trace-1-1}
\end{multlined}
\end{equation}
The remaining terms corresponding to  $ 1 \leq i \leq n-1$ can be written as follows using Proposition \ref{prop:borcherds-modified-2}
\begin{equation}
\begin{multlined}
\frac{1}{ 2 \pi i } \sum_{i=1}^{n-1} \res_{z_0} \res_u z_1 \dots z_n F(z_0,\dots,z_n,q) P_1 \left( \frac{z_n}{u},q \right) \\ 
\tr_M \sigma(a_0,z_0) X(a_1,z_1) \dots X(a_{i-1},z_{i-1}) X(a,u) X(a_i,z_i) \dots X(a_n,z_n)q^{L_0} \\ 
-
\frac{1}{ 2 \pi i } \sum_{i=2}^{n-1} \res_{z_0} \res_u z_1 \dots z_n F(z_0,\dots,z_n,q) P_1 \left( \frac{z_n}{u},q \right) \\ 
\tr_M \sigma(a_0,z_0) X(a_1,z_1) \dots X(a_{i},z_{i}) X(a,u) X(a_{i+1},z_{i+1}) \dots X(a_n,z_n)q^{L_0}.
\label{eq:trace-1-2}
\end{multlined}
\end{equation}
This is a telescoping sum and the only two terms that survive are
\begin{equation}
\begin{multlined}
\frac{1}{ 2 \pi i } \res_{z_0} \res_u z_1 \dots z_n F(z_0,\dots,z_n,q) P_1 \left( \frac{z_n}{u},q \right) \tr_M \sigma(a_0,z_0) X(a,u) X(a_1,z_1)\dots X(a_n,z_n)q^{L_0} \\ 
-\frac{1}{ 2 \pi i } \res_{z_0} \res_u z_1 \dots z_n F(z_0,\dots,z_n,q) P_1 \left( \frac{z_n}{u},q \right) \\ 
\tr_M \sigma(a_0,z_0) X(a_1,z_1) \dots X(a_{n-1},z_{n-1}) X(a,u) X(a_n,z_n)q^{L_0}.
\label{eq:trace-1-2-bis}
\end{multlined}
\end{equation}
The second term using Proposition \ref{prop:borcherds-modified} is written as 
\begin{equation}
{\multlinegap=200pt
\begin{multlined}
F_1^n\Bigl(a_0,\dots,a_{n-1}, a_{[\zeta(t)-\pi i]}a_n, F(z_0,\dots,z_n,q)\Bigr) = \\ 
 \frac{1}{2 \pi i} \res_{z_0} \res_u z_1 \dots z_n F(z_0,\dots,z_n,q) P_1 \left( \frac{z_n}{u},q \right) \\ \tr_M \sigma(a_0,z_0) X(a_1,z_1) \dots X(a_{n-1},z_{n-1}) X(a,u) X(a_n,z_n) q^{L_0} \\ 
- \frac{1}{2 \pi i} \res_{z_0} \res_u z_1 \dots z_n F(z_0,\dots,z_n,q) \left( P_1\left( \frac{z_nq}{u},q \right)  - 2\pi i  \right) \\ 
\tr_M \sigma(a_0,z_0)X(a_1,z_1)\dots X(a_n,z_n)X(a,u)  q^{L_0} .
\label{eq:trace1-2-a}
\end{multlined}
}
\end{equation}
We see that the first two terms of \eqref{eq:actual-trace-prop} add up to 
\begin{equation}
\begin{multlined}
\frac{1}{ 2 \pi i } \res_{z_0} \res_u z_1 \dots z_n F(z_0,\dots,z_n,q) P_1 \left( \frac{z_n}{u},q \right) \tr_M \sigma(a,u) X(a_0,z_0) \dots X(a_n,z_n)q^{L_0} \\ 
+ \frac{1}{ 2 \pi i } \res_{z_0} \res_u z_1 \dots z_n F(z_0,\dots,z_n,q) P_1 \left( \frac{z_n}{u},q \right)  \tr_M X(a,u) \sigma(a_0,z_0) X(a_1,z_1) \dots X(a_n,z_n)q^{L_0} \\ 
-\frac{1}{ 2 \pi i } \res_{z_0} \res_u z_1 \dots z_n F(z_0,\dots,z_n,q) P_1 \left( \frac{z_n}{u},q \right) \tr_M X(a_0,z_0) \sigma(a,u) X(a_1,z_1) \dots X(a_n,z_n) q^{L_0}  \\ 
- \frac{1}{2 \pi i} \sum_{i=1}^n \res_{z_0} \res_u z_1 \dots z_n F(z_0,\dots,z_n,q) \left( P_1\left( \frac{z_nq}{u},q \right)  - 2\pi i  \right) \\ 
\tr_M \sigma(a_0,z_0)\dots X(a_n,z_n)X(a,u)  q^{L_0} .
\label{eq:first+second-trace}
\end{multlined}
\end{equation}
In order to compute the third term of \eqref{eq:actual-trace-prop}, we recall from \ref{no:fourier-n-variable} that $F(q^{-1}z_0, z_1,\dots,z_n,q)$ is the Fourier series of $f(t_0,\dots,t_n,\tau)$ convergent in the domain $|z_0| < |z_n| < \dots <|z_1| < |q^{-1} z_0|$. Using Proposition \ref{prop:borcherds-modified-fourier} we see that this term is a telescoping sum and the two surviving terms are 
\begin{equation}
\begin{multlined}
-\frac{1}{2\pi i} \res_{z_0} \res_u  z_1\dots z_n F(z_0,\dots,z_n,q) P_1\left( \frac{z_n}{u} \right) \tr_M \sigma(a,u)X(a_0,z_0) X(a_2,z_2) \dots X(a_n,z_n) q^{L_0} \\ 
+ \frac{1}{2\pi i} \res_{z_0} \res_u  z_1\dots z_n F(q^{-1} z_0,\dots,z_n,q) P_1\left( \frac{z_n}{u} \right) \tr_M \sigma(a,u) X(a_1,z_1) \dots X(a_{n},z_{n}) X(a_0,z_0) q^{L_0} .
\label{eq:third-term-trace}
\end{multlined}
\end{equation}
The first term of \eqref{eq:third-term-trace} cancels the first term in \eqref{eq:first+second-trace}. Commuting $X(a_0,z_0)$ with $q^{L_0}$ using \eqref{eq:q0-xoperators} in the third term of \eqref{eq:first+second-trace} that term cancels the second term of \eqref{eq:third-term-trace}. Finally commuting $X(a,u)$ with $q^{L_0}$ in the second term of \eqref{eq:first+second-trace} we see that this term plus the fourth term add up to 
\[
\res_{z_0} \res_u z_1 \dots z_n F(z_0,\dots,z_n,q) \tr_M \sigma(a_0,z_0) X(a_1,z_1) \dots X(a_n,z_n) X(a,u) q^{L_0}, 
\]
proving the proposition.
\end{proof}
Recall the complex $C_\bullet^n$ of \ref{defn:chiral-homology-with-supports} computing the chiral homology of the vertex algebra $(V, Y[\cdot,z],\tilde{\omega}, \vac)$ with coefficients in $V^{\otimes n}$ in genus $1$. The following proposition shows that $F^n_1$ vanishes on boundaries. 
\begin{prop} Let $a, b, a_1,\dots,a_n \in V$ and $f = f(u,v,t_1,\dots,t_{n},\tau) \in \cF_{n+2}$. Recall the definition of $d_2$ in \eqref{eq:d2-general-def}. We let $f_{0ik}$, $f_{1ik}$ and $f_{00k}$ be the coefficients of the Laurent expansions of $f$ about $u-t_i$, $v - t_i$ and $u-v$ respectively as in \eqref{eq:5.12b}. Let $F_{0ik}(z_0,\dots,z_n,q)$ and $F_{1ik}(z_0,\dots,z_n,q)$ and $F_{00k}(z_0,\dots,z_n,q)$  be their respective Fourier expansions, convergent in the domain
\[ |z_0 q| < |z_n| < \dots < |z_1| < |z_0|. \]
The linear functional $F^n_1$ satisfies:
\begin{equation}
\begin{multlined}
\sum_{i=1}^n \sum_{k} F^n_1 \Bigl( a,a_1,\dots, b_{[k]}a_i,\dots,a_n, F_{1ik} \Bigr) - \sum_{i=1}^n \sum_k F^n_1 \Bigl( b, a_1,\dots, a_{[k]} a_i, \dots,a_n, F_{0ik} \Bigr) \\ 
- \sum_{k} F^n_1 \Bigl( a_{[k]}b, a_1,\dots,a_n, F_{00k} \Bigr) = 0.
\end{multlined}
\label{eq:8-trace-d2}
\end{equation}
\label{prop:8-trace-d2}
\end{prop}
\begin{proof}
Using Theorem \ref{prop:borcherds-modified-fourier} and the convention of Remark \ref{rem:ordering}, the first term of \eqref{eq:8-trace-d2} can be written as
\[
\begin{multlined}
\frac{1}{2 \pi i} \sum_{i = 1}^n  \res_{z_0} \res_{u} F_>(z_0,u, z_1,\dots, z_n,q) z_1 \dots z_n  \tr_M \sigma(a,z_0) X(a_1,z_1) \dots X(b,u) X(a_i,z_i) \dots X(a_n,z_n) q^{L_0}  \\ 
- \frac{1}{2 \pi i} \sum_{i = 1}^n \res_{z_0} \res_u F_>(z_0, u, z_1,\dots, z_n,q)  z_1 \dots z_n   
\tr_M \sigma(a,z_0) X(a_1,z_1) \dots X(a_i,z_i) X(b,u) \dots X(a_n,z_n) q^{L_0}. 
\end{multlined}
\]
This is a telescoping sum and the only surviving terms after relabeling $u \leftrightarrow z_0$ are:
\begin{equation}
\begin{multlined}
\frac{1}{2 \pi i}  \res_{z_0} \res_{u} z_1\dots z_n F_>(u,z_0, z_1,\dots, z_n,q)  \tr_M \sigma(a,u) X(b,z_0) X(a_1,z_1) \dots X(a_n,z_n) q^{L_0}  \\ 
- \frac{1}{2 \pi i}  \res_{z_0} \res_u z_1 \dots z_n F_>(u, z_0, z_1,\dots, z_n, q)  
\tr_M \sigma(a,u) X(a_1,z_1)  \dots X(a_n,z_n) X(b,z_0) q^{L_0}. 
\label{eq:8-d2-1}
\end{multlined}
\end{equation}
Similarly, the second term can be written as:
\begin{equation}
\begin{multlined}
-\frac{1}{2 \pi i} \res_{z_0}\res_u z_1 \dots z_n  F_>(u,z_0,z_1, \dots, z_n,q) \tr_M \sigma(b,z_0) X(a,u) X(a_1,z_1) \dots X(a_n,z_n) q^{L_0}  \\ 
+ \frac{1}{2 \pi i} \res_{z_0} \res_u z_1\dots z_nF_>(u, z_0 , \dots, z_n, q)  \tr_M \sigma(b,z_0) X(a_1,z_1) \dots X(a_n,z_n) X(a,u) q^{L_0}. 
\end{multlined}
\label{eq:8-d2-2}
\end{equation}
Finally the third term is written using Proposition \ref{prop:borcherds-modified-b} as 
\begin{equation}
\begin{multlined}
-\frac{1}{2 \pi i} \res_{z_0} \res_u z_1 \dots z_n F_>(u,z_0,\dots,z_n,q) \tr_M \Bigl(\sigma(a,u) X(b,z_0) + X(a,u) \sigma(b,z_0) \Bigr) X(a_1,z_1) \dots X(a_n,z_n) q^{L_0} \\ 
+\frac{1}{2 \pi i} \res_{z_0} \res_u z_1 \dots z_n F_>(u,z_0,\dots,z_n,q) \tr_M
\Bigl(\sigma(b,z_0) X(a,u) + X(b,z_0) \sigma(a,u) \Bigr) X(a_1,z_1) \dots
X(a_n,z_n) q^{L_0}.
\end{multlined}
\label{eq:8-d2-3}
\end{equation}
The first summands of each term in \eqref{eq:8-d2-3} cancel the first summands
of \eqref{eq:8-d2-1} and \eqref{eq:8-d2-2}. The second term of \eqref{eq:8-d2-3} after commuting with $q^{L_0}$ using \eqref{eq:q0-xoperators} equals
\[ -\frac{1}{2 \pi i} \res_{z_0} \res_u F_{u,z_0,\dots,z_n} (q^{-1} u,z_0,\dots,z_n,q) \sigma(b,z_0)X(a_1,z_1) \dots X(a_n,z_n) X(a,u) q^{L_0}. \]
Noting that 
\[ F_{u,z_0,\dots,z_n} (q^{-1} u,z_0,\dots,z_n,q) = F_{z_0,\dots,z_n,u}(u,z_0,\dots,z_n,q), \]
this term cancels the second term of \eqref{eq:8-d2-2}. Similarly the fourth term of \eqref{eq:8-d2-3} cancels the second term of \eqref{eq:8-d2-1}. 
\end{proof}
\begin{lem} Let $f(t_0,\dots,t_n, \tau) \in \cF_{n+1}$ and consider its Laurent expansion
\[ f = \sum_k f_{0ik}(t_1,\dots,t_n,\tau) (t_0-t_i)^k. \]
The series $F^n_1$ satisfies
\[ 
\begin{multlined}
2 \pi i \sum_{i=1}^n z_i\frac{d}{d z_i} F^n_1 \Bigl(a_0,\dots,a_n, F(z_0,\dots,z_n,q) \Bigr) = \\ 
\sum_{i=1}^n \sum_k F^n_1 \Bigl( \tilde{\omega},a_1,\dots,a_{0[k]}a_i,\dots,a_n, F_{0ik}(z_1,\dots,z_n, q) \Bigr)
\end{multlined}
\]
\label{lem:9-sum-of-deriv}
\end{lem}
\begin{proof}
From \eqref{eq:lem821} the left hand side equals:
\[ 
\begin{multlined}
2  \pi i \sum_{i=1}^{n} F^n_1 \Bigl(a_0, a_1,\dots,(L_{-1} + L_0)a_i, \dots,a_n, F(z_0,\dots,z_n,q) \Bigr) \\  
+  2 \pi i F^n_1 \left( a_0,\dots,a_n, \sum_{i=1}^n z_i \frac{d}{dz_i} F(z_0,\dots,z_n,q)  \right). 
\end{multlined}
\]
Using \eqref{eq:8.trading-2}, Lemma \ref{lem:4.sum-deriv} and that $\tilde{\omega}_{[0]} = (2 \pi i) (L_{-1} + L_0)$, this can be written as 
\[ \sum_{i=0}^{n} F^n_1 \Bigl(a_0, a_1,\dots, \tilde{\omega}_{[0]}a_i, \dots,a_n, F(z_0,\dots,z_n,q) \Bigr)  \]
The lemma follows from Proposition \ref{prop:8-trace-d2} with $a = \tilde{\omega}$ and $b = a_0$. 
\end{proof}
\begin{nolabel} Let $a_0,\dots,a_n  \in V$ and $F(z_0,\dots,z_n,q)$ be the Fourier series of $f(t_0,\dots,t_n,\tau) \in \cF_{n+1}$ in the domain $|qz_0| < |z_n| < \dots < |z_1| < |z_0|$.
Recall the definition of $Df \in \cF_{n+1}$ in \ref{no:p1-mult-der}. Let $\overline D$ be the differential operator $D$ in the exponential variables $q,z_i$ as in \eqref{eq:4.14b}, that is 
\[
\overline{D}F = \left( (2 \pi i)^2 q \frac{d}{dq} - \sum_{i=0}^{n-1} \left(
P_1\left( \frac{z_n}{z_i},q \right) + \pi i \right) 2 \pi i z_i \frac{d}{d z_i}
\right) F(z_1,\dots,z_n,q) ,
\]
is the Fourier series of $Df$ convergent in the domain $|qz_0| < |z_n| < \dots < |z_1| < |z_0|$ as in \eqref{eq:4.14b}. For any formal series $S(z_1,\dots,z_n,q)$ on $n+1$ variables $z_1,\dots,z_n,q$, we will denote by $\overline{D}S$ the differential operator above with the sum starting from $i=1$. 
\label{no:new-diff-variables}
\end{nolabel}
\begin{thm} Let $a_0,\dots,a_n \in V$ and $f = f(t_0,\dots,t_n,\tau) \in \cF_{n+1}$. 
  The following differential equation holds
\begin{equation}
\begin{multlined}
 \left( \overline{D}  - (2 \pi i)^2 \frac{c}{24} \right)   F^n_1\left( a_0,\dots,a_n,F(z_0,\dots,z_n,q) \right) = \\
F^n_1 \Bigl(a_0,\dots,a_{n-1}, \tilde{\omega}_{[\zeta]}a_n, F(z_0,\dots,z_n,q) \Bigr)   
+ F^n_1 \Bigl( a_0,\dots,a_n, \overline{D}F(z_0,\dots,z_n,q)  \Bigr)  \\
+ F_1^n \left( a_0,\dots,a_n, F(z_0,\dots,z_n,q) P_2\left( \frac{z_n}{z_0},q \right) \right) \\ 
- \sum_{i=0}^{n-1} \sum_{m \geq 0} (-1)^m  F^n_1 \left( a_0,\dots,  \tilde{\omega}_{[m+1]}a_i,\dots,a_n, F(z_0,\dots,z_n,q)P_{m+2}\left( \frac{z_n}{z_i},q \right) \right)  \\ 
- \sum_{i = 1}^n \sum_{k} F^n_1 \left( \tilde{\omega}, a_1,\dots a_{0[k]}a_i,\dots,a_n, F_{0ik}(z_1,\dots,z_n,q) \left( P_1\left( \frac{z_n}{z_0},q \right) + \pi i \right) \right).
\end{multlined}
\end{equation}
The last term vanishes if $a_0 \otimes \dots \otimes a_n \otimes f \in \ker d_1$.
\label{thm:last-diff-eq}
\end{thm}
\begin{proof}
Consider Proposition \ref{prop:actual-trace-prop} with $a = \tilde{\omega} = (2 \pi i)^2 \left( \omega - \tfrac{c}{24} \vac \right)$, \eqref{eq:actual-trace-prop} reads
\begin{equation}
\begin{multlined}
 \sum_{i=0}^{n-1} \sum_{m \geq 0} (-1)^m  F^n_1 \left( a_0,\dots,  \tilde{\omega}_{[m]}a_i,\dots,a_n, F(z_0,\dots,z_n,q)P_{m+1}\left( \frac{z_n}{z_i},q \right) \right) \\ 
+ F^n_1 \left( a_0,\dots,a_{n-1},\tilde{\omega}_{[\zeta(t) - \pi i]}a_n, F(z_0,\dots,z_n,q)   \right) \\
- \sum_{i = 1}^n \sum_{k} F^n_1 \left( \tilde{\omega}, a_1,\dots a_{0[k]}a_i,\dots,a_n, F_{0ik}(z_1,\dots,z_n,q)P_1\left( \frac{z_n}{z_0},q \right) \right) = \\ 
\res_{z_0} z_1\dots z_n F(z_0,\dots,z_n,q) \tr_M \tilde{\omega}_{(\deg \tilde{\omega} -1)}  \sigma(a_0,z_0)X(a_1,z_1)\dots X(a_n,z_n) q^{L_0}
\end{multlined}
\label{eq:actual-trace-prop-2}
\end{equation} 
Using that $\tilde{\omega}_0 = (2 \pi i)^2 (L_0 - \tfrac{c}{24})$, the right hand side equals
\begin{equation}
\begin{multlined}
 (2 \pi i)^2 \left( q \frac{d}{dq} - \frac{c}{24} \right) F^n_1 \Bigl(a_0,\dots,a_n, F(z_0,\dots,z_n,q) \Bigr) \\ 
- (2 \pi i)^2 \res_{z_0} z_1\dots z_n \left( q\frac{d}{dq} F(z_0,\dots,z_n,q) \right) \tr_M \sigma(a_0,z_0) X(a_1,z_1)\dots X(a_n,z_n) q^{L_0}. 
\end{multlined}
\label{eq:8.28}
\end{equation}
Using that $\tilde{\omega}_{[0]} = (2\pi i)(L_0 + L_{-1})$ and \eqref{eq:8.trading-2}, the summand $i=m=0$ of the first term in \eqref{eq:actual-trace-prop-2} equals
\begin{equation}
\begin{multlined}
 - 2 \pi i F^n_1 \left( a_0,\dots,a_n, P_1\left( \frac{z_n}{z_0},q \right) z_0 \frac{d}{d z_0} F(z_0,\dots,z_n,q) \right) \\ +  F^n_1 \left( a_0,\dots,a_n, F(z_0,\dots,z_n,q) P_2 \left( \frac{z_n}{z_0},q \right) \right). 
\end{multlined}
\label{eq:8.29}
\end{equation}
Similarly using \eqref{eq:lem821} the summands $m = 0$ and $1 \leq i \leq n-1$ in the first term of \eqref{eq:actual-trace-prop-2} equal:
\begin{equation}
\begin{multlined}
2 \pi i \sum_{i = 1}^{n-1} z_i \frac{d}{d z_i} F^n_1 \left(a_0,\dots,a_n, F(z_0,\dots,z_n,q) P_1\left( \frac{z_n}{z_i
},q \right)  \right) \\ 
- 2 \pi i \sum_{i=1}^{n-1} F^n_1 \left( a_0,\dots,a_n, P_1\left( \frac{z_n}{z_i},q \right) z_i \frac{d}{d z_i} F(z_0,\dots,z_n,q) \right) \\ 
+ \sum_{i = 1}^{n-1} F^n_1 \left( a_0,\dots,a_n, P_2 \left( \frac{z_n}{z_i},q \right)  F(z_0,\dots,z_n,q) \right) = \\ 
\sum_{i=1}^{n-1} P_1\left( \frac{z_n}{z_i},q \right) 2 \pi i z_i\frac{d}{d z_i} F^n_1\left( a_0,\dots,a_n,F(z_0,\dots,z_n,q) \right) \\ 
- 2  \pi i \sum_{i=1}^{n-1} F^n_1 \left( a_0,\dots,a_n, P_1\left( \frac{z_n}{z_i},q \right) z_i \frac{d}{d z_i} F(z_0,\dots,z_n,q) \right).
\end{multlined}
\label{eq:8.30}
\end{equation}
We see that the second term of the right hand side of  \eqref{eq:8.30} plus the first term of \eqref{eq:8.29} minus the second term of \eqref{eq:8.28} equal
\begin{equation}
\begin{multlined}
 F^n_1 \Bigl( a_0,\dots,a_n, \overline{D}F(z_0,\dots,z_n,q)  \Bigr) + \frac{(2 \pi i)^2}{2} \sum_{i = 0}^{n-1}  F^n_1 \left( a_0,\dots,a_n, z_i \frac{d}{dz_i} F(z_0,\dots,z_n,q) \right) \\ = 
 F^n_1 \Bigl( a_0,\dots,a_n, \overline{D}F(z_0,\dots,z_n,q)  \Bigr) - \frac{(2 \pi i)^2}{2}  F^n_1 \left( a_0,\dots,a_n, z_n \frac{d}{dz_n} F(z_0,\dots,z_n,q) \right),
\end{multlined}
\label{eq:8.38}
\end{equation}
where we used Lemma \ref{lem:4.sum-deriv} in the last equation. We use Lemma \ref{lem:residue-trading} to express
\begin{equation}
\begin{multlined}
 F^n_1\Bigl(a_0,\dots,a_{n-1}, \tilde{\omega}_{[-\pi i]} a_n, F(z_0,\dots,z_n,q) \Bigr) = \frac{(2 \pi i)^2}{2} F^n_1 \Bigl( a_0,\dots,a_n, z_n \frac{d}{d z_n} F(z_0,\dots,z_n,q) \Bigr) \\ 
- \frac{(2 \pi i)^2}{2} z_n \frac{d}{d z_n} F^n_1 \Bigl( a_0,\dots,a_n,
F(z_0,\dots,z_n, q)\Bigr).
\end{multlined}
\label{eq:8.39}
\end{equation}
The first term cancels the second term of \eqref{eq:8.38} and the second term can be written as 
\[
\begin{multlined}
 \frac{(2 \pi i)^2}{2} \sum_{i=1}^{n-1} z_i \frac{d}{d z_i} F^n_1\Bigl(a_0,\dots,a_n, F(z_0,\dots,z_n,q) \Bigr) \\ 
- \pi i 
\sum_{i=1}^n \sum_k F^n_1 \Bigl( \tilde{\omega},a_1,\dots,a_{0[k]}a_i,\dots,a_n,
F_{0ik}(z_1,\dots,z_n, q) \Bigr),
\end{multlined}
\]
using Lemma \ref{lem:9-sum-of-deriv}.
Replacing in \eqref{eq:actual-trace-prop-2} the Theorem follows. 
\end{proof}
\begin{nolabel} Let $1 \leq j \leq n+1$ and $f = f(t_0,\dots,t_{j-1},s,t_j,\dots,t_n, \tau) \in \cF_{n+2}$. 
We let 
\[ h(v,s,t_0,\dots,t_n, \tau) = f(t_0,\dots,t_{j-1},v,t_j,\dots,t_n, \tau) \left( \zeta\left( s-v,\tau \right) - \pi i  \right). \]
We will need the Laurent expansions of $h$ and $f$ in powers of  $v - t_i$:
\[
\begin{aligned}
f(t_0,\dots,v,t_j,\dots,t_n, \tau) &=  \sum_{k} f_{vik} (t_0,\dots,t_n,\tau) (v-t_i)^k,\\  
h(v,s,t_0,\dots,t_n,\tau) &= \sum_k h_{vik} (s,t_0,\dots,t_n,\tau) (v-t_i)^k, \\
h_{vik}(s,t_0,\dots,t_n, \tau) &= f_{vik}(t_0,\dots,t_n,\tau) \left( \zeta(s-t_i,\tau) - \pi i \right) \\ 
& \quad + f_{vi,k-1}(t_0,\dots,t_n,\tau) \left(\wp_2(s-t_i,\tau) + g_2(\tau) \right) \\ 
& \quad + \sum_{m \geq 0} f_{vi,k-m-2}(t_0,\dots,t_n,\tau) \wp_{m+3}(s-t_i,\tau)
.
\end{aligned}
\]
Notice that $F(z_0,\dots,u,z_j,\dots,z_n,q) P_1\left( \tfrac{u}{w},q \right)$ is the Fourier expansion of $h$ in the domain $|q z_0| < |z_n| < \dots < |z_j| < |u| < |w| < |z_{j-1}|< \dots < |z_0|$. We let $F_{vik}(z_0,\dots,z_n,q)$ be the Fourier expansion of $f_{vik}(t_0,\dots,t_m, \tau)$, in the domain $|qz_0| < |z_n| < \dots <|z_0|$. Similarly, we consider the Laurent expansion of $f$ in powers of $t_0 - t_i$ for $1 \leq i \leq n$: 
\[ f(t_0,\dots,t_{j-1},v,t_j,\dots,t_n,\tau) = \sum_k f_{0ik} (t_1,\dots,t_{j-1},v,t_j,\dots,t_n,\tau) (t_0-t_i)^k, \]
And let $F_{0ik}(z_1,\dots,z_{j-1},u,z_j,\dots,z_n,q)$ be its Fourier expansion convergent in the domain $|qz_1| < |z_n| < \dots < |z_j| < |u|< |z_{j-1}| < \dots < |z_1|$. 

The following Proposition expresses an $n+1$-point functional in terms of
$n$-point functionals and the trace of the zero mode. This is the analog of
\cite[Prop. 4.3.4]{zhu} in the degree zero case. 
\label{no:prop-9.15-ordering}
\end{nolabel}
\begin{prop}Let $a,a_0,\dots,a_n \in V$, $1 \leq j \leq n+1$, and $f =
f(t_0,\dots,t_{j-1},s,t_j,\dots,t_n,\tau) \in \cF_{n+2}$.  Let
$F(z_0,\dots,z_{j-1},w,z_j,\dots,z_n,q)$ be the Fourier
expansion of $f$ in the domain 
\begin{equation*}
|qz_0| < |z_n| < \dots <|z_{j}| < |w| < |z_{j-1}| < \dots
< |z_1| < |z_0|,
\end{equation*} 
Then 
\begin{equation}
\begin{multlined}
F^{n+1}_1\Bigl(a_0,\dots,a_{j-1},a,a_j,\dots,a_n,F(z_0,\dots,z_{j-1},w,z_j,\dots,z_n,q)
\Bigr) \\ = \res_{z_0} \res_u z_1\dots,z_n F_>(z_0,\dots,z_{j-1},u,z_j,\dots,z_n,q) \tr_M
X(a,u) \sigma(a_0,z_0)X(a_1,z_1) \dots X(a_n,z_n) q^{L_0} \\
+ \sum_{i=0}^{j-1} \sum_k \sum_{m \geq 0} F^n_1 \left( a_0,\dots,a_{[k]}a_i,\dots,a_n,
F_{vi,k-m}(z_0,\dots,z_n,q) \left( P_{m+1}\left( \frac{z_i q}{w},q \right) - 2
\pi i \delta_{m,0} \right) \right) \\
+ \sum_{i=j}^{n} \sum_k \sum_{m \geq 0} F^n_1 \left( a_0,\dots,a_{[k]}a_i,\dots,a_n,
F_{vi,k-m}(z_0,\dots,z_n,q) P_{m+1}\left( \frac{z_i}{w},q \right) \right) \\
-\sum_{i=1}^n \sum_k F^n_1 \left( a, a_1,\dots, a_{0[k]}a_i,\dots,a_n, F_{0ik,z_0,\dots,z_n}(z_1,\dots,z_{j-1},z_0,z_j,\dots,z_n,q)\left( P_1\left( \frac{z_0 q}{w},q \right) - 2 \pi i \right) \right).
\end{multlined}
\label{eq:8.14-prop1}
\end{equation}
Where we have used the notation of \ref{no:prop-9.15-ordering} on the right hand side.
\label{prop:8.14}
\end{prop}
\begin{proof}
It follows from \eqref{eq:delta-log} that 
\begin{equation} \label{eq:8-delta} w^{-1} P_1\left( \frac{u}{w},q \right) - w^{-1} \left( P_1\left( \frac{u q}{w},q \right)-2 \pi i \right) = 2 \pi i \delta(u,w) = 2 \pi i \sum_{n \in \mathbb{Z}} w^{-n-1}u^n. 
\end{equation}
Consider
\begin{equation}
\begin{multlined}
F^{n+1}_1\Bigl(a_0,\dots,a_{j-1},a,a_j,\dots,a_n,F(z_0,\dots,z_{j-1},w,z_j,\dots,z_n,q)
\Bigr) \\ = \res_{z_0} w z_1\dots,z_n F(z_0,,\dots,z_{j-1},w,z_j,\dots,z_n,q) \tr_M \sigma(a_0,z_0)\dots X(a,w)X(a_j,z_j)\dots X(a_n,z_n)q^{L_0} \\ 
=
\frac{1}{2 \pi i} \res_{z_0} \res_u z_1\dots,z_n
F(z_0,\dots,z_{j-1},u,z_j,\dots,z_n,q)P_1\left( \frac{u}{w},q  \right) \\ \tr_M
\sigma(a_0,z_0)\dots X(a_{j-1},z_{j-1})X(a,u)X(a_j,z_j)\dots X(a_n,z_n)q^{L_0} \\ 
-
\frac{1}{2 \pi i} \res_{z_0} \res_u z_1\dots,z_n
F(z_0,\dots,z_{j-1},u,z_j,\dots,z_n,q)\left(P_1\left( \frac{uq}{w},q  \right) -
2\pi i \right) \\ \tr_M \sigma(a_0,z_0)\dots
X(a_{j-1},z_{j-1})X(a,u)X(a_j,z_j)\dots X(a_n,z_n)q^{L_0} .
\label{eq:8.14.1}
\end{multlined}
\end{equation}
Using \eqref{eq:borcherds-modified-3} we can express the first term on the right hand side of \eqref{eq:8.14.1} as
\begin{equation}
{\multlinegap=120pt
\begin{multlined}
\frac{1}{2 \pi i} \res_{z_0} \res_u z_1\dots,z_n
F(z_0,\dots,z_{j-1},u,z_j,\dots,z_n,q)P_1\left( \frac{u}{w},q  \right) \\ \tr_M
\sigma(a_0,z_0)\dots X(a_{j-1},z_{j-1})X(a,u)X(a_j,z_j)\dots X(a_n,z_n)q^{L_0} \\ 
=
\frac{1}{2 \pi i} \res_{z_0} \res_u z_1\dots,z_n
F_>(z_0,\dots,z_{j-1},u,z_j,\dots,z_n,q)P_1\left( \frac{u}{w},q  \right) \\
\tr_M \sigma(a_0,z_0)\dots X(a_j,z_j) X(a,u)X(a_{j+1},z_{j+1})\dots X(a_n,z_n)q^{L_0} \\
+
\sum_{k} \sum_{m \geq 0} \res_{z_0} z_1\dots,z_n
F_{vj,k-m,>}(z_0,\dots,z_n,q)P_{m+1}\left( \frac{z_j}{w},q  \right) \\ \tr_M
\sigma(a_0,z_0)\dots X(a_{[k]}a_j,z_j)\dots X(a_n,z_n)q^{L_0} .
\end{multlined}
\label{eq:8.14.3}
}
\end{equation}
Repeating this argument we can express this term as 
\begin{equation}
{\multlinegap=120pt
\begin{multlined}
\frac{1}{2 \pi i} \res_{z_0} \res_u z_1\dots,z_n
F(z_0,\dots,z_{j-1},u,z_j,\dots,z_n,q)P_1\left( \frac{u}{w},q  \right) \\ \tr_M
\sigma(a_0,z_0)\dots X(a_{j-1},z_{j-1})X(a,u)X(a_j,z_j)\dots X(a_n,z_n)q^{L_0} \\ 
=
\frac{1}{2 \pi i} \res_{z_0} \res_u z_1\dots,z_n F_>(z_0,\dots,z_{j-1},u,z_j,\dots,z_n,q)P_1\left( \frac{u}{w},q  \right) \\ \tr_M \sigma(a_0,z_0)\dots X(a_n,z_n)X(a,u) q^{L_0} \\
+
\sum_{i =j}^n \sum_k \sum_{m \geq 0} \res_{z_0} z_1\dots,z_n
F_{vi,k-m,>}(z_0,\dots,z_n,q)P_{m+1}\left( \frac{z_i}{w},q  \right) \\ \tr_M
\sigma(a_0,z_0)\dots X(a_{[k]}a_i,z_i)\dots X(a_n,z_n)q^{L_0} .
\end{multlined}
\label{eq:8.14.4}
}
\end{equation}
Similarly, using \eqref{eq:borcherds-modified-3} and repeating the same procedure we can express the second term on the right hand side of \eqref{eq:8.14.1} as 
\begin{equation}
\begin{multlined}
-
\frac{1}{2 \pi i} \res_{z_0} \res_u z_1\dots,z_n
F(z_0,\dots,z_{j-1},u,z_j,\dots,z_n,q)\left(P_1\left( \frac{uq}{w},q  \right) -
2\pi i \right) \\ \tr_M \sigma(a_0,z_0)\dots X(a_{j-1},z_{j-1})X(a,u)X(a_j,z_j)\dots X(a_n,z_n)q^{L_0} 
\\ 
= - \frac{1}{2 \pi i} \res_{z_0} \res_u z_1 \dots z_n F_>(z_0,\dots,z_{j-1},u,z_j,\dots,z_n,q) \left( P_1\left( \frac{u q}{w},q \right) - 2 \pi i \right) \\ 
\tr_M \sigma(a_0,z_0) X(a,u) X(a_1,z_1)\dots X(a_n,z_n) q^{L_0} \\ 
+ \sum_{i=1}^{j-1} \sum_k \sum_{m \geq 0} \res_{z_0} z_1\dots z_n
F_{vi,k-m,>}(z_0,\dots,z_n,q) \left( P_{m+1}\left( \frac{z_i q}{w},q\right) - 2 \pi i \delta_{m,0} \right) \\
\tr_M \sigma(a_0,z_0) \dots X(a_{[k]}a_i,z_i) \dots X(a_n,z_n) q^{L_0} .
\end{multlined}
\label{eq:8.14.5}
\end{equation}
Using Proposition \ref{prop:borcherds-modified-b}, the first term on the right hand side of \eqref{eq:8.14.5} can be written as 
\begin{equation}
\begin{multlined}
 - \frac{1}{2 \pi i} \res_{z_0} \res_u z_1 \dots z_n F_>(z_0,\dots,z_{j-1},u,z_j,\dots,z_n,q) \left( P_1\left( \frac{u q}{w},q \right) - 2 \pi i \right) \\ 
\tr_M \sigma(a_0,z_0) X(a,u) X(a_1,z_1)\dots X(a_n,z_n) q^{L_0} \\ 
=- \frac{1}{2 \pi i} \res_{z_0} \res_u z_1 \dots z_n F_>(z_0,\dots,u,z_j,\dots,z_n,q) \left( P_1\left( \frac{u q}{w},q \right) - 2 \pi i \right) \\ 
\tr_M \Bigl( X(a,u)\sigma(a_0,z_0) + \sigma(a,u)X(a_0,z_0) \Bigr) X(a_1,z_1)\dots X(a_n,z_n) q^{L_0} \\ 
+ \frac{1}{2 \pi i} \res_{z_0} \res_u z_1 \dots z_n F_>(z_0,\dots,u,z_j,\dots,z_n,q) \left( P_1\left( \frac{u q}{w},q \right) - 2 \pi i \right) \\ 
\tr_M  X(a_0,z_0)\sigma(a,u) X(a_1,z_1)\dots X(a_n,z_n) q^{L_0} \\
+ \sum_k \sum_{m \geq 0} \res_{z_0} z_1\dots z_n
F_{v0,k-m,>}(z_0,\dots,z_n,q) \left( P_{m+1}\left( \frac{z_0 q}{w},q\right) - 2 \pi i \delta_{m,0} \right) \\
\tr_M \sigma(a_{[k]}a_0,z_0)X(a_1,z_1) \dots X(a_n,z_n) q^{L_0} .
\end{multlined}
\label{eq:8.14.6}
\end{equation}
Commuting $X(a,u)$ with $q^{L_0}$ using \eqref{eq:q0-xoperators} in the first term
of \eqref{eq:8.14.6} we cancel the first term in \eqref{eq:8.14.4} and remains
\[ 
 \res_{z_0} \res_u z_1 \dots z_n F_>(z_0,\dots,u,z_j,\dots,z_n,q) \tr_M
\sigma(a_0,z_0) X(a_1,z_1)\dots X(a_n,z_n) X(a,u) q^{L_0}.
\]
 Commuting
$X(a_0,z_0)$ with $q^{L_0}$ in the second term, we can write this term as
\[
\begin{multlined}
 \frac{1}{2 \pi i} \res_{z_0} \res_u z_1 \dots z_n F_>(z_0,\dots,u,z_j,\dots,z_n,q) \left( P_1\left( \frac{u q}{w},q \right) - 2 \pi i \right) \\ 
\tr_M  X(a_0,z_0)\sigma(a,u) X(a_1,z_1)\dots X(a_n,z_n) q^{L_0}  \\
=
 \frac{1}{2 \pi i} \res_{z_0} \res_u z_1 \dots z_n
F_{z_0,u,z_1,\dots,z_n}(q^{-1} z_0,z_1,\dots,u,z_j,\dots,z_n,q) \left( P_1\left( \frac{u q}{w},q \right) - 2 \pi i \right) \\ 
\tr_M \sigma(a,u) X(a_1,z_1)\dots X(a_n,z_n) X(a_0,z_0) q^{L_0}.
\end{multlined}
\]
Noting that $F_{z_0,u,z_1,\dots,z_n}(q^{-1}z_0,\dots,u,z_j,\dots,z_n,q) =
F_{u,z_1,\dots,z_n,z_0}(z_0,\dots,u,z_j,\dots,z_n)$ this last
expression equals
\[
\begin{multlined}
\frac{1}{2 \pi i} \res_{z_0} \res_u z_1 \dots z_n
F_{>}(z_0,\dots,u,z_j,\dots,z_n,q) \left( P_1\left( \frac{u q}{w},q \right) - 2 \pi i \right) \\ 
\tr_M \sigma(a,u) X(a_1,z_1)\dots X(a_n,z_n) X(a_0,z_0) q^{L_0}.
\end{multlined}
\]
Repeated use of Proposition \ref{prop:borcherds-modified-fourier} expresses this
term as 
\begin{equation}
\begin{multlined}
\frac{1}{2 \pi i} \res_{z_0} \res_u z_1 \dots z_n
F_{>}(z_0,\dots,u,z_j,\dots,z_n,q) \left( P_1\left( \frac{u q}{w},q \right) - 2 \pi i \right) \\ 
\tr_M \sigma(a,u)X(a_0,z_0) X(a_1,z_1)\dots X(a_n,z_n) q^{L_0} \\
- \sum_{i=1}^n \sum_k \res_{u} z_1\dots z_n F_{0ik,>}(z_1,\dots,u,z_j,\dots,z_n,q)  \left( P_1\left( \frac{u q}{w},q \right) -
2 \pi i \right)\\  \tr_M \sigma(a,u)X(a_1,z_1)\dots X(a_{0[k]}a_i,z_i)\dots
X(a_n,z_n)q^{L_0}.
\end{multlined}
\label{eq:8.14.7}
\end{equation}
The second term is the last term of \eqref{eq:8.14-prop1} and the first term cancels
the second summand on the first term in \eqref{eq:8.14.6}, proving the
Proposition.  
\end{proof}
We will need the following generalization of Proposition \ref{prop:actual-trace-prop}:
\begin{prop} Let $a,a_0,\dots,a_n \in V$, $n \geq 1$, $1 \leq j \leq n+1$ and $f(t_0,\dots,t_{j-1},s,t_j,\dots,t_n,\tau) \in \cF_{n+2}$. With the notation of \ref{no:prop-9.15-ordering}, the following holds:
\begin{equation}
\begin{multlined}
+ \sum_{i=0}^{n-1} \sum_k \sum_{m \geq 0} (-1)^m  F^n_1 \left( a_0,\dots,  a_{[k]}a_i,\dots,a_n, F_{vi,k-m}(z_0,\dots,z_n,q)P_{m+1}\left( \frac{z_n}{z_i},q \right) \right) \\ 
+ \sum_k F^n_1 \Bigl( a_0,\dots,a_{[k]}a_n, F_{vn,k+1}(z_0,\dots,z_n,q) \Bigr) \\ 
- \pi i \sum_k F^n_1 \Bigl( a_0,\dots,a_{[k]}a_n, F_{vn,k}(z_0,\dots,z_n,q) \Bigr) \\ 
- \sum_k \sum_{j \geq 1}  F^n_1 \Bigl( a_0,\dots,a_{[k]}a_n, G_{2j}(q) F_{vn,k+1-2j}(z_0,\dots,z_n,q) \Bigr) \\ 
-\sum_{i=1}^n \sum_k F^n_1 \left( a,a_1,\dots,a_{0[k]}a_i,\dots,a_n,F_{0ik,z_0,\dots,z_n}(z_1,\dots,z_{j-1},z_0,z_j,\dots,z_n,q) P_1\left( \frac{z_n}{z_0},q \right) \right) = \\ 
\res_{z_0} \res_u z_1\dots z_n F_>(z_0,\dots,z_{j-1},u,z_{j+1},\dots,z_n,q)
\tr_M X(a,u)  \sigma(a_0,z_0)X(a_1,z_1)\dots X(a_n,z_n) q^{L_0}.
\label{eq:actual-trace-propb}
\end{multlined}
\end{equation}
\label{prop:actual-trace-propb}
\end{prop}
\begin{proof}
Using \eqref{eq:borcherds-modified-2-b} we express the term $i=0$ in the first summand of \eqref{eq:actual-trace-propb} as 
\begin{equation}
\begin{multlined}
\frac{1}{ 2 \pi i } \res_{z_0} \res_u z_1 \dots z_n F_>(z_0,\dots,z_{j-1},u,z_j,\dots,z_n,q) P_1 \left( \frac{z_n}{u},q \right) \tr_M \sigma(a,u) X(a_0,z_0) \dots X(a_n,z_n)q^{L_0} \\ 
+ \frac{1}{ 2 \pi i } \res_{z_0} \res_u z_1 \dots z_n F_>(z_0,\dots,z_{j-1},u,z_j,\dots,z_n,q) P_1 \left( \frac{z_n}{u},q \right) \\  \tr_M X(a,u) \sigma(a_0,z_0) X(a_1,z_1) \dots X(a_n,z_n)q^{L_0} \\ 
-\frac{1}{ 2 \pi i } \res_{z_0} \res_u z_1 \dots z_n F_>(z_0,\dots,z_{j-1},u,z_j,\dots,z_n,q) P_1 \left( \frac{z_n}{u},q \right) \\  \tr_M X(a_0,z_0) \sigma(a,u) X(a_1,z_1) \dots X(a_n,z_n) q^{L_0}  \\ 
-\frac{1}{ 2 \pi i } \res_{z_0} \res_u z_1 \dots z_n F_>(z_0,\dots,z_{j-1},u,z_j,\dots,z_n,q) P_1 \left( \frac{z_n}{u},q \right)  \\ \tr_M \sigma(a_0,z_0) X(a,u) X(a_l,z_1) \dots X(a_n,z_n) q^{L_0}.
\label{eq:trace-1-1b}
\end{multlined}
\end{equation}
The remaining terms corresponding to  $ 1 \leq i \leq n-1$ can be written as follows using Proposition \ref{prop:borcherds-modified-2}
\begin{equation}
\begin{multlined}
\frac{1}{ 2 \pi i } \sum_{i=1}^{n-1} \res_{z_0} \res_u z_1 \dots z_n F_>(z_0,\dots,z_{j-1},u,z_j,\dots,z_n,q) P_1 \left( \frac{z_n}{u},q \right) \\ 
\tr_M \sigma(a_0,z_0) X(a_1,z_1) \dots X(a_{i-1},z_{i-1}) X(a,u) X(a_i,z_i) \dots X(a_n,z_n)q^{L_0} \\ 
-
\frac{1}{ 2 \pi i } \sum_{i=2}^{n-1} \res_{z_0} \res_u z_1 \dots z_n F_>(z_0,\dots,z_{j-1},u,z_j,\dots,z_n,q) P_1 \left( \frac{z_n}{u},q \right) \\ 
\tr_M \sigma(a_0,z_0) X(a_1,z_1) \dots X(a_{i},z_{i}) X(a,u) X(a_{i+1},z_{i+1}) \dots X(a_n,z_n)q^{L_0}.
\label{eq:trace-1-2b}
\end{multlined}
\end{equation}
This is a telescoping sum and the only two terms that survive are
\begin{equation}
\begin{multlined}
\frac{1}{ 2 \pi i } \res_{z_0} \res_u z_1 \dots z_n F_>(z_0,\dots,z_{j-1},u,z_j,\dots,z_n,q) P_1 \left( \frac{z_n}{u},q \right) \tr_M \sigma(a_0,z_0) X(a,u) X(a_1,z_1)\dots X(a_n,z_n)q^{L_0} \\ 
-\frac{1}{ 2 \pi i } \res_{z_0} \res_u z_1 \dots z_n F_>(z_0,\dots,z_{j-1},u,z_j,\dots,z_n,q) P_1 \left( \frac{z_n}{u},q \right) \\ 
\tr_M \sigma(a_0,z_0) X(a_1,z_1) \dots X(a_{n-1},z_{n-1}) X(a,u) X(a_n,z_n)q^{L_0}.
\label{eq:trace-1-2-bisb}
\end{multlined}
\end{equation}
The sum of the second, third and fourth terms of \eqref{eq:actual-trace-propb} can be written as follows using Proposition \ref{prop:borcherds-modified}:
\begin{equation}
{\multlinegap=100pt
\begin{multlined}
 \frac{1}{2 \pi i} \res_{z_0} \res_u z_1 \dots z_n F_>(z_0,\dots,z_{j-1},u,z_j,\dots,z_n,q) P_1 \left( \frac{z_n}{u},q \right) \\ \tr_M \sigma(a_0,z_0) X(a_1,z_1) \dots X(a_{n-1},z_{n-1}) X(a,u) X(a_n,z_n) q^{L_0} \\ 
- \frac{1}{2 \pi i} \res_{z_0} \res_u z_1 \dots z_n F_>(z_0,\dots,z_{j-1},u,z_j,\dots,z_n,q) \left( P_1\left( \frac{z_nq}{u},q \right)  - 2\pi i  \right) \\ 
\tr_M \sigma(a_0,z_0)X(a_1,z_1)\dots X(a_n,z_n)X(a,u)  q^{L_0} .
\label{eq:trace1-2-ab}
\end{multlined}
}
\end{equation}
We see that the first four terms of \eqref{eq:actual-trace-propb} add up to 
\begin{equation}
\begin{multlined}
\frac{1}{ 2 \pi i } \res_{z_0} \res_u z_1 \dots z_n F_>(z_0,\dots,z_{j-1},u,z_j,\dots,z_n,q) P_1 \left( \frac{z_n}{u},q \right) \\  \tr_M \sigma(a,u) X(a_0,z_0) \dots X(a_n,z_n)q^{L_0} \\ 
+ \frac{1}{ 2 \pi i } \res_{z_0} \res_u z_1 \dots z_n F_>(z_0,\dots,z_{j-1},u,z_j,\dots,z_n,q) P_1 \left( \frac{z_n}{u},q \right) \\   \tr_M X(a,u) \sigma(a_0,z_0) X(a_1,z_1) \dots X(a_n,z_n)q^{L_0} \\ 
-\frac{1}{ 2 \pi i } \res_{z_0} \res_u z_1 \dots z_n F_>(z_0,\dots,z_{j-1},u,z_j,\dots,z_n,q) P_1 \left( \frac{z_n}{u},q \right) \\ \tr_M X(a_0,z_0) \sigma(a,u) X(a_1,z_1) \dots X(a_n,z_n) q^{L_0}  \\ 
- \frac{1}{2 \pi i} \sum_{i=1}^n \res_{z_0} \res_u z_1 \dots z_n F_>(z_0,\dots,z_{j-1},u,z_j,\dots,z_n,q) \left( P_1\left( \frac{z_nq}{u},q \right)  - 2\pi i  \right) \\ 
\tr_M \sigma(a_0,z_0)\dots X(a_n,z_n)X(a,u)  q^{L_0} .
\label{eq:first+second-traceb}
\end{multlined}
\end{equation}
In order to compute the last term of the left hand side of \eqref{eq:actual-trace-propb} we use Proposition \ref{prop:borcherds-modified-fourier} to express this term as a telescoping sum and the two surviving terms are 
\begin{equation}
\begin{multlined}
-\frac{1}{2\pi i} \res_{z_0} \res_u  z_1\dots z_n F_>(z_0,\dots,z_{j-1},u,z_j,\dots,z_n,q) P_1\left( \frac{z_n}{u} \right) \\ \tr_M \sigma(a,u)X(a_0,z_0) X(a_2,z_2) \dots X(a_n,z_n) q^{L_0} \\ 
+ \frac{1}{2\pi i} \res_{z_0} \res_u  z_1\dots z_n F_>(z_0,\dots,z_{j-1},u,z_j,\dots,z_n,q) P_1\left( \frac{z_n}{u} \right) \\  \tr_M \sigma(a,u) X(a_1,z_1) \dots X(a_{n},z_{n}) X(a_0,z_0) q^{L_0} .
\label{eq:third-term-traceb}
\end{multlined}
\end{equation}
The first term of \eqref{eq:third-term-traceb} cancels the first term in
\eqref{eq:first+second-traceb}. Commuting $X(a_0,z_0)$ with $q^{L_0}$ using
\eqref{eq:q0-xoperators} in the third term of \eqref{eq:first+second-traceb} we cancel the second term of \eqref{eq:third-term-traceb}. Finally commuting $X(a,u)$ with $q^{L_0}$ in the second term of \eqref{eq:first+second-traceb} we see that this term plus the fourth term add up to 
\[
\res_{z_0} \res_u z_1 \dots z_n F_>(z_0,\dots,z_{j-1},u,z_j,\dots,z_n,q) \tr_M \sigma(a_0,z_0) X(a_1,z_1) \dots X(a_n,z_n) X(a,u) q^{L_0}.
\]
Finally commuting $X(a,u)$ with $q^{L_0}$ in this last expression we obtain the right hand side of \eqref{eq:actual-trace-propb}, proving the proposition.
\end{proof}
Combining Proposition \ref{prop:8.14} with Proposition
\ref{prop:actual-trace-propb} we obtain a recurrent formula expressing $n+1$
point functions in terms of $n$-point functions. This is the analog of
 \cite[Prop. 4.3.4]{zhu} in the degree $0$ case. 
\begin{prop} With the notation of Proposition \ref{prop:8.14} we have
\begin{multline}
F^{n+1}_1\Bigl(a_0,\dots,a_{j-1},a,a_j,\dots,a_n,F(z_0,\dots,z_{j-1},w,z_j,\dots,z_n,q)
\Bigr) =  \\ 
 \sum_{i=0}^{n-1} \sum_k \sum_{m \geq 0} (-1)^{m} F^n_1 \left( a_0,\dots, a_{[k]}a_i,
\dots,a_n, F_{vi,k-m}(z_0,\dots,z_n,q) P_{m+1}\left( \frac{z_n}{z_i},q
\right) \right) \\ 
+ \sum_k F^n_1 \Bigl( a_0,\dots,a_{[k]}a_n, F_{vn,k+1}(z_0,\dots,z_n,q) \Bigr) \\ 
- \pi i \sum_k F^n_1 \Bigl( a_0,\dots,a_{[k]}a_n, F_{vn,k}(z_0,\dots,z_n,q) \Bigr) \\ 
- \sum_k \sum_{j \geq 1}  F^n_1 \Bigl( a_0,\dots,a_{[k]}a_n, G_{2j}(q) F_{vn,k+1-2j}(z_0,\dots,z_n,q) \Bigr) \\ 
+ \sum_{i=0}^{j-1} \sum_k \sum_{m \geq 0} F^n_1 \left( a_0,\dots,a_{[k]}a_i,\dots,a_n,
F_{vi,k-m}(z_0,\dots,z_n,q) \left( P_{m+1}\left( \frac{z_i q}{w},q \right) - 2
\pi i \delta_{m,0} \right) \right) \\
+ \sum_{i=j}^{n} \sum_k \sum_{m \geq 0} F^n_1 \left( a_0,\dots,a_{[k]}a_i,\dots,a_n,
F_{vi,k-m}(z_0,\dots,z_n,q) P_{m+1}\left( \frac{z_i}{w},q \right) \right) \\ 
-\sum_{i=1}^n \sum_k F^n_1 \left( a, a_1,\dots, a_{0[k]}a_i,\dots,a_n, F_{0ik,z_0,\dots,z_n}(z_1,\dots,z_0,z_j,\dots,z_n,q) \right.  \\ \left.  \left( P_1\left( \frac{z_0 q}{w},q \right) + P_1\left( \frac{z_n}{z_0},q \right) - 2 \pi i \right) \right).
\label{eq:8.15-prop1}
\end{multline}
\label{prop:8.15}
\end{prop}

\section{Convergence}\label{sec:series.exp}
In this section we identify conditions on the vertex algebra $V$ which guarantee
that the chiral homology $H_1^{\text{ch}}(V)$ has finite rank as an
$\mring$-module. This is the content of Proposition \ref{prop:H1.fin.gen}. Under
these finiteness hypotheses on $V$ we go on to prove that the chiral homology
$H_1^{\text{ch}}(E_\tau, V)$ forms a vector bundle with connection over the disc
$\{q \in \mathbb C \mid |q| < 1\}$ with regular singularity at $q=0$. We relate
series solutions of flat sections of (the dual of) chiral homology with the
Hochschild homology $\Hoch_1(\zhu(V))$ of the Zhu algebra and the Koszul
homology $\HK_1(A)$ of the associated graded $A = \gr V$ in
\ref{sec:Frob.descent} below, and deduce conditions on $V$ which guarantee
vanishing of $H_1^{\text{ch}}(E_\tau, V)$ for all $\tau \in \mathbb H$ in
Theorem \ref{thm:chiral.H1.vanishing}. Finally, in Theorem \ref{thm:convergence1}, we deduce convergence of first
trace functions under appropriate hypotheses on $V$. 
\begin{nolabel}
In this section we write $C_\bullet$ for the $1$-point chiral homology complex
$C_\bullet^{n=1}$ of Definition \ref{defn:chiral-homology-with-supports}.
According to \ref{no:one-coinv-does-nothing} we have
$H_i^{\text{ch}}(V)_{\text{tr}} \cong
H_i^{\text{ch}}(V) = H_i(C_\bullet^{n=1})$. We regard $C_\bullet$ as a complex of $\mring$-modules. As noted in \ref{nolabel:A.C.comparison} the specialization $\mathbb C \otimes_{\mring} H_i^{\text{ch}}(V)$ recovers the chiral homology $H_i^{\text{ch}}(E_\tau, V)$ of the elliptic curve $E_\tau$.

The relations $L_{-1} + \partial/\partial_{u} = 0$ in $C_2$ and $L_{-1} + \partial/\partial_{t_0} = 0$ in $C_1$ allow one to ``trade'' a partial derivative for a copy of $L_{-1}$, thus presenting $C_k$ as a quotient of
\[
V^{\otimes k+1} \otimes \cE_k, \quad \text{where} \quad \cE_k = \mathring{\mathbb{J}_*^{k}} / \left<\partial_{t_i} f(t_1,\ldots,t_k) \mid f \in \mathring{\mathbb{J}_*^{k}} \right>_{i=1,\ldots,k}.
\]
In order to proceed it is convenient to identify a finite set of generators of $\cE_k$. The rank of the $\mring$-module $\cE_k$ is equal to the $k^{\text{th}}$ Betti number of $(X \backslash 0)^{k} \backslash \Delta$, where $\Delta \subset X^k$ is the diagonal divisor. Explicit generators of $\cE_k$ were given in \cite{eh2018} and the differentials $d_1$ and $d_2$ written in terms of them \cite[Lemma 10.2]{eh2018}. We now recall these formulas for ease of reference. The space $\cE_0$ is generated by the class of the constant function $1$. The space $\cE_1$ is generated by the classes of the constant function $1$, and of the Weierstrass function $\wp(t_0)= \wp_2(t_0) + g_2(\tau)$. Finally $\cE_2$ is generated by five classes: those of $1$, $\wp(u-v)$, $\wp(u)\wp(v)$ as well as
\begin{align*}
\zuta(u, v) &= \zeta(u-v) + \zeta(v) - \zeta(u) \\
\text{and} \quad
\ZZZ(u, v) &= \wp(u-v) \zuta(u, v) + \frac{1}{2} \wp'(u-v).
\end{align*}
We now have:
\end{nolabel}
\begin{lem}[{\cite[Lemma 10.2]{eh2018}}]\label{lem:2.1.diff}
The following relations hold in $C_\bullet$:
\begin{align}
d_1\left(a \otimes a_1 \otimes 1 \right)
= {} & a_{(0)}a_1 \otimes 1, \label{eq:d1-part-0}
\\
d_1\left(a \otimes a_1 \otimes \wp(t_0) \right)
= {} & a_{(\wp)}a_1 \otimes 1. \label{eq:d1-part-wp}
\end{align}
and
\begin{align}
d_2\left(a\otimes b \otimes a_1 \otimes 1 \right)
= {} & \left( a \otimes b_{(0)}a_1 - a_{(0)}b \otimes a_1 -b \otimes a_{(0)}a_1 \right) \otimes 1, \label{eq:d2-i.1} \\
d_2\left( a \otimes b \otimes a_1 \otimes \wp(u-v) \right)
= {} & - a_{(\wp)}b \otimes a_1 \otimes 1  
 \label{eq:d2-i.2}\\
&\hspace{-1.9cm}+ \sum_{j \in \Z_+} \frac{1}{j!} \left( L_{-1}^j a \otimes b{(j)}a_1- L_{-1}^j b \otimes a{(j)}a_1 \right) \otimes \wp(t_0), \nonumber
\\
d_2\left(a \otimes b \otimes a_1 \otimes \wp(u) \right)
= {} & -b \otimes a_{(\wp)}a_1 \otimes \wp(t_0) 
 \label{eq:d2-i.3}\\
&+ \left( a \otimes b{_{(0)}}a_1 + b_{(0)}a \otimes a_1 \right) \otimes \wp(t_0), \nonumber
\\
d_2\left(- a \otimes b \otimes a_1 \otimes \zuta(u, v) \right)
= {} & \left( a_{(\zeta)}b \otimes a_1 - b \otimes a_{(\zeta)}a_1 - a \otimes b_{(\zeta)}a_1 \right) \otimes 1  
 \label{eq:d2-i.4}\\
&\hspace{-1.9cm}+ \sum_{j \in \Z_+} \frac{(-1)^j}{(j+1)!} L_{-1}^{j}\left( a_{(j+1)}b \right) \otimes a_1 \otimes \wp(t_0) \nonumber
\\
&\hspace{-1.9cm}- \sum_{j \in \Z_+} \frac{1}{(j+1)!} \left(L_{-1}^{j}a \otimes b(j+1)a_1 + L_{-1}^{j}b \otimes a(j+1)a_1 \right) \otimes \wp(t_0), \nonumber
\\
d_2\left(a \otimes b \otimes a_1 \otimes \ZZZ(u, v) \right)
= {} & -\sum_{j \in \Z_+} \frac{(-1)^j}{(j+1)!} L_{-1}^{j}\left( a_{(x^{j+1} \wp(x))}b \right) \otimes a_1 \otimes \wp(t_0) 
 \label{eq:d2-i.5}\\
&\hspace{-1.9cm}+\sum_{j \in \Z_+} \frac{1}{(j+1)!} \left(L_{-1}^{j}a \otimes b_{(x^{j+1}\wp(x))}a_1 + _{-1}^{j}b \otimes a_{(x^{j+1}\wp(x))}a_1 \right) \otimes \wp(t_0) \nonumber
\\
&\hspace{-1.9cm}+(2\pi i)^2 q \frac{d}{dq} \left( a_{(\zeta)}b \otimes a_1 - a \otimes b_{(\zeta)}a_1 - b \otimes a_{(\zeta)}a_1 \right) \otimes 1, \nonumber
\\
d_2\left( a  \otimes b \otimes a_1 \otimes \wp(u) \wp(v) \right)
= {} & \left( a \otimes b_{(\wp)}a_1 - b \otimes a_{(\wp)}a_1 \right) \otimes \wp(t_1)  \label{eq:d2-i.6} \\
&- \sum_{j \in \Z_+} \frac{(-1)^j}{(j+1)!} L_{-1}^{j}\left( a_{(x^{j+1} \wp'(x))}b \right) \otimes a_1 \otimes \wp(t_0) \nonumber
\\
&- \left(  a_{(\wp)}b \otimes a_1 - b_{(\wp)}a \otimes a_1 \right) \otimes \wp(t_0) 
-(2 \pi i)^{2} q \frac{d}{dq} a_{(\wp)}b \otimes a_1 \otimes 1. \nonumber
\end{align}
\end{lem}

\begin{nolabel}
We now recall the notion of standard filtration \cite{li05} on a conformal vertex algebra, and use it to define an increasing filtration on the complex $C_\bullet$. Let $\{a^i | i \in I\}$ be a strong generating set of the vertex algebra $V$. We suppose that the strong generators are of homogeneous conformal weight $\D(a^i)$. The standard filtration on $V$, relative to this choice, is the increasing filtration $\{G_p V\}$ in which $G_p V$ is the span of the vectors
\begin{align*}
a^{i_1}_{(-n_1)}\cdots a^{i_r}_{(-n_r)}\vac
\end{align*}
where $r \geq 0$,  $a^1, \ldots, a^r \in V$ and $n_1, \ldots, n_r \in \Z_{> 0}$ satisfying $\sum_j \D(a^{i_j}) \leq p$. The following properties are easily verified: If $a \in G_pV$ and $b \in G_{p'}V$ then
\begin{align*}
a_{(n)}b &\in G_{p+p'}V \qquad \text{for all $n \in \Z$}, \\
a_{(n)}b &\in G_{p+p'-1}V \qquad \text{if $n \geq 0$}, \\
L_{-1}a &\in G_p V.
\end{align*}
For $p \in \Z_+$ we define $G_p C_{k}$ to be the image in $C_{k}$ of the subspace of $ V^{\otimes k+1} \otimes \mathring{\mathbb{J}}^k_*$ spanned by
\begin{align*}
\bigcup_{\sum p_i \leq p} G_{p_1}V \otimes \cdots \otimes G_{p_{k+1}}V \otimes \mathring{\mathbb{J}}^k_*.
\end{align*}
This defines an increasing filtration of $C_\bullet$ by $\mring$-modules, using here that the filtration $G$ itself is invariant under $L_{-1}$.
\end{nolabel}

\begin{nolabel}
The associated graded $A = \gr_G V$ carries a commutative algebra structure with product induced by the normally ordered product $a \otimes b \mapsto a_{(-1)}b$, indeed a Poisson vertex algebra structure. We now have
\[
\gr_G C_k \cong A^{\otimes k+1} \otimes \mathring{\mathbb{J}}^k_*,
\]
with differentials in the associated graded given by
\begin{align}
d_1\left( a \otimes a_1 \otimes 1 \right) &= 0, \label{eq:G.filtered.d1.1} \\
d_1\left( a \otimes a_1 \otimes \wp(t_0) \right) &= a_{(-2)}a_1 \otimes 1, \label{eq:G.filtered.d1.wp}
\end{align}
and (see Lemma \ref{lem:2.1.diff})
\begin{align}
d_2\left(a \otimes b \otimes a_1 \otimes 1 \right) 
= {} & 0,\label{eq:G.filtered.d2.1}
\\
d_2\left(a \otimes b \otimes a_1 \otimes \wp(u-v) \right)
= {} & - a_{(-2)}b \otimes a_1 \otimes 1\label{eq:G.filtered.d2.2} \\
d_2\left(a \otimes b \otimes a_1 \otimes \wp(u) \right)
= {} & -b \otimes a_{(-2)}a_1 \otimes 1\label{eq:G.filtered.d2.3} \\
d_2\left(-a \otimes b \otimes a_1 \otimes \zuta(u, v) \right) 
= {} & \left( a_{(-1)}b \otimes a_1 - b \otimes a_{(-1)}a_1 - a \otimes b_{(-1)}a_1 \right) \otimes 1 \label{eq:G.filtered.d2.4}\\
d_2\left( a \otimes b \otimes a_1 \otimes \ZZZ(u, v) \right) 
= {} & -\left( a_{(-1)}b \otimes a_1 - a \otimes b_{(-1)}a_1 - b \otimes a_{(-1)}a_1 \right) \otimes \wp(t_0) \label{eq:G.filtered.d2.5}\\
d_2\left(a \otimes b \otimes a_1 \otimes \wp(u)\wp(v) \right)
= {} & \left( a \otimes b_{(-2)}a_1 - b \otimes a_{(-2)}a_1 \right) \otimes \wp(t_0).\label{eq:G.filtered.d2.6}
\end{align}
To simplify the subsequent discussion we introduce a new complex $\widetilde{C}_\bullet = \widetilde{C}_\bullet^{(1)} \oplus \widetilde{C}_\bullet^{(2)}$, defined as follows:
\begin{align*}
\widetilde{C}_0^{(1)} &= 0 \\
\widetilde{C}_1^{(1)} &= A^{\otimes 2} \otimes \mathbb{C} [1] \\
\widetilde{C}_2^{(1)} &= A^{\otimes 3} \otimes \left( \mathbb{C} [1] \oplus
\mathbb{C} [\wp(u)] \oplus \mathbb{C}[\wp(u-v)] \oplus \mathbb{C}[\zuta(u,v)]
\right),
\end{align*}
with differentials given by the formulas \eqref{eq:G.filtered.d1.1} and \eqref{eq:G.filtered.d2.1}-\eqref{eq:G.filtered.d2.4}, and 
\begin{align*}
\widetilde{C}_0^{(2)} &= A \otimes \mathbb{C} [1] \\
\widetilde{C}_1^{(2)} &= A^{\otimes 2} \otimes \mathbb{C} [\wp(u)] \\
\widetilde{C}_2^{(2)} &= A^{\otimes 3} \otimes \left( \mathbb{C}[\ZZZ(u, v)]
\oplus \mathbb{C}[\wp(u)\wp(v)] \right),
\end{align*}
with differentials given by the formulas \eqref{eq:G.filtered.d1.wp}, \eqref{eq:G.filtered.d2.5} and \eqref{eq:G.filtered.d2.6}. In these formulas $\mathbb{C} [\wp(u)]$ is to be understood as the $\mathbb{C}$-vector space with basis consisting of the formal symbol $[\wp(u)]$, etc. We emphasize that $\widetilde{C}_\bullet$ is defined over $\mathbb{C}$ rather than $\mring$.
\begin{lem}\label{lem:assoc.gr.complex}
The differentials defined above make $\widetilde{C}_\bullet$ into a complex. There exists a surjection of complexes $\widetilde{C}_\bullet \otimes_{\mathbb{C}} \mring \rightarrow \gr_G C_\bullet$ which induces a surjection of homology groups in degrees $0$ and $1$.
\end{lem}

\begin{proof}
The first claim is essentially a direct consequence of the construction of $\widetilde{C}_\bullet$. However, since it will not take much space, we check that $d^2=0$ in $\widetilde{C}_\bullet$ directly. For $\widetilde{C}_\bullet^{(1)}$ there is nothing to check. For $\widetilde{C}_\bullet^{(2)}$ we need to check that
\begin{align*}
(a_{(-1)}b)_{(-2)}c - a_{(-2)}b_{(-1)}c - b_{(-2)}a_{(-1)}c
\qquad \text{and} \qquad 
a_{(-2)}b_{(-2)}c - b_{(-2)}a_{(-2)}c,
\end{align*}
vanish in $A$. For this we use Borcherds identity \eqref{eq:borcherds-def} which, in terms of modes, asserts that
\begin{align*}
\sum_{j \in \Z_+} \binom{m}{j} (a_{(n+j)}b)_{(m+k-j)}a_1 = \sum_{j \in \Z_+}
(-1)^j \binom{n}{j} \left(a_{(m+n-j)} b_{(k+j)} - (-1)^n b_{(n+k-j)} a_{(m+j)}
\right) a_1,
\end{align*}
for all $a, b, a_1 \in V$ and $m, k, n \in \Z$. Putting $k = m = n = -1$ yields
\begin{align*}
\sum_{j \in \Z_+} (-1)^j (a_{(-1+j)}b)_{(-2-j)} a_1 = \sum_{j \in \Z_+} \left(
a_{(-2-j)} b_{(-1+j)} a_1 + b_{(-2-j)} a_{(-1+j)} a_1 \right) =
a_{(-2)}b_{(-1)}a_1 + b_{(-2)}a_{(-1)}a_1.
\end{align*}
In the associated graded we discard terms $a_{(j)}b$ for $j \in \Z_+$, and the identity becomes
\begin{align*}
(a_{(-1)}b)_{(-2)}a_1 = a_{(-2)}b_{(-1)}a_1 + b_{(-2)}a_{(-1)}a_1,
\end{align*}
proving the first of the two relations. For the other we use the commutator formula
\begin{align*}
[a_{(-2)}, b_{(-2)}]a_1 = \sum_{j \in \Z_+} \binom{-2}{j}
(a_{(j)}b)_{(-4-j)}a_1,
\end{align*}
which, again, vanishes in the associated graded.

The surjection $\widetilde{C}_\bullet \otimes_{\mathbb{C}} \mring \rightarrow \gr_G C_\bullet$ is the obvious one. Since $C_0 = V \otimes \mathbb{J}_*^0 \cong V \otimes \mring$ the kernel vanishes in degree $0$. Passing to the associated long exact sequence, it follows that the induced maps in homology are surjections as claimed.
\end{proof}

We now consider the commutative algebra $A = \gr_G V$ with derivation induced by $L_{-1}$. Since in general $Y(L_{-1}a, z) = \partial_z Y(a, z)$ we have $a_{(-2)}b = (L_{-1}a) \cdot b$. First we note that $\widetilde{C}^{(2)}_1$ modulo the image of \eqref{eq:G.filtered.d2.5} coincides with $\Hoch_1(A) \cong \Omega_{A/\mathbb{C}}$. We now verify that the differential \eqref{eq:G.filtered.d2.6} coincides with the differential in the Koszul complex \eqref{eq:d2.HP2} and conclude that $H_{1}(\widetilde{C}^{(2)}_\bullet) \cong \HK_{1}(A)$. We remark that the Koszul homology of $A = \gr V$ appeared in a similar way in \cite{eh2018} upon passage to an associated graded.

We now turn to $H_1(\widetilde{C}_\bullet^{(1)})$. We recall Zhu's $C_2$-algebra
$R_V$, briefly described in Section \ref{no:C2.def}. The quotient of
$\widetilde{C}^{(1)}_1 = A \otimes A$ by the images of
\eqref{eq:G.filtered.d2.2} and \eqref{eq:G.filtered.d2.3} is just $R_V \otimes
R_V$. The quotient by the image of the differential \eqref{eq:G.filtered.d2.4}
is precisely the module of K\"{a}hler differentials $\Hoch_1(R_V) \cong \Omega_{R_V/\mathbb{C}}$.

It follows from Lemma \ref{lem:assoc.gr.complex} and the remarks above that $H_1(\gr_G C_\bullet)$ is a quotient of
\begin{align*}
\mring \otimes_{\mathbb{C}} \left( \HK_{1}(A) \oplus \Hoch_1(R_V) \right).
\end{align*}
If we assume finite dimensionality of the term in parentheses then we will obtain a finiteness condition on chiral homology. We now prove a strengthening of this result.
\end{nolabel}

\begin{prop}\label{prop:H1.fin.gen}
Let $V$ be a strongly finitely generated conformal vertex algebra, $A = \gr V$ its associated graded and $R_V$ its $C_2$-algebra. Suppose that $\dim \HK_{1}(A) < \infty$. If $\dim \Omega_{R_V} < \infty$, or more generally if $\dim \HP_1(R_V) < \infty$, then the chiral homology group $H_{1}(C_\bullet)$ is finitely generated as a $\mring$-module.
\end{prop}

\begin{proof}
The homology groups of $C_\bullet$ are computed by the spectral sequence associated with the filtration $G$. More precisely $H_\bullet(C_\bullet)$ acquires a filtration, which we also call $G$, defined by $G_p H_n(C_\bullet) = \text{Im}(H_n(G_p C_\bullet) \rightarrow H_n(C_\bullet))$, and there exist isomorphisms
\begin{align}\label{eq:SS.isom}
\gr_G H_{n}(C_\bullet) \cong \bigoplus_{p-q=n} E^\infty_{p, q}.
\end{align}
Since $G$ is compatible with the $\mring$-module structure of $C_\bullet$, all
morphisms in the spectral sequence and in particular the isomorphism
\eqref{eq:SS.isom} are morphisms of $\mring$-modules. In particular $\gr_G
H_{1}(C_\bullet)$ is a subquotient of $H_{1}(\gr_G C_\bullet)$ and hence of
\begin{align*}
\mring \otimes_{\mathbb{C}} \left( \HK_{1}(A) \oplus \Hoch_1(R_V) \right).
\end{align*}
Suppose $\dim \Hoch_1(R_V) < \infty$. Since $\mring$ is a Noetherian ring, it follows that $\gr_G H_{1}(C_\bullet)$ is a Noetherian, and hence finitely generated, $\mring$-module. In general let $R$ be a ring and $M$ an $R$-module with increasing filtration $F$ exhaustive and bounded below. If $\gr_F M$ is finitely generated as an $R$-module then so is $M$, as can be proved by a straightforward induction. We conclude that $H_{1}(C_\bullet)$ is a finitely generated $\mring$-module.

Now we suppose $\dim \HP_1(R_V) < \infty$ (dropping the hypothesis $\dim \Hoch_1(R_V) < \infty$). We will show that $\bigoplus_{p-q=1} E^2_{p, q}$ is finitely generated as a $\mring$-module.

The total complex of the $E^1$ page is $H_\bullet(\gr_G C_\bullet)$ with differential which we denote $d^{(1)}$. By Lemma \ref{lem:assoc.gr.complex} the groups $H_\bullet(\gr_G C_\bullet)$ receive surjections from $\mring \otimes H_\bullet(\widetilde{C}_\bullet)$ in degrees $0$ and $1$. The following diagram presents $\mring \otimes H_\bullet(\widetilde{C}_\bullet)$:
\begin{align}\label{diag.page.E1}
\xymatrix{
 & \mring \otimes \Om_{R_V} \ar@{->}[dl]_{\partial_1^{\text{P}}} & \mring \otimes H_{2}(\widetilde{C}_\bullet^{(1)}) \ar@{->}[dl] \ar@{->}[l] \\
\mring \otimes R_V & \mring \otimes \HK_{1}(A) \ar@{->}[l]^{\partial_1} & \mring \otimes H_{2}(\widetilde{C}_\bullet^{(2)}) \ar@{->}[ul] \ar@{->}[l], \\
}
\end{align}
the top row being $H_{\bullet}(\widetilde{C}_\bullet^{(1)})$ and the bottom row $H_{\bullet}(\widetilde{C}_\bullet^{(2)})$. The arrows here denote the components of $d^{(1)}$. For example $\partial_1^{\text{P}}$ is induced by \eqref{eq:d1-part-0} and thus coincides with \eqref{eq:d1.HP2} for the Poisson algebra $R_V$, while $\partial_1$ is induced by \eqref{eq:d1-part-wp} the $\wp$-product. The kernel of $d^{(1)}_1$ fits into an exact sequence
\begin{align}\label{eq:ker.d1}
0 \rightarrow \ker(\partial_1^{\text{P}}) \rightarrow \ker(d^{(1)}_1) \rightarrow \operatorname{im}(\partial_1^\text{P}) \cap \operatorname{im}(\partial_1) \rightarrow 0,
\end{align}
the arrows sending $x$ to $(x, 0)$ and $(x, y)$ to $\partial_1(y)$, respectively.

By equation \eqref{eq:G.filtered.d2.1} terms of the form $a \otimes b \otimes a_1 \otimes 1$ are closed in $\widetilde{C}_{2}^{(1)}$, and so survive in $H_{2}(\widetilde{C}_\bullet^{(1)})$. The differential of such terms is given by \eqref{eq:d2-i.1} which we see lies entirely in $\mring \otimes \Omega_{R_V}$ and is given by
\[
a \otimes b_{(0)}a_1 - a_{(0)}b \otimes a_1 - b \otimes a_{(0)}a_1.
\]
The image of these terms coincides with the image $\operatorname{im}(\partial_2^{\text{P}})$ of the differential in the Poisson homology complex of $R_V$ as in \eqref{eq:d2.HP}. Therefore we have a surjection
\begin{align*}
\frac{\ker(d^{(1)}_1)}{\operatorname{im}(\partial_2^{\text{P}})} \twoheadrightarrow \frac{\ker(d^{(1)}_1)}{\operatorname{im}(d^{(1)}_2)} \cong \bigoplus_{p-q=1} E^2_{p, q}.
\end{align*}
The exact sequence \eqref{eq:ker.d1} induces the following exact sequence
\begin{align}\label{eq:coho.d1}
0 \rightarrow \HP_1(R_V) \rightarrow \frac{\ker(d^{(1)}_1)}{\operatorname{im}(\partial_2^{\text{P}})} \rightarrow \operatorname{im}(\partial_1^\text{P}) \cap \operatorname{im}(\partial_1) \rightarrow 0.
\end{align}
The rank of the middle term is thus bounded above by $\dim \HK_{1}(A) + \dim \HP_1(R_V)$. Once again it follows that $H_{1}(C_\bullet)$ is finitely generated, as required.
\end{proof}

Proposition \ref{prop:H1.fin.gen} above is in the same spirit as the following
finiteness result on $H_0^{\text{ch}}(V)$. Under the hypothesis $\dim {R_V} <
\infty$ ($C_2$-cofiniteness) Proposition \ref{prop:H0.fin.gen} was proved by Zhu
\cite{zhu} by an induction on conformal weight in $C_0^{n=1} = V$. It was
observed by Arakawa and Kawasetsu \cite{arakawa-kawasetsu} that Zhu's proof
requires only the weaker hypothesis of $\dim \HP_0(R_V) < \infty$. 
\begin{prop}\label{prop:H0.fin.gen}
Let $V$ be a conformal vertex algebra. If $\dim {R_V} < \infty$, or more generally if $\dim \HP_0(R_V) < \infty$, then $H_0^{\text{ch}}(V)$ is finitely generated as a $\mring$-module.
\end{prop}

Damiolini et. al. \cite{damiolini2020conformal} extended Zhu's results to
to prove finite dimensionality of $H^{\ch}_0(X, \cA_V)$ for an arbitrary smooth
curve $X$, where $\cA_V$ is the chiral algebra associated to the $C_2$--cofinite
vertex algebra $V$. It would be interesting to know if finite dimensionality of
$\HK_\bullet(A)$ and $\HP_\bullet(R_V)$ sufice to guarantee finite dimensionality
of higher chiral homology in arbitrary genera.

\begin{lem} \label{lem:finite-dim}
Let $V$ be a vertex algebra for which $H_i^{\text{ch}}(V)$, for $i=0$ or $i=1$, is finitely generated as a $\mring$-module, let $\tau \in \mathbb H$ and let $E_\tau = \mathbb C / \Z + \Z \tau$. Then
\[
\dim_{\mathbb C} H_i^{\text{ch}}(E_\tau, V) < \infty.
\]
\end{lem}

Indeed the dimension is bounded above by the rank of $H_i^{\text{ch}}(V)$ as a $\mring$-module.

\begin{proof}
The chiral homology of $E_\tau$ is obtained specializing $q$ to $e^{2\pi i \tau} \in \mathbb C$, i.e., as the tensor product
\[
H_i^{\text{ch}}(E_\tau, V) \cong \mathbb C \otimes_{\mring} \HchiralV,
\]
with $\mring$-action on $\mathbb C$ given by the evaluation map sending $G_k \in \mathbb{QM_*}$ to $G_k(e^{2\pi i\tau}) \in \mathbb{C}$ for all $k$. It is immediate that for any fixed $\tau \in \mathbb H$ the vector space $H_i^{\text{ch}}(E_\tau, V)$ is finite dimensional over $\mathbb C$, with dimension bounded above by the rank of $H_i^{\text{ch}}(V)$ as a $\mring$-module.
\end{proof}

\begin{thm} \label{thm:convergence2}
Let $V$ be a finitely strongly generated conformal vertex algebra, $A = \gr_F V$ its associated graded and $R_V = A^0$. Let $E_\tau = \mathbb C / \Z + \Z \tau$ be a smooth elliptic curve. If $V$ satisfies (1) $\dim \HP_1(R_V) < \infty$ and (2) that
the kernel of the canonical surjection $J R_V \twoheadrightarrow A$ is finitely generated as a differential ideal, then 
\[
\dim H^{\text{ch}}_1(E_\tau, V) < \infty.
\]
\end{thm}

\begin{proof}
By Proposition \ref{prop:Koszul.finiteness} the second of our hypotheses implies that $\dim \HK_1(A) < \infty$. Therefore the hypotheses of Proposition \ref{prop:H1.fin.gen} are satisfied and hence, by Lemma \ref{lem:finite-dim}, the conclusion follows.
\end{proof}

\begin{nolabel}
In fact the dimension of $H_i^{\text{ch}}(E_\tau, V)$ is uniformly bounded above
by the rank of $H_i^{\text{ch}}(V)$ as a $\mring$-module. Below, using
techniques from the theory of differential equations and $\mathcal D$-modules,
we will obtain stronger results on the dimension of chiral homology. The chiral
complexes $\qmring \otimes_{\mathbb{M}_*} C_\bullet^{n}$ come equipped with
connections defined by equations \eqref{eq:connection_c0} and
\eqref{eq:connection_2}. Inspecting these formulas we see that the connections
are regular with regular singular point $q=0$. Indeed, upon restricting to the
case of $n=1$ and passing to homology $H_i^{\text{ch}}(V)_{\text{tr}} \cong H_i^{\text{ch}}(V)$, these connections take the form
\begin{align*}
\nabla_\tau = (2\pi i)^2 q \frac{\partial}{\partial q} - H^{(i)},
\end{align*}
with $H^{(i)}$ as in equations \eqref{eq:H0.op.def} and \eqref{eq:H1.op.def} manifestly regular in $q$. It follows that $H_i^{\text{ch}}(V)$ are regular $\CD$-modules with regular singular point $q=0$. This $\CD$-module structure was exploited by Zhu in the case $i=0$ considered in \cite{zhu} to prove modular invariance of characters of $V$-modules for $V$ a rational $C_2$-cofinite conformal vertex algebra.
\end{nolabel}

\begin{lem}\label{lem:Hi.vector.bundle}
Let $i=0$ or $i=1$ and let $V$ be a vertex algebra satisfying the hypotheses of Proposition \ref{prop:H0.fin.gen} or, respectively, Proposition \ref{prop:H1.fin.gen}. Then the chiral homology $H_i^{\text{ch}}(E_\tau, V)$, for $\tau \in \mathbb H$ the upper half complex plane, defines a vector bundle with a flat connection.
\end{lem}

\begin{proof}
The $\qmring$-module $H_i^{\text{ch}}(V)$ defines a quasicoherent sheaf over
$\mathbb H$, after analytification. By {\cite[Proposition 8.8]{katz70}} an
$\cO$-quasicoherent $\CD$-module is locally free. This sheaf is coherent by
Lemma \ref{lem:finite-dim}, hence a vector bundle with flat connection.
\end{proof}

\begin{nolabel}
We now study flat sections of the dual of chiral homology (Definition \ref{defn:chiral-cohomology-flat-sections}) in the case $n=1$. These are elements of
\[
H_i^{\text{ch}}(V)^* = \Hom_{\qmring}(\qmring \otimes_{\mring} H_i^{\text{ch}}(V), \cF),
\]
where $\cF  = \cF_0 \cong \cF_1$ is the ring of holomorphic functions on $\mathbb H$, which satisfy the differential equation \eqref{eq:flatness-0}. We work in the analytic topology. Before proceeding we recall a classical result from the theory of ordinary differential equations (ODE).
\end{nolabel}

\begin{thm}[{\cite[Chapter III, Theorem 1.3.1]{Haefliger-D-mod}}]\label{thm:Fuchs-theorem}
Suppose $A(q)$ is an $N \times N$ matrix whose entries are analytic functions in the domain $\{ q \in \mathbb C \mid 0 < |q| < \epsilon\}$ and has a simple singularity at $q=0$. Then a fundamental system of solutions of the ODE
\begin{align}\label{eq:ODE.reg.sing.type}
w'(q) = \frac{A(q)}{q} w(q),
\end{align}
is given by
\begin{align}\label{eq:Frob.exp}
W(q) = U(q)\exp\left(\log(q)M\right),
\end{align}
for some constant matrix $M$, and $U(q)$ an $N \times N$ matrix of functions analytic in $\{ q \in \mathbb C \mid |q| < \epsilon\}$.
\end{thm}

\begin{nolabel}
An equivalent formulation is obtained upon expanding \eqref{eq:Frob.exp} relative to a basis of eigenvectors, namely: any solution $w(q)$ of the ODE \eqref{eq:ODE.reg.sing.type} in a connected simply connected domain in $\{q : 0 < |q| < \epsilon\}$ is given by a convergent series of the form
\begin{align}\label{eq:Frob.exp.2}
w_i(q) = \sum_{p=0}^N \log(q)^p \sum_{j} q^{\lambda_{j}} \sum_{n=0}^\infty U^{(i)}_{pjn} q^n.
\end{align}
In general we refer to such an expression as a \emph{Frobenius expansion}. In the present case the complex numbers $\lambda_{j}$ are essentially the eigenvalues of the matrix $M$. It follows, moreover, that a formal series solution $w(q)$ of \eqref{eq:ODE.reg.sing.type} of the form \eqref{eq:Frob.exp.2}, is convergent.
\end{nolabel}

\begin{nolabel}\label{sec:Frob.descent}
We now suppose, for either $i=0$ or $i=1$, that $H_i^{\text{ch}}(V)$ is finitely generated as a $\mring$-module, and we let $\{\alpha_1, \ldots, \alpha_N \}$ be a set of generators of the finitely generated $\qmring$-module $\qmring \otimes_{\mring} H_i^{\text{ch}}(V)$. We denote by $A(q)$ a matrix (with entries in $\qmring \subset \mathbb C[[q]]$) of the operator $\tfrac{1}{(2\pi i)^2} \nabla_\tau$ relative to the generators. We remark that $A(q)$ is not uniquely prescribed in this way since $\{\alpha_1, \ldots, \alpha_N \}$ is not necessarily a basis.

Let $S \in H^i_{\text{ch}}(V)^{*\nabla}$ be a flat section of the dual of chiral homology and, for $k = 1, \ldots, N$, write $S_k(\tau) = S(\alpha_k, \tau)$. Here, as always, $q = e^{2\pi i \tau}$. We denote by $w(\tau)$ the column vector with entries given by the functions $S_k(\tau)$. Then the differential equation \eqref{eq:flatness-0} implies that $w(\tau)$ is a solution of the ODE
\begin{align*}
\frac{d w}{dq} = \frac{A(q)}{q} w.
\end{align*}
The entries of $A(q)$ are formal power series with radius of convergence $\epsilon = 1$, we may thus apply Theorem \ref{thm:Fuchs-theorem} to deduce that the functions $S_k(\tau)$ possess convergent Frobenius expansions of the form \eqref{eq:Frob.exp.2}.

For any $\alpha \in H_i^{\text{ch}}(V)$ we write $\alpha = \sum_{k=0}^N f_k(q) \alpha_k$ with coefficients $f_k(q) \in \mring$, so that $S(\alpha, \tau) = \sum_{k=0}^N f_k(q) w_k(\tau)$, and
\begin{align}\label{eq:Frob.exp.S}
S(\alpha, \tau) = \sum_{p=0}^N \log(q)^p \sum_{j} q^{\lambda_{j}} \sum_{n=0}^\infty U_{pjn}(\alpha) q^n,
\end{align}
where now $U_{pjn} : H_i^{\text{ch}}(V) \rightarrow \mathbb C$ are $\mathbb C$-linear maps. We shall call a coefficient $U_{pj0}$ a \emph{leading coefficient} if $\lambda_j$ is minimal in the sense that $\lambda_j-1$ does not appear in the middle sum in \eqref{eq:Frob.exp.S}.

Let $U$ be a leading coefficient of $S$ and let $\alpha \in H_i^{\text{ch}}(V)$ be a class which vanishes at $q=0$, i.e., which vanishes in $\mathbb C \otimes_{\qmring} H_i^{\text{ch}}(V)$ where the $\qmring$-action on $\mathbb C$ is the evaluation map sending $G_k \in \qmring$ to $G_k(0) \in \mathbb{C}$ for all $k$. Since $S$ is $\qmring$-linear it is clear that $U(\alpha) = 0$. Thus $U$ descends to a $\mathbb C$-linear map
\[
U : H_i^{\text{ch}}(V, q=0) \rightarrow \mathbb C.
\]
\end{nolabel}

\begin{nolabel}
It is instructive to recall Zhu's analysis of the connection on chiral homology in degree $0$. We first recall the definition and basic properties of the Zhu algebra $\zhu(V)$ \cite{zhu}. Let 
\begin{align}\label{eq:p.zeta.limits}
\begin{split}
f(z) &= 2\pi i \frac{e^{2\pi i z}}{e^{2\pi i z}-1} = \zeta(z, q=0) + \pi i, \\
g(z) &= (2\pi i)^2 \frac{e^{2\pi i z}}{(e^{2\pi i z}-1)^2} = \wp(z, q=0),
\end{split}
\end{align}
and define, for all $a, b \in V$, the bilinear products
\begin{align*}
a \circ b &= \res_w w^{-2} (1+w)^{\D(a)} a(w)b \, dw = a_{[g]}b, \\
\quad \text{and} \quad
a * b &= \res_w w^{-1}(1+w)^{\D(a)} a(w)b \, dw = a_{[f]}b.
\end{align*}
Then $\zhu(V)$ is defined to be the quotient vector space $V / V \circ V$
equipped with the associative product $a \otimes b \mapsto a * b$. This is an
associative algebra with unit $\vac$. The following formula \cite[equation
(3.15)]{eh2019} for the commutator in $\zhu(V)$ will be useful\footnote{We
remark that this equation appears in \cite[p. 296]{zhu} with a typo.}
\begin{align}\label{eq:Zhucomm}
a * b - b * a = 2\pi i \, a_{[0]}b \pmod{V_{[g]}V},
\end{align}
where in the right hand side we have used the vertex algebra structure $(V, Y[\cdot,z])$ of
\ref{no:zhu-alter-vertex}.  
Let $M = \bigoplus_{n=0}^\infty M_{n}$ be a positive energy $V$-module graded by generalized eigenspaces $M_{n}$ for $L_0$ of eigenvalue $\lambda+n$ for some constant $\lambda \in \mathbb C$. Then $M_0$ acquires the structure of a $\zhu(V)$-module via the assignment $[a] m = a_0 m$ for all $a \in V$ and $m \in M_0$.
\end{nolabel}

\begin{nolabel}
Now let $S \in H_0^{\text{ch}}(V)^{*\nabla}$ be a flat section of the dual of
zeroth chiral homology of the vertex algebra $(V, Y[\cdot,z])$. Concretely this amounts to a map
\begin{align*}
S : \qmring \otimes V \rightarrow \cF_0,
\end{align*}
which annihilates elements of the form
\begin{align}\label{eq:one.point.vanishes.on}
\begin{split}
a_{[0]}a_1 \quad
\text{and} \quad a_{[\wp]}a_1 = a_{[-2]}a_1 + \sum_{k \geq 2} (2k-1) G_{2k}(q) a_{[2k-2]}a_1,
\end{split}
\end{align}
for all $a, a_1 \in V$, and which satisfies the ODE
\begin{align*}
(2\pi i)^2 q\frac{d}{dq} S(a, \tau) = S(L_{[\zeta]}a, \tau),
\end{align*}
for all $a \in V$.
Let $U$ be a leading coefficient of $S$. As shown above $U$ descends to an element of the $\mathbb C$-linear dual of
\[
H_0^{\text{ch}}(V, q=0) \cong \Hoch_0(\zhu(V)).
\]
More concretely $U$ annihilates the restrictions of elements \eqref{eq:one.point.vanishes.on} at $q=0$, namely
\begin{align}\label{eq:one.point.vanishes.on.restricted}
a_{[0]}b  \quad
\text{and} \quad a_{[g]}b \quad \text{where} \quad g(t) = \wp(t, q=0) = (2\pi i)^2 \frac{e^{2\pi i t}}{(e^{2\pi i t}-1)^2}.
\end{align}
It follows that $U$ descends to a linear function on $\zhu(V) / [\zhu(V), \zhu(V)] = \Hoch_0(\zhu(V))$.
\end{nolabel}

\begin{nolabel}
We now recall some results from \cite{eh2018} on the structure of $H_1^{\text{ch}}(V, q=0)$ and its relation to the Hochschild homology $\Hoch_1(\zhu(V))$. In loc. cit. we identified a subcomplex $B_\bullet$ of the chiral complex $C_\bullet^{n=1}(q=0)$ such that the quotient complex computes Hochschild homology of $\zhu(V)$ in degrees $0$ and $1$. The associated graded complex $\gr_G B_\bullet$ is closely related to $\widetilde{C}^{(2)}_\bullet$ introduced above, and its homology is computed by a spectral sequence whose first page is the Koszul homology $\HK_\bullet(\gr V)$. Summarizing:
\end{nolabel}

\begin{prop}[{\cite{eh2018}}]\label{prop:eh2018.summary}
Let $V$ be a conformal vertex algebra. There exists a surjection
\[
\rho : H_1^{\text{ch}}(V, q=0) \rightarrow \Hoch_1(\zhu(V)),
\]
and a surjection
\[
\HK_1(\gr V) \rightarrow \ker(\rho).
\]
In particular if $V$ is classically free then $H_1^{\text{ch}}(V, q=0) \cong \Hoch_1(\zhu(V))$.
\end{prop}

\begin{nolabel}
Dual to the surjection $\rho$ we have an injection
\[
\iota : \Hoch_1(\zhu(V))^* \rightarrow H_1^{\text{ch}}(V)^*.
\]
Since in degree $1$ the subcomplex $B_\bullet$ consists of classes of the form $a \otimes a_1 \otimes \wp(t_0)$, the morphism $\iota$ is simply evaluation on homology classes associated with constant function, that is
\begin{align*}
\iota(\varphi)[a \otimes a_1 \otimes 1] &= \varphi(a \otimes a_1), \\
\iota(\varphi)[a \otimes a_1 \otimes \wp(t_0)] &= 0.
\end{align*}
Suppose now that $V$ is classically free, i.e., that $\HK_1(\gr V)$, so that
$\iota$ becomes an isomorphism. It would be very interesting in general to
determine the extent to which leading coefficients of flat sections $S \in
H_1^{\text{ch}}(V)^{*\nabla}$ are controlled by the Hochschild homology $\Hoch_1(\zhu(V))$. We hope to study this question in future work.
\end{nolabel}

\begin{thm}\label{thm:chiral.H1.vanishing}
Let $V$ be a strongly finitely generated conformal vertex algebra, $A = \gr V$ its associated graded and $R_V$ its $C_2$-algebra. Suppose that $V$ is classically free, that $\dim \HP_1(R_V) < \infty$, and that $\Hoch_1(\zhu(V)) = 0$. Then the chiral homology $H_1^{\text{ch}}(V)$ vanishes.
\end{thm}

\begin{proof}
Since $V$ is classically free we have $\dim \HK_{1}(A) = 0$, so the conditions of Proposition \ref{prop:H1.fin.gen} are satisfied, $H_1^{\text{ch}}(V)$ is a finitely generated $\mring$-module. Moreover by Proposition \ref{prop:eh2018.summary} we have an isomorphism
\[
H_1^{\text{ch}}(V, q=0) \similarrightarrow \Hoch_1(\zhu(V)).
\]
Suppose $H_1^{\text{ch}}(V) \neq 0$, and let $S \in H_1^{\text{ch}}(V)^{*\nabla}$ be a nonzero flat local section. Then the Frobenius expansion of $S$ possesses some nonzero leading coefficient $U$ which descends to a linear function on $\Hoch_1(\zhu(V))$. Since $\Hoch_1(\zhu(V)) = 0$ we have an immediate contradiction, and so $H_1^{\text{ch}}(V) = 0$ as required.
\end{proof}
\begin{nolabel}
We remark on the hypotheses of Theorem \ref{thm:chiral.H1.vanishing}. The
hypothesis $\dim \HP_1(R_V) < \infty$ is relatively mild, and is analogous to
the condition $\dim \HP_0(R_V) < \infty$ which is weaker than
$C_2$-cofiniteness. The condition that $\Hoch_1(\zhu(V)) = 0$ is satisfied
whenever $\zhu(V)$ is semisimple, for example, and in particular whenever the
vertex algebra $V$ is rational. The condition of classical freeness is satisfied
for many vertex algebras, and the question of validity or failure seems to be a subtle and interesting question. We discuss this further in Section \ref{sec:examples} below on examples.
\end{nolabel}
\begin{thm} \label{thm:convergence1}
Let $V$ be a vertex algebra of central charge $c$, satisfying the hypotheses of Proposition \ref{prop:H1.fin.gen}. Let $M$ be a $V$-module and $E$ a self-extension of $M$, and let
\[
a_0 \otimes a_1 \otimes F(z_0, z_1)
\]
be the Fourier series expansion of a $1$-cycle $a_0 \otimes a_1 \otimes f(t_0, t_1)$ in the chiral homology complex $C^{n=1}_\bullet$ \eqref{eq:homology-def}. Putting $q = e^{2\pi i \tau}$ where $\tau$ lies in the upper half complex plane, the trace function
\begin{align*}
F_1^1(a_0 \otimes a_1 \otimes F(z_0, z_1, q))q^{-\frac{c}{24}} = \res_{z_0} z_1 F(z_0,
z_1, q) \tr_M \sigma(a_0, z_0) X(a_1, z_1) q^{L_0 - \frac{c}{24}},
\end{align*}
converges to a holomorphic function of $\tau$, and is independent of $t_1$. 
\end{thm}
\begin{proof}
By Corollary \ref{cor:sum-of-omega} the series $F_1^1(a_0 \otimes a_1 \otimes
F(z_0, z_1, q))$, which \emph{a priori} depends on $z_1$ and $q$, is independent
of $z_1$. In Theorem \ref{thm:cor-der-dq} it is shown that $F_1^1$ satisfies the
differential equation \eqref{eq:dq-der-1}. This differential equation coincides
formally with \eqref{eq:flatness-0} under the identification $z_1 = e^{2\pi i
t_1}$ (the extra term involving $\frac{c}{24}$ in \eqref{eq:dq-der-1} is
cancelled by the factor $q^{-\frac{c}{24}})$, and in particular $q=0$ is a regular singular point.

The fact that $F_1^1$ annihiliates coboundaries is proved in Proposition \ref{prop:new-deriv-generic}. Since we have restricted $F_1^1$ to $1$-cycles it descends, as in Proposition \ref{prop:H1.fin.gen}, to a morphism from a finitely generated $\qmring$-module $\qmring\otimes_{\mring} H_1^{\ch}(V)$. Choosing generators produces a finite rank system of ODEs with regular singular point $q=0$, thus as above the series defining $F_1^1(a_0 \otimes a_1 \otimes F(z_0, z_1, q))$ converges in simply connected subdomains of $0 < |q| < 1$.
\end{proof}
\section{Self-extensions and the first chiral homology} \label{sec:formal-to-conformal}
In section \ref{sec:higher.traces} we associated to a vertex algebra $V$, its
module $M$ and its self--extension $E$, a family of linear functionals $F_1^n$.
In section \ref{sec:series.exp} we proved that under certain finiteness
conditions on $V$ these linear functionals converge in the case $n=1$ to
holomorphic functions of $\tau$. In this section we prove the convergence of
these functionals in the general $n$ case. We show that the limits can be
analytically extended to meromorphic elliptic functions and that these
extensions, after a suitable normalization factor and exponential change of
coordinates, constitute a degree $1$ conformal block on the torus. 
\begin{nolabel} Let $V$ be a vertex algebra, $M$ its module and $\Psi$ an
$\End(M)$--valued derivation of $V$.  Let $a_0,\dots,a_n$ in $V$ and $f \in
\bJ^{n+1}_*$ and consider $\alpha = a_0 \otimes \dots \otimes a_n \otimes f \in
C^n_1$. Suppose $\alpha$ is a cycle, that is $d_1 \alpha = 0$.  Let $F(z_0,\dots,z_n,q)$
be the Fourier expansion of $f$ in the domain $|qz_0| < |z_n| < \dots < |z_0|$
as in \ref{no:fourier-n-variable}. Consider the linear functional $F_1^n
(a_0,\dots,a_n, F(z_0,\dots,z_n,q))$ defined by \eqref{eq:8.1}. 
\begin{thm*}
Suppose 
that the finiteness conditions of Proposition \ref{prop:H1.fin.gen} are satisfied,
then $F_1^n$ is convergent; Moreover, let $\tilde{F}_1^n(a_0,\dots,a_n, F(z_0,\dots,z_n,q))$ be its meromorphic continuation and put
\begin{equation}
S^n_1 \Bigl( a_0,\dots,a_n, f;t_1,\dots,t_n,\tau \Bigr) =  \tilde{F}_1^n \Bigl(
a_0,\dots,a_n, F \left( z_0,\dots,e^{2 \pi i t_n},e^{2 \pi i \tau} \right) \Bigr) q^{- \frac{c}{24}}.
\label{eq:10.1}
\end{equation}
Then 
the family $S_1 = \{ S_1^n \}_{n \geq 1}$ satisfies the genus--one
property for the vertex algebra $(V, Y[\cdot,t], \tilde{\omega}, \vac)$. 
\end{thm*}
\label{thm:convergence3}
\end{nolabel}

The proof of this Theorem will be given in the next sub-sections. Notice that b) in Definition \ref{defn:conf-1-blocks} is
automatic. We proceed by induction in $n$. In \ref{no:induction-proof} we consider the
$n=1$ case. In \ref{no:converge-n-func} we prove that $F^{n}_1$ converges for
all $n$. In \ref{no:symmetry-traces} we prove that $S^{n}_1$ are well defined
on cycles in $C^{n}_1$. In \ref{no:traces-ellitpic-11} we prove that $S^n_1$ are
elliptic functions and in \ref{no:in-coh-class} that $S^{n}_1$ are flat with
respect to the connection on the moduli space of marked elliptic curves. 
\begin{nolabel}We want to prove that $S^n_1$ is a well defined flat section of
the dual to chiral homology, that is $S^n_1 \in H^{ch}_1(V^{\otimes
n})^{*\nabla}$. We will proceed by induction on $n$. Let us prove first the
$n=1$ case. Notice that c) in definition \ref{defn:conf-1-blocks} is empty in
this case.  Let $\alpha \in C^1_1$ be
such that $d_1 \alpha = 0$.  The fact that
$S^1_1( \alpha; t_1,\tau) \in \cF_1$,
that is, it is independent of $t_1$,
follows immediately from Theorem \ref{thm:convergence1}. The fact that it
vanishes if $\alpha = d_2 \beta$ for $\beta \in C^1_2$ follows from Proposition
\ref{prop:new-deriv-generic}, hence $S_1^1$ defines a section of the dual of
chiral homology $H^{\text{ch}}_1(V)^*$. The fact that it is a flat section, that
is, that it satisfies \eqref{eq:flatness-0}, follows at once from Theorem
\ref{thm:cor-der-dq} after the coordinate change $q = e^{2 \pi i \tau}$ and $z_j
= e^{2 \pi i t_j}$. The extra $\tfrac{c}{24}$ in the left hand side of \eqref{eq:dq-der-1}
is cancelled with the factor $q^{-\frac{c}{24}}$ in \eqref{eq:10.1}. 
\label{no:induction-proof}
\end{nolabel}
\begin{nolabel} \label{no:symmetry-traces} Let us now assume that $S^n_1$ is a
well defined flat section in $H^\text{ch}_1(V^{\otimes n})^{*\nabla}$ and that
it satisfies c) in \ref{defn:conf-1-blocks}. We will
prove that so is $S^{n+1}_1$ and that it satisfies the insertion formula
\eqref{eq:insertion-conf-2}. We start by proving that the linear functionals
converge and satisfy the insertion formula \eqref{eq:insertion-conf-2}.  

Assume the notation of Proposition \ref{prop:8.14}, so that we have
$a_0,\dots,a_n, a \in V$, $1 \leq j \leq n+1$ and $f =
f(t_0,\dots,t_{j-1},s,t_{j},\dots,t_n, \tau) \in \cF_{n+2}$. Suppose that
\[ \alpha = a_0 \otimes \dots a_{j-1} \otimes a \otimes a_{j} \otimes \dots
\otimes a_n \otimes f \in C^{n+1}_1, \]
is a cycle, $d_1 \alpha = 0$. 
In the same way as in \ref{no:ins-well-defined} 
we notice that the right hand side of \eqref{eq:8.15-prop1} consists of $F^n_1$ applied to
elements in $C^n_1$ that are themselves cycles, therefore they are convergent by
induction.
\label{no:converge-n-func}
\end{nolabel}
\begin{nolabel} 
The meromorphic continuation of the limit in the right hand side of \eqref{eq:8.15-prop1} is independent of $j$.
Therefore it follows from Proposition \ref{prop:8.15} that 
for $\sigma \in S_n$ we have  
\begin{equation}
 \tilde{F}_1^n\left( a_0, a_{\sigma(1)},\dots, a_{\sigma(n)}, F\left( z_0,
z_{\sigma(1)},\dots,z_{\sigma(n)},q
\right)\right) = \tilde{F}_1^n \left( a_0,\dots,a_n, F\left( z_0,\dots,z_n,q
\right) \right). 
\label{eq:symmetry-traces}
\end{equation}
Thus $S^{n+1}_1$ satisfies \eqref{eq:s-symmetric}. The
fact that $S^{n+1}_1$ vanishes on elements like \eqref{eq:c1-quotb2} follows
from \eqref{eq:8.trading-2}. 
\end{nolabel}
\begin{nolabel}
For $1 \leq j \leq n$ we have from \eqref{eq:symmetry-traces} that
\begin{multline*} 
S^{n+1}_1\Bigl(a_0,\dots, a_{j-1}, a, a_{j+1},\dots,a_n,
f(t_0,\dots,t_{j-1},t,t_{j+1},\dots,t_n,\tau);
t_1,\dots,t_{j-1},t,t_{j+1},\dots,t_n,\tau\Bigr) = \\ 
S^{n+1}_1\Bigl(a_0,a,a_1,\dots, a_n,
f(t_0,t,t_1,\dots,t_n,\tau);
t,t_1,\dots,t_n , \tau\Bigr)
\end{multline*}
We proved in \ref{no:ins-well-defined} that the right hand side is an elliptic function in
$\cF_{n}$, it follows that
the left hand side is elliptic as well. 
\label{no:traces-ellitpic-11}
\end{nolabel}
\begin{nolabel}The fact that $S^{n+1}_1$ vanishes on boundaries Follows from
Proposition \ref{prop:8-trace-d2}. The fact that $S^{n+1}_1$ satisfies the
differential equations \eqref{eq:connection-tn} follows from \eqref{eq:lem821}.
And finally the fact that $S^{n+1}_1$ satisfies \eqref{eq:flatness-0} follows
from Theorem \ref{thm:last-diff-eq}, thus proving that $S^{n+1}_1 \in
H^{\text{ch}}_1(V^{\otimes n})^{*\nabla}$ and we have established the induction
step. 
\label{no:in-coh-class}
\end{nolabel}

\section{Examples} \label{sec:examples}

In this section we work out the constructions and results introduced in this
paper in the context of several specific examples. Firstly we derive
consequences of the finiteness and vanishing results on chiral homology of
Section \ref{sec:series.exp} above for rational vertex algebras. Secondly we work out a simple example of higher trace functions for the Heisenberg vertex algebra.

\begin{nolabel}
The condition that $V$ be classically free is satisfied by universal enveloping
vertex algebras such as the Heisenberg vertex algebra, the universal affine
vertex algebra $V^k(\g)$ at arbitrary level, their quantum Hamiltonian
reductions, that is the universal affine $W$ algebras $W^k(\g,f)$, and the universal Virasoro vertex algebra $\vir^c$ at arbitrary central charge. The question of classical freeness of rational vertex algebras is quite subtle: the boundary minimal models $\vir_{2, 2k+1}$ are rational and classically free, but other minimal models such as the Ising model $\vir_{3, 4}$ are not. It is also known that the rational affine vertex algebras $V_1(\mathfrak{g})$ are classically free, as are $V_k(\mathfrak{sl}_2)$ for $k \in \Z_+$. This latter statement is believed to generalize from $\mathfrak{sl}_2$ to other semisimple Lie algebras $\g$. We now prove the following corollary of Theorem \ref{thm:chiral.H1.vanishing}.
\end{nolabel}

\begin{cor} \label{cor:examples}
The chiral homology $H_1^{\text{ch}}(E_\tau, V)$ of the smooth elliptic curve $E_\tau = \mathbb C / \Z + \Z \tau$ with coefficients in the vertex algebra $V$ vanishes in the following cases:
\begin{itemize}
\item[(a)] $V = \vir_{2, 2s+1}$ is a boundary Virasoro minimal model,

\item[(b)] $V = V_k(\mathfrak{sl}_2)$ is the simple affine vertex algebra at level $k \in \Z_+$,

\item[(c)] $V = V_1(\mathfrak{g})$ is the simple affine vertex algebra at level $1$, for $\mathfrak{g}$ a simple Lie algebra.
\end{itemize}
\end{cor}

\begin{proof}
For $V$ any of the vertex algebras listed in the statement, $\zhu(V)$ is semisimple since $V$ is rational, and hence $\Hoch_1(\zhu(V)) = 0$. We also have $\dim \HP_1(R_V) < \infty$ as these vertex algebras are all $C_2$-cofinite.

For all three parts it suffices, therefore, to verify that $V$ is classically free. Parts (a) and (b) were established in \cite{eh2018} (in the case of (b) relying heavily on results of \cite{meurman-primc}). Classical freeness of $V_1(\mathfrak{g})$ is proved in \cite{efeigin}. The assertions now follow from Theorem \ref{thm:chiral.H1.vanishing}.
\end{proof}
\begin{rem} For $\g$ of $ADE$ type, item c) above is a particular case of a
stronger result of Gaitsgory for lattice vertex algebras
\cite{gaitsgory-lattice} (see also \cite[\S 4.9]{beilinsondrinfeld}).
\label{rem:lattice}
\end{rem}

In a similar manner we prove the following.

\begin{cor}
The chiral homology $H_1^{\text{ch}}(E_\tau, V)$ of the smooth elliptic curve $E_\tau = \mathbb C / \Z + \Z \tau$ with coefficients in the Ising model $\vir_{3, 4}$ is finite dimensional.
\end{cor}

\begin{proof}
It was established in \cite{aeh2020} that, while $V = \vir_{3, 4}$ is not classically free, the kernel of $J R_V \rightarrow A = \gr_F V$ is generated as a differential ideal by a single element. Hence $\dim \HK_1(A) = 1$. Finite dimensionality now follows from Theorem \ref{thm:convergence2}.
\end{proof}
\begin{rem} The above proof shows that $\dim H^{ch}_1(E_\tau, \vir_{3,4}) \leq
1$. In \cite{aeh2020} we found an explicit generator for $\HK_1(A)$ in this
case, we do not know if the image of this element in $H^{ch}_1(E_\tau,
\vir_{3,4})$ vanishes or not. 
\label{rem:ising}
\end{rem}

\begin{nolabel}
We describe an example of higher trace function, associated with the Heisenberg vertex algebra. Let $\h = \mathbb C x$ be a one dimensional vector space with a non degenerate bilinear form $(x, x) = 1$, and let $\widehat{\h} = \h[t, t^{-1}] \oplus \mathbb C K$ be its affinization, i.e., the Lie algebra defined by
\[
[xt^m, xt^n] = m \delta_{m, -n} K, \qquad [K, \widehat{\h}] = 0.
\]
We write $x_m$ for $xt^m$. The vacuum module
\[
H = U(\widehat{\h}) \otimes_{U(\h[t] \oplus \mathbb C K)} \mathbb C v,
\]
in which $K v = v$ and $\h[t] v = 0$, carries the natural structure of a conformal vertex algebra with vacuum vector $\vac = 1 \otimes v$. This is the rank one Heisenberg vertex algebra. We identify $a_{(-1)}\vac = a_{-1}\vac$ with $a \in \h$, the vertex algebra structure is determined by the quantum fields $Y(a, z) = \sum_{n \in \Z} a_{n} z^{-n-1}$. The conformal vector is given by $\omega = x_{-1}x_{-1}\vac$ and has central charge $c = 1$.
\end{nolabel}

\begin{nolabel}
Let $D$ be a finite dimensional $\mathbb C[x]$-module, we consider the induced $\widehat{\h}$-module
\[
M_D = U(\widehat{\h}) \otimes_{U(\h[t] \oplus \mathbb C K)} D,
\]
in which the action of $x_m$ on $D$ for $m > 0$ is trivial, $K$ acts by $1$ and $x_0$ acts by $x$. Then $M_D$ is naturally a $H$-module. The one dimensional $\mathbb C[x]$-module $\mathbb C$ corresponds to $H$ itself.

Let us consider $D = \mathbb C[x] / (x^2)$ which is a non trivial extension
\[
0 \rightarrow \mathbb C \rightarrow D \rightarrow \mathbb C \rightarrow 0.
\]
The $H$-module $M_D$ is a self-extension of $H$ and possesses a PBW basis
\begin{align}\label{eq:H.basis}
x_{-N}^{\lambda_N} \cdots x_{-1}^{\lambda_1} x_0^{\lambda_0} \vac
\end{align}
in which $\lambda_i \in \Z_{\geq 0}$ for all $i$ and $\lambda_0 \leq 1$. The monomials for which $\lambda_0 = 1$ form a basis of the embedded copy $H \subset M_D$. 
\end{nolabel}

\begin{nolabel}
The Zhu algebra $\zhu(H)$ of $H$ is isomorphic to $\mathbb C[x]$. Indeed $x$ here is the image of $x_{-1}\vac$ in the quotient that defines $\zhu(H)$.
\end{nolabel}

\begin{nolabel}
We now consider the higher trace function
\[
F_1(a \otimes b \otimes f(t_0, t_1)) = \res_{z_0} z_1 \tr_{H} \sigma(a, z_0) X(b, z_1) F(z_0, z_1) q^{L_0}
\]
associated with the self-extension $M_D$ of $H$ described above, and a cycle $a \otimes
b \otimes f(t_0, t_1) \in C_1^{n=1}$. In particular if we put $f(t_0, t_1) = 1$ and $b = \vac$ then we have
\[
d_1(a \otimes b \otimes 1) = a_{(0)}\vac = 0.
\]
We may thus consider the higher trace functions 
\[
F_1(a \otimes \vac \otimes 1) = \tr_{H} \sigma(a)_0 \vac_0 q^{L_0} = \tr_{H} \sigma(a)_0 \vac_0 q^{L_0}.
\]
Since $x_0$ is central in $\widehat{\h}$ it follows that $x_0$ acts relative to the basis of monomials \eqref{eq:H.basis} of $M_D$ by increasing $\lambda_0=0$ to $\lambda_0=1$ and annihilating those monomials for which $\lambda_0=1$. Therefore $\sigma(x)_0$ is simply the identity on $H$. Hence
\begin{align*}
F_1(x \otimes \vac \otimes 1) = \tr_{H} q^{L_0} = \prod_{n=1}^\infty \frac{1}{1-q^n}.
\end{align*}
\end{nolabel}
\section{Conclusions} \label{sec:conclusion}
We constructed a complex that computes the chiral homology in degrees
$0$ and $1$ of an
elliptic curve $E_\tau$ with coefficients in a vertex algebra $V$ supported at
the marked point $0 \in E_\tau$. We generalized these complexes to
$n$-insertions of $V$ at the points $t_1,\dots,t_n \in E_\tau$. As the
collection $\{t_i  \}$ moves in $\Sym^n E_\tau$ and $\tau$ moves in $\mathbb{H}$
the homologies of these complexes acquire a flat connection. We give explicit
description of this connection in terms of an insertion of the conformal vector
and a Weierstrass $\zeta$ function, and showed that this connection has a
regular singular point as $\tau \rightarrow i \infty$. We found finiteness conditions on
the vertex algebra $V$ to guarantee that the first chiral homology of $E_\tau$
with coefficients in $V$ is finite dimensional. We defined the degree $1$ analog
of the space of conformal blocks, or $n$-point functions of a vertex algebra, as a compatible collection
of flat sections of the dual of the first chiral homology group with
coefficients in $n$ insertions of $V$. We prove the modular invariance of the
space of these higher analogs of $n$-point functions and show that
self-extensions of modules produce such compatible systems of flat sections. 

As a technical tool to achieve the above, we proved a version of Borcherds
formula, relating two different vertex algebra structures on $V$, as well as
elliptic functions and their Fourier expansions. We believe this theorem will be
of independent interest. 

There are a number of results which we have not achieved. In degree zero, Zhu
proved that the space of irreducible modules of $V$ provides a basis for the
dual of zeroth chiral homology under certain finiteness conditions of $V$. We do
not have an analogous statement in degree $1$. In the classically free case, we have an isomorphism 
$H_1^{\text{ch}}(V, q=0) \similarrightarrow \Hoch_1(\zhu(V))$. It would be interesting to find a way of inducing linear functionals on $\Hoch_1(\zhu(V))$ to first conformal blocks on the torus. The analogous statement in degree $0$ corresponds to inducing irreducible modules of $\zhu(V)$ to irreducible modules of $V$. 

Our approach follows Zhu's technique of restricting to the point at infinity $q=0$ in the moduli space of elliptic curves. More generally Miyamoto \cite{miyamoto-c2} has studied the zeroth chiral homology of non-rational,
$C_2$-cofinite vertex algebras, by using \emph{higher Zhu algebras}. These
algebras relate to non-reduced neighborhoods $\Spec
\mathbb{C}[q]/(q^n)$  of the point at infinity, as the Zhu algebra does in the
case $n=1$. It would be interesting to see if these algebras play a similar role
for chiral homology in degree $1$. 

There are two lines of research that are natural continuations for this work:
higher degree chiral homology and higher genera. The first seems to be an
immediate extension of the complexes constructed in this article. What remains
to be done is to determine the higher analogs of the linear functionals
associated to self-extensions of modules. Well known examples of $C_2$-cofinite
but non-rational vertex algebras, like the triplet vertex algebra, admit higher
self-extensions of modules (but not classes in $\text{Ext}^1$). It would be
interesting if these classes generate higher chiral homology classes on elliptic
curves. 

To generalize this work to higher genera a natural approach is to use the
techniques of \cite{damiolini2020conformal}. As mentioned in the introduction,
it would be quite interesting to
know if the finiteness conditions found in genus $1$ in this article suffice to
guarantee finite dimensionality at arbitrary genus. 
\def\cprime{$'$}

\end{document}